\documentclass[10pt,a4paper]{amsart} 

\usepackage[utf8]{inputenc}
\usepackage[T1]{fontenc}
\usepackage{amsmath, amssymb}
\usepackage{amsthm}
\usepackage{mathtools}   

\usepackage{bbm} 
\usepackage{csquotes} 

\usepackage{scalerel}   

\usepackage{xcolor}            

\usepackage{hyperref}
\hypersetup{
    unicode=true,          
    pdftitle={},    
    pdfauthor={},     
    pdfsubject={},   
    pdfkeywords={}{}, 
    pdfborder={0 0 0}
}

\usepackage{enumitem} 
\setenumerate[0]{label=(\roman*)}   

\numberwithin{equation}{section}

\newtheorem{thm}{Theorem}[section]
\newtheorem{prop}[thm]{Proposition}
\newtheorem{lem}[thm]{Lemma}
\newtheorem{coro}[thm]{Corollary}

\theoremstyle{definition}
\newtheorem*{definition}{Definition}
\newtheorem*{example}{Example}
\newtheorem*{remark}{Remark}

\mathcode`l="8000
\begingroup
\makeatletter
\lccode`\~=`\l
\DeclareMathSymbol{\lsb@l}{\mathalpha}{letters}{`l}
\lowercase{\gdef~{\ifnum\the\mathgroup=\m@ne \ell \else \lsb@l \fi}}
\endgroup

\DeclareMathOperator{\GL}{GL}
\DeclareMathOperator{\SL}{SL}
\DeclareMathOperator{\SO}{SO}
\DeclareMathOperator{\End}{End}

\DeclareMathOperator{\Mat}{\mathcal{M}}
\DeclareMathOperator{\Supp}{supp}
\DeclareMathOperator{\Stab}{Stab}

\newcommand{\bracket}[1]{\langle#1\rangle}
\newcommand{\dd}{\,\mathrm{d}}

\newcommand{\diag}{\mathrm{diag}}

\newcommand{\op}{\mathrm{op}}

\newcommand{\pp}{\boxplus}
\newcommand{\ff}{*}
\newcommand{\mm}{\boxminus}

\newcommand{\R}{\mathbb{R}}
\newcommand{\C}{\mathbb{C}}
\newcommand{\Z}{\mathbb{Z}}

\newcommand{\N}{\mathbb{N}}
\newcommand{\Q}{\mathbb{Q}}

\newcommand{\Tbb}{\mathbb{T}}

\newcommand{\Ebb}{\mathbb{E}}
\newcommand{\Pbb}{\mathbb{P}}

\newcommand{\Ecal}{\mathcal{E}}

\newcommand{\Lcal}{\mathcal{L}}
\newcommand{\Ncal}{\mathcal{N}}

\newcommand{\Pcal}{\mathcal{P}}

\newcommand{\dtil}{\tilde{d}}

\newcommand{\tE}{\bar{E}}

\newcommand{\eps}{\varepsilon}
\newcommand{\hhat}{\widehat}

\newcommand{\tmu}{\bar{\mu}}

\DeclareMathOperator{\ess}{ess}

\newcommand{\abs}[1]{\lvert#1\rvert}    
\newcommand{\abse}[1]{\left\lvert#1\right\rvert} 
\newcommand{\absbig}[1]{\bigl\lvert#1\bigr\rvert} 
\newcommand{\absBig}[1]{\Bigl\lvert#1\Bigr\rvert} 

\newcommand{\norm}[1]{\lVert#1\rVert}   
   
\newcommand{\normbig}[1]{\bigl\lVert#1\bigr\rVert}

\newcommand{\qnorm}[1]{\lvert#1\rvert^{\scaleobj{0.7}{\sim}}}

\newcommand{\Probbig}[1]{\Pbb\bigl[\,#1\,\bigr]}
\newcommand{\ProbBig}[1]{\Pbb\Bigl[\,#1\,\Bigr]}
\newcommand{\ProbcondBig}[2]{\Pbb\Bigl[\, #1 \,\Big\vert\, #2\,\Bigr]}

\newcommand{\set}[2]{\{\, #1 \,\vert\, #2 \,\}}
\newcommand{\setbig}[2]{\bigl\{\, #1 \,\big\vert\, #2 \,\bigr\}}
\newcommand{\setBig}[2]{\Bigl\{\, #1 \,\Big\vert\, #2 \,\Bigr\}}

\DeclareMathOperator{\Nbdtil}{\tilde{N}bd}
\DeclareMathOperator{\Ball}{B}

\renewcommand{\phi}{\varphi}
\DeclareMathOperator{\Span}{Span}

\DeclareMathOperator{\1}{\mathbbm{1}}

\DeclareMathOperator{\mes}{m}

\newcommand{\mybullet}{\boldsymbol{\,\cdot\,}}

\newcommand{\mymod}{\;\mathrm{mod}\;}

\DeclareMathOperator{\NC}{NC} 
\DeclareMathOperator{\ANC}{ANC}

\begin{document}

\title{Semisimple random walks on the torus}

\author{Weikun He}
\thanks{}
\address{Institute of Mathematics, Academy of Mathematics and System Science, CAS, Zhongguancun East Road 55, Beijing 100190, P.R. China}
\email{heweikun@amss.ac.cn}

\author{Nicolas de Saxcé}
\thanks{}
\address{CNRS -- Université Sorbonne Paris Nord, LAGA, 93430 Villetaneuse, France.}
\email{desaxce@math.univ-paris13.fr}



\date{}

\begin{abstract}
We study linear random walks on the torus and show a quantitative equidistribution statement, under the assumption that the Zariski closure of the acting group is semisimple.
\end{abstract}

\maketitle

\section{Introduction}

Let $d\geq 2$ be an integer and $\Tbb^d=\R^d/\Z^d$ the torus of dimension $d$.
We study a random walk $(x_n)_{n\geq 0}$ on $\Tbb^d$ given by
\[
\forall n\geq 0,\quad x_n = g_n\dots g_1x_0
\]
where $(g_n)_{n\geq 1}$ is a sequence of independent identically distributed random variables with law $\mu$ on $\GL_d(\Z)$.
Let $\Gamma$ denote the group generated by the support of $\mu$, and $G$ be the Zariski closure of $\Gamma$ in $\GL_d(\R)$.
In \cite{BFLM}, Bourgain, Furman, Lindenstrauss and Mozes showed that if $G$ acts strongly irreducibly and proximally on $\R^d$, then the random walk $(x_n)_{n\geq 0}$ equidistributes in law to the Haar probability measure $\mes_{\Tbb^d}$ as soon as $x_0$ is irrational, i.e.
\[
\forall x_0\not\in\Q^d/\Z^d,
\quad \mu^{*n}*\delta_{x_0} \rightharpoonup^*_{n\to+\infty} \mes_{\Tbb^d}.
\]
Moreover, this result is quantitative : an explicit rate of convergence is obtained in terms of the distance from $x_0$ to rational points of small denominator.
Following their strategy, we showed in \cite{HS2019} that their theorem is still valid without the proximality assumption, as long as the action of $G$ on $\R^d$ is strongly irreducible.
On the other hand, the theory developed by Benoist and Quint in their series of articles \cite{bq1,bq2,bq3,bqfv} made it clear that when studying random walks on homogeneous spaces, it is most natural to only assume that the acting algebraic group $G$ is semisimple.
Indeed, under this assumption \cite[Theorem~1.1]{bq2} gives a full classification of stationary measures, which in turn implies the very general equidistribution results of \cite{bq3}.
It is therefore desirable to obtain quantitative convergence results similar to those of \cite{BFLM} or \cite{HS2019} in this more general setting, and this is the goal of the present article.

\bigskip

Of course, without the irreducibility assumption, there can exist some proper closed $\Gamma$-invariant subsets of $\Tbb^d$, and the random walk may not equidistribute to the Haar measure, even if the starting point $x_0$ is irrational.
So in order to state our main result, we need to set up some notation.
Let $G^\circ$ denote the identity component of $G$ for the Zariski topology.
The subalgebra of $\Mat_d(\R)$ generated by $G^\circ$ is denoted by $E$.
Since $G$ is semisimple, we may write
\(
\R^d = V_0\oplus V_1 \oplus\dots\oplus V_r
\)
where for $i = 0,\dotsc,r$, $V_i$ is a maximal sum of simple $G$-modules having the same top Lyapunov exponent for the action of $\mu$.
Reordering the subspaces $V_i$, we may assume in addition that 
\[
\lambda_1(\mu,V_1)
>\dots > \lambda_1(\mu,V_r)>\lambda_1(\mu,V_0) = 0.
\]
The space $V_0$ will play a special role in our analysis of the random walk behavior.
By a result of Furstenberg, the image of $G$ in $\GL(V_0)$ is compact, and for that reason, we shall say that $V_0$ is the \emph{sum of all compact factors} of $G$ in $\R^d$.
We define a quasi-norm on $\R^d$ by
\[
\abs{v} = \max_{0\leq i\leq r} \norm{v_i}^{\frac{1}{\lambda_1(\mu,V_i)}}
\]
where $v = v_0+\dotsc+v_r$ is the decomposition of $v$ according to the direct sum $\R^d = \oplus_{i=1}^r V_i$.
By convention, we set $\frac{1}{0}=+\infty$ and
\begin{equation*}
\norm{v_0}^{+\infty} = \begin{cases} 0 & \text{if } \norm{v_0} \leq 1,\\ +\infty &\text{otherwise.}\end{cases}
\end{equation*}
This quasi-norm induces a quasi-distance on $\R^d$ given by $\dtil(x,y)=\abs{x-y}$, which projects to a quasi-distance on $\Tbb^d$, still denoted by $\dtil$.
A finite measure $\mu$ on $\GL_d(\Z)$ is said to have a \emph{finite exponential moment} if there exists $\tau>0$ such that
\[
\int\norm{g}^\tau\dd\mu(g) < +\infty,
\]
where $\norm{\cdot}$ denotes a norm on the space $M_d(\R)$ of $d\times d$ matrices; this definition does not depend on the choice of the norm.
Our goal is the following theorem.

\begin{thm}[Equidistribution of semisimple linear random walks on $\Tbb^d$]
\label{thm:finali}
Let $\mu$ be a probability measure on $\GL_d(\Z)$ having a finite exponential moment.
Denote by $G \subset \GL_d(\R)$ the algebraic group generated by $\mu$, by $G^\circ$ its identity component, and let $E$ be the subalgebra of $\Mat_d(\R)$ generated by $G^\circ$. 
As above, we write $V_0$ for the sum of all compact factors of $G$ in $\R^d$.
If the algebraic group $G$ is semisimple, then for every $\lambda\in(0,1)$, there exists $C=C(\mu,\lambda)\geq 0$ such that the following holds.

Given $x_0\in\Tbb^d$, 
assume that for some $t\in(0,\frac{1}{2})$, $a_0\in\Z^d \setminus\{0\}$, and $n \geq C\log\frac{\norm{a_0}}{t}$,
\[
\abs{\hhat{(\mu^{*n}*\delta_{x_0})}(a_0)} \geq t.
\]
Then, there exists $\gamma\in G/G^\circ$ such that, denoting
\(
W_0 = (a_0\gamma E)^\perp
\),
one has
\[
\dtil\bigl(x_0-\frac{p}{q}-v, W_0 \bigr) \leq e^{-n\lambda}
\]
for some $v\in V_0$, $p \in \Z^d$ and $q\in\Z \setminus\{0\}$ such that
\(
\max(\norm{v}, \abs{q}) \leq \left(\frac{\norm{a_0}}{t}\right)^C.
\)
\end{thm}
In the above, of course, $\frac{p}{q}$, $v$ and $W_0$ are identified with their projection to the torus $\Tbb^d$.
A slightly more precise version of Theorem~\ref{thm:finali} is stated below as Theorem~\ref{thm:final}.

\begin{remark}
If $G$ is connected, i.e. $G=G^\circ$, then $W_0=(a_0E)^\perp$ is entirely determined by $a_0$.
Existence of a large Fourier coefficient $\hhat{(\mu^{*n}*\delta_{x_0})}(a_0)$ implies that the starting point of the random walk is close to a rational translate of an invariant closed subset of the form $W_0+\Ball_{V_0}(0,R) \mod \Z^d$, where $\Ball_{V_0}(0,R)$ denotes the centered closed ball of radius $R$ in $V$ with respect to some $G$-invariant Euclidean norm on $V$ and $R$ is controlled in terms of $\norm{a_0}$ and $\abs{\hhat{(\mu^{*n}*\delta_{x_0})}(a_0)}$.
\end{remark}

\begin{example}[Reducible random walk]
Fix a probability measure $\mu_0$ on $\SL_2(\Z)$ such that $\Supp\mu_0$ generates $\SL_2(\Z)$, and let $\mu=\mu_0\otimes\mu_0$.
Using the block diagonal embedding $\SL_2\times\SL_2\hookrightarrow\SL_4$, we view $\mu$ as a probability measure on $\SL_4(\Z)$.
In that setting, $E=M_2(\R)\times M_2(\R)$ and $V_0=\{0\}$.\\
Assume as in the theorem that $\hhat{(\mu^{*n}*\delta_{x_0})}(a_0)$ is large.
Write $a_0=(a_1,a_2,a_3,a_4)$.
\begin{itemize}
\item If $(a_1,a_2)$ and $(a_3,a_4)$ are both nonzero,
then $a_0E=\R^4$ and therefore $W_0=\{0\}$.
Thus, the starting point $x_0$ must be close to a rational point with small denominator.
\item If $(a_1,a_2) = 0$ and $(a_3,a_4) \neq 0$, then $a_0 E = \{0\} \oplus \R^2$ so that $W_0 = \R^2 \oplus \{0\}$.
The theorem only allows us to conclude that $x_0$ is close to a rational translate of the invariant subtorus $\Tbb^2\times\{0\} = W_0 \mod \Z^4$.
In particular, we can conclude nothing about the first two coordinates of $x_0$.
And indeed, if one starts from a point $x_0$ whose third and fourth coordinates are zero, then the random walk is trapped in the proper invariant subset $\Tbb^2\times\{0\}$.
For every frequency of the form $a_0=(0,0,a_3,a_4)$ and for all $n$, one has $\hhat{(\mu^{*n}*\delta_{x_0})}(a_0)=1$.
\end{itemize}
\end{example}

\begin{example}[Compact factors and satellite measures]
Consider the quadratic form $Q(x,y,z)=x^2+y^2-\sqrt{2}z^2$ on $\R^3$, and $\SO_Q\subset\SL_3$ its special orthogonal group:
\[
\SO_Q = \{g\in\SL_3\ |\ ^t\!gJ_Qg = J_Q\},
\quad\mbox{where}\
J_Q = \diag(1,1,-\sqrt{2}).
\]
If $g$ is an element of the group $\Gamma=\SO_Q(\Z[\sqrt{2}])$ of elements of $\SO_Q$ with entries in the ring $\Z[\sqrt{2}]$, one can write $g=A+\sqrt{2}B$, with $A,B$ in $M_3(\Z)$.
The map
\[
g=A+\sqrt{2}B \mapsto \begin{pmatrix} A & 2B\\ B & A\end{pmatrix}
\]
embeds $\Gamma$ into $\SL_6(\Z)$.
Let $\mu$ be a probability measure on $\SL_6(\Z)$ whose support generates the group $\Gamma$.\\
\indent Since $\begin{pmatrix} A & 2B\\ B & A\end{pmatrix}$ is conjugate to $\diag(A+\sqrt{2}B,A-\sqrt{2}B)$, so that $\Gamma$ preserves a direct sum decomposition $\R^6=\R^3\oplus\R^3$.
The group $\Gamma$ acts on the second factor as a subgroup of $\SO_{\bar{Q}}$, where $\bar{Q}(x,y,z)=x^2+y^2+\sqrt{2}z^2$ is positive definite,
so there is a non-trivial compact factor $V_0\simeq\R^3$.
Note that $V_0$ embeds densely in $\Tbb^6$.\\
\indent On the other hand, one can check that in that setting $E$ is conjugate under $\begin{pmatrix}\sqrt{2}I_3 & -\sqrt{2}I_3\\ I_3 & I_3\end{pmatrix}$ to the block diagonal subalgebra $M_3(\R)\times M_3(\R)$.
If $a_0\in\Z^6\setminus\{0\}$, both projections of $a_0$ to the $\R^3$ factors are non-zero (note that the direct sum decomposition is not defined over $\Q$) and therefore one always has $a_0E=\R^6$, whence $W_0=\{0\}$.\\
\indent
Existence of a large Fourier coefficient $\hhat{(\mu^{*n}*\delta_{x_0})}(a_0)$ implies that up to a rational translation with small denominator, the starting point $x_0$ is close to the image in $\Tbb^6$ of a ball of controlled radius in $V_0$.
Note that if the starting point $x_0$ lies on the embedded leaf $V_0$, then the random walk equidistributes with respect to (the image in $\Tbb^6$ of) the uniform probability measure on the sphere containing $x_0$ for the quadratic form $x^2+y^2+\sqrt{2}z^2$ on $V_0$.
\end{example}

If the sequence $(\mu^{*n}*\delta_{x_0})$ does not converge to the Haar measure $\mes_{\Tbb^d}$ in the weak-$*$ topology, then, by Weyl's equidistribution criterion, there are $a_0\in \Z^d \setminus\{0\}$ and $t > 0$ such that $\abs{\hhat{(\mu^{*n}*\delta_{x_0})}(a_0)} \geq t$ for an unbounded sequence of $n \in \N$.
Letting $n$ go to infinity along this sequence, we deduce the following qualitative statement from the above theorem.
\begin{coro}[Qualitative statement]
\label{cr:qual}
Let $\mu$ be a probability measure on $\GL_d(\Z)$ having a finite exponential moment.
Denote by $G \subset \GL_d(\R)$ the algebraic group generated by $\mu$.
Assume that $G$ is semisimple. 
Then for any point $x_0\in\Tbb^d$, either 
\[
\mu^{*n}*\delta_{x_0} \rightharpoonup^* \mes_{\Tbb^d},
\]
or 
\[
x_0 \in \Q^d + V_0 + W_0 \mod \Z^d,
\]
where $V_0$ denotes the sum of all compact factors of $G$ in $\R^d$ and $W_0$ is a proper rational subspace of $\R^d$ invariant under the action of the identity component $G^\circ$ of $G$.
\end{coro}

As a consequence, we recover the classification of orbit closures due to Guivarc'h and Starkov~\cite{GuivarchStarkov} and Muchnik~\cite{Muchnik}.
\begin{coro}[Classification of orbit closures]
Let $\Gamma \subset \GL_d(\Z)$ be a subgroup whose Zariski closure $G$ is semisimple.
Let $x_0 \in \Tbb^d$. Then the orbit closure $\overline{\Gamma x}$ is either the whole $\Tbb^d$ or contained in a $\Gamma$-invariant closed subset of the form
\[\frac{1}{q}\Z^d + \Ball_{V_0}(0,R) + \bigcup_{\gamma \in G/G^\circ} \gamma W_0 \mod \Z^d\]
where $q$ is a nonzero integer, $\Ball_{V_0}(0,R)$ is a ball in $V_0$, the sum of all compact factors of $G$ in $\R^d$ and $W_0$ is a proper rational subspace invariant under the action of the identity component $G^\circ$ of $G$.
\end{coro}

The qualitative statement could also be reformulated more simply as follows.
\begin{coro}[Equidistribution]
Let $\mu$ be a probability measure on $\GL_d(\Z)$ having a finite exponential moment.
Denote by $\Gamma \subset \GL_d(\Z)$ the subgroup generated by $\mu$.
Assume that the Zariski closure of $\Gamma$ is semisimple. 
Then for any $x_0\in\Tbb^d$, either $\mu^{*n}*\delta_{x_0} \rightharpoonup^* \mes_{\Tbb^d}$ 
or $x_0$ is contained in a proper $\Gamma$-invariant closed subset.
\end{coro}

A particularly simple case of the above results is when the group $\Gamma$ acts strongly irreducibly on $\Q^d$, that is, when $\Gamma$ preserves no nontrivial finite union of proper subspaces of $\Q^d$.
Then, for any $a_0\in \Z^d \setminus \{0\}$ and any $\gamma \in G$, one must have $a_0\gamma E=(\R^d)^*$, so we obtain a simpler equidistribution statement.

\begin{coro}[Equidistribution of $\Q^d$-irreducible random walks]
\label{thm:qirr}
Assume that $G$ is semisimple and acts strongly irreducibly on $\Q^d$.
Then for every $\lambda\in(0,1)$, there exist $C=C(\mu,\lambda)\geq 0$ such that the following holds.

Given $x_0\in\Tbb^d$, 
assume that for some $t\in(0,\frac{1}{2})$, $a_0\in\Z^d$, and $n \geq C\log\frac{\norm{a_0}}{t}$,
\[
\abs{\hhat{(\mu^{*n}*\delta_{x_0})}(a_0)} \geq t.
\]
Then there exists $v \in V_0$, $p \in \Z^d$ and $q\in\Z \setminus \{0\}$ such that
\(
\max(\norm{v}, \abs{q}) \leq \left(\frac{\norm{a_0}}{t}\right)^C
\)
and 
\[
\dtil\bigl(x_0-\frac{p}{q}-v, 0\bigr) \leq e^{-n\lambda}.
\]
In particular, if $x_0$ does not lie on a rational translate of the $V_0$ leaf in $\Tbb^d$, then $\mu^{*n}*\delta_{x_0}$ converges to $\mes_{\Tbb^d}$.
\end{coro}

It was observed by Benoist and Quint \cite[Corollary~1.4]{bq2} that if $G$ is semisimple without compact factors and acts irreducibly on $\Q^d$, then $\mes_{\Tbb^d}$ is the only atom-free $\mu$-stationary probability measure on $\Tbb^d$.
By the results of \cite{bq3}, this implies that the Cesàro averages $\frac{1}{n}\sum_{k=0}^{n-1}\mu^{*k}*\delta_{x_0}$ converge to $\mes_{\Tbb^d}$.
The above corollary immediately shows that convergence also holds without the averaging process.
When $\mu$ is a symmetric probability measure on $\SL_d(\Z)$, a general result of Bénard~\cite[Theorem~1]{benard} implies this qualitative statement, but without the symmetry assumption the result seems to be new.

\begin{coro}
Assume that $G$ is semisimple without compact factors and acts strongly irreducibly on $\Q^d$.
Then, for every $x_0$ irrational in $\Tbb^d$, the sequence of measures $(\mu^{*n}*\delta_{x_0})_{n\geq 0}$ converges in law to $\mes_{\Tbb^d}$.
\end{coro}

One motivation to carry out the rather technical proof presented here is its application to the spectral gap property for subgroups of algebraic groups, modulo arbitrary integers.
Indeed, following a strategy of Bourgain and Varjú \cite{bv}, one can use Theorem~\ref{thm:finali} to answer a question of Salehi Golsefidy and Varjú \cite[Question~2]{SGV}.
A particular case of the problem was studied in \cite{HeSaxce_expansion}, and we hope to generalize those results in a forthcoming paper.

\subsection{Outline of the proof}
The paper is entirely devoted to the proof of Theorem~\ref{thm:finali}, for which we use the strategy introduced in \cite{BFLM}, and more precisely the variant used in \cite{HS2019} to avoid the proximality assumption.
Section~\ref{sec:sumprod} deals with discretized algebraic combinatorics in semisimple algebras: we prove some Fourier decay estimate for multiplicative convolutions of measures satisfying natural non-concentration conditions, Theorem~\ref{thm:fourier}, generalizing results of Bourgain \cite{Bourgain2010} for the real line.
The main input for our proof is a sum-product theorem for representations of real Lie groups \cite[Theorem~1.1]{HS2018}, which easily implies the discretized sum-product theorem in semisimple algebras; then we use some $L^2$-flattening lemma similar to the one used by Bourgain and Gamburd in their work on the spectral gap property.

After that, in order to apply the combinatorial results of the previous section to the random walk, we need to check that the measure $\mu^{*n}$ appropriately rescaled is not concentrated near proper affine subspaces of $E$, nor near singular elements; this is done in Section~\ref{sec:nonconcentration}.
Just as in \cite{HS2019}, the argument ultimately relies on the spectral gap property modulo primes obtained by Salehi Golsefidy and Varjú \cite{SGV}.
However, because the rescaling automorphism is no longer a homothety, the proof involves a detailed analysis of the behavior of the random walk with respect to a quasi-norm on the algebra $E$.
To help the reader understand the main ideas of the proof without having to go through all the technical details, we start with the simpler case where $E$ is simple; even in that case, the argument is different and simpler than the one presented in \cite{HS2019}, where similar estimates are needed.

In Section~\ref{sec:fourier}, we prove Theorem~\ref{thm:decaymun}, an important Fourier decay estimate for the law of the random walk.
This simply follows from a combination of the two previous sections when the group $G$ is connected, but becomes more complicated without this assumption.
We follow the argument used in \cite[Appendix B]{HLL2021}, with minor modifications.

Section~\ref{sec:wiener} makes the link between the random walk on $G$ and the random walk on $\Tbb^d$.
The Fourier decay obtained in the previous section shows that if $\mu^{*n}*\delta_{x_0}$ has one large Fourier coefficient, then reducing slightly the value of $n$, the measure $\mu^{*n}*\delta_{x_0}$ has many large Fourier coefficients.
Using a quantitative version of Wiener's lemma, one infers a first \enquote{granulation statement}: $\mu^{*n}*\delta_{x_0}$ is concentrated near a finite set of well-separated points in $\Tbb^d$.

To conclude the proof of Theorem~\ref{thm:finali}, we run backwards the random walk, starting from the granulation estimate mentioned above.
The argument uses in particular the diophantine properties of the random walk, and the exponential unstability of closed invariant subsets, obtained using a drift function, as in Eskin-Margulis \cite{em} or Benoist-Quint \cite{bq2}.
This is the content of Section~\ref{sec:granulation}.

\subsection{Concluding remarks}

\noindent\textit{Affine random walks.}
After some first results of J.-B. Boyer \cite{Boyer}, it was explained in \cite{HLL2020} how to obtain quantitative equidistribution of affine random walks on the torus, under the assumption that the action on $\R^d$ is strongly irreducible.
The arguments in that paper could be adapted to our setting.

\bigskip

\noindent\textit{More general homogeneous spaces.}
Benoist and Quint \cite{bq1,bq2,bq3} have obtained equidistribution results that are valid in the much more general setting of homogeneous spaces of Lie groups.
One drawback is that their convergence theorems are not quantitative, and only concern the Cesàro averages $\frac{1}{n}\sum_{k=0}^{n-1}\mu^{*n}*\delta_{x_0}$.

On this subject, the first author has obtained, in collaboration with Lakrec and Lindenstrauss, some partial results for affine random walks on nilmanifolds \cite{HLL2021}; these spaces may be seen as the simplest generalization of tori, but the analysis already becomes much more intricate.
Very recently, in collaboration with Bénard~\cite{benardhe}, using a new approach avoiding Fourier analysis, the first author has also been able to obtain results for random walks on finite-volume spaces of the form $G/\Lambda$, where $G$ is $\SO(2,1)$ or $\SO(3,1)$, and $\Lambda$ a lattice in $G$.

In a slightly different direction, W. Kim \cite{kim} studied effective equidistribution of expanding translates in the space of affine lattices.
Also in a different direction, Lindenstrauss and Mohammadi~\cite{LM}, Yang~\cite{yang}, and Lindenstrauss, Mohammadi and Wang~\cite{lmw} have studied effective density and equidistribution in some homogeneous spaces.
Although these equidistribution results do not deal with random walks, some of the techniques used there are similar enough to ours to be mentioned here.

\subsection{Notation} Here is a list of notation we use.
\begin{itemize}
\item $f \ll g$, $g \gg f$, $f = O(g)$, there exists a constant $C > 0$ such that $f \leq C g$.
\item $f \asymp g$ if $f \ll g$ and $g \ll f$.
\item $\Ball(x,r)$, the ball of center $x$ and radius $r$. 
\item $\Ball_V(\mybullet,\mybullet)$, ball in the ambient space $V$. 
\item $H^\circ$, the identity component with respect to the Zariski topology of the algebraic group $H$.
\item $V^*$, the space of linear forms on a linear space $V$.
\item $\lambda_1(\mu,V)$, the top Lyapunov exponent associated to the random walk on a Euclidean space $V$ defined by a probability measure $\mu$ supported on a group acting linearly on $V$.
\item $\mu*\nu$, multiplicative convolution.
\item $\mu^{*k} = \mu * \dotsm * \mu$, multiplicative convolution power.
\item $\mu \pp \nu$, additive convolution.
\item $\mu^{\pp k} = \mu \pp \dotsb \pp \mu$, additive convolution power.
\item $\mu \mm \nu$, the image measure of $\mu \otimes \nu$ under the map $(x,y) \mapsto x - y$.
\item $\1_A(x) = 1$ if $x \in A$, $\1_A(x) = 0$ otherwise. 
\item $\#A$, cardinality of a finite set $A$.  
\item $\abs{A}$, Lebesgue measure for subsets $A$ of an Euclidean space or a torus.  
\item $\qnorm{\mybullet}$, a quasi-norm
\item $\tilde{d}(\mybullet,\mybullet)$, a quasi-distance, usually associated to a quasi-norm.
\item $\tilde{\Ball}(\mybullet,\mybullet)$,  ball with respect to $\tilde{d}$.
\item $\Pbb[\mybullet]$ and $\Pbb[\mybullet\mid \mybullet]$, probability and conditional probability.
\item $f_*\mu$, image measure of $\mu$ under the map $f$.
\item $\Mat_d(\R)$, the space of $d \times d$ real matrices.
\item $\Pcal(X)$, the space of Borel probability measure on a topological space.
\item $\bracket{\mybullet,\mybullet}$, according to the context, the natural pairing $V^* \times V \to \R$ or the natural pairing $\Z^d \times \Tbb^d  \to \Tbb$. 
\end{itemize}

\section{Sum-product, \texorpdfstring{$L^2$}{L2}-flattening and Fourier decay}
\label{sec:sumprod}

In this section, we study multiplicative convolutions of measures on a semisimple associative algebra $E$.
Our goal is to derive Theorem~\ref{thm:fourier} below, which shows that under some natural non-concentration assumptions, such multiplicative convolutions admit a polynomial Fourier decay.
This generalizes results of Bourgain~\cite{Bourgain2010} for $E=\R$, of Li~\cite{Li2018} for $E=\R\oplus\dots\oplus\R$, and of \cite{HS2019} for a simple algebra $E$.

\bigskip

Let $E$ be a normed real algebra of finite dimension.
The determinant $\det_E(a)$ of an element $a\in E$ is simply defined as the determinant of the multiplication map $E\to E$, $x\mapsto ax$.
Given $\rho>0$, we let
\[
S_E(\rho) = \set{ x\in E}{\abs{\det\nolimits_E(x)} \leq\rho }.
\]
If $W\subset E$ is any subset, we let $W^{(\rho)}$ denote the $\rho$-neighborhood of $W$, defined by
\[
W^{(\rho)} = \set{x\in E}{\exists w\in W:\, \norm{x-w}<\rho}.
\]
The following definition summarizes the non-concentration conditions we shall need in order to prove some Fourier decay for multiplicative convolutions.

\begin{definition}[Non-concentration conditions]
Let $\eps > 0$, $\kappa > 0$, $\tau > 0$ be parameters.
We say a measure $\eta$ on $E$ satisfies $\NC_0(\eps,\kappa,\tau)$ at scale $\delta > 0$ if
\begin{enumerate}
\item $\Supp \eta \subset \Ball(0,\delta^{-\eps})$;
\item for every $x \in E$, $\eta(x + S_E(\delta^{\eps})) \leq \delta^{\tau}$;
\item for every $\rho \in[\delta,1]$ and every proper affine subspace $W \subset E$, $\eta(W^{(\rho)}) \leq \delta^{-\eps} \rho^\kappa$.
\end{enumerate}
We say that a measure $\eta$ on $E$ satisfies $\NC(\eps,\kappa,\tau)$ at scale $\delta > 0$ if it can be written as a sum of measures
\[
\eta = \eta_0 + \eta_1
\quad\mbox{with}\quad 
\left\{\begin{array}{l}
\eta_0\ \mbox{satisfying}\ \NC_0(\eps,\kappa,\tau)\\
\eta_1(E) \leq \delta^{\tau}.
\end{array}\right.
\]
\end{definition}

Given a finite measure $\mu$ on $E$, its Fourier transform $\hat{\mu}$ is the function on the dual space $E^*$ given by the expression
\[
\forall \xi \in E^*,\quad \hat{\mu}(\xi) = \int_E e^{2i\pi\bracket{\xi,x}}\dd\mu(x).
\]
If $\nu$ is another finite measure on $E$, the multiplicative convolution $\mu*\nu$ is defined as the image measure of $\mu\otimes\nu$ on $E\times E$ under the map $(x,y)\mapsto xy$.
It should not be confused with the additive convolution $\mu\pp\nu$, image of $\mu\otimes\nu$ under the map $(x,y)\mapsto x+y$.

\begin{thm}[Fourier decay of multiplicative convolutions]
\label{thm:fourier}
Let $E$ be a normed finite-dimensional semisimple algebra over $\R$. 
Given $\kappa > 0$, there exists $s = s(E,\kappa) \in \N$ and $\eps = \eps(E,\kappa) > 0$ such that for any parameter $\tau \in {(0, \eps \kappa)}$ the following holds for any scale $\delta > 0$ sufficiently small.

If $\eta_1, \dotsc, \eta_s$ are probability measures on $E$ satisfying $\NC(\eps,\kappa,\tau)$ at scale $\delta$,
then for all $\xi \in E^*$ with $\delta^{-1 + \eps} \leq \norm{\xi} \leq \delta^{-1 - \eps}$,
\[
\abs{(\eta_1 * \dotsm * \eta_s)^{\wedge}(\xi)} \leq \delta^{\eps \tau}.
\]
\end{thm}

For $E = \R$, this is due to Bourgain~\cite[Lemma 8.43]{Bourgain2010}. 
For algebras of the form $E = \R \oplus \dotsb \oplus \R$, this is due to Li~\cite[Theorem 1.1]{Li2018}.
We shall first prove this theorem when all $\eta_i$ are equal, i.e. $\eta_1 = \dots = \eta_s = \eta$ and then deduce the general statement from this particular case following the argument in \cite[Proof of Theorem B.3]{HLL2020}.
Alternatively, one could adapt the first part of the proof to handle directly the general case, but this would make notation cumbersome.

The proof we give for Theorem~\ref{thm:fourier} follows a strategy originating in the work of Bourgain, Glibichuk and Konyagin \cite{bgk} on exponential sums in finite fields: one deduces the bound on the exponential sum from a combinatorial \enquote{sum-product} statement, using an $L^2$-flattening statement.
In our case, the combinatorial input is a discretized sum-product theorem in semisimple algebras, which follows from a general sum-product statement for representations of real Lie groups obtained in \cite[Theorem~2.3]{HS2018}.

\subsection{Sum-product in semisimple algebras}

Sum-product estimates go back to the work of Erd\H{o}s and Szemerédi \cite{erdosszemeredi} who showed that there exists some positive constant $\eps$ such that for any subset $A$ of integers,
\[
\abs{A+A} + \abs{AA} \geq \abs{A}^{1+\eps}
\]
where $A+A$ and $AA$ denote respectively the sum-set and the product-set of $A$, defined by
\(
A+A = \{a+b\ ;\ a,b\in A\}
\)
and
\(
AA = \{ab\ ;\ a,b\in A\}
\).
In the following, we consider a normed semisimple algebra $E$ of finite dimension over $\R$, and our goal is to prove a similar statement for subsets $A\subset E$, with the cardinality replaced by the covering number $\Ncal(A,\delta)$ of $A$ at small scale $\delta>0$.
Recall that by definition, $\Ncal(A,\delta)$ is the minimal cardinality of a cover of $A$ by balls of radius $\delta$ in $E$.
In order to ensure that the covering number of $A$ at scale $\delta$ grows under addition or multiplication, one of course has to assume that $A$ is not essentially equal to a ball in some subalgebra of $E$.
We make a stronger assumption and require that $A$ is not concentrated near any proper affine subspace of $E$.

\begin{definition}[Affine non-concentration]
Let $V$ be a Euclidean space, and $\eps,\kappa > 0$ two parameters. 
We say a subset $A \subset V$ satisfies $\ANC(\eps,\kappa)$ at scale $\delta$ if 
\begin{enumerate}
\item $A \subset \Ball(0,\delta^{-\eps})$ and
\item for every $\rho \geq \delta$ and every proper affine subspace $W \subset V$, $\Ncal(A\cap W^{(\rho)},\delta) \leq \delta^{-\eps} \rho^\kappa \Ncal(A,\delta)$.
\end{enumerate}
\end{definition}

Essentially, we want to show that if $E$ is a semisimple algebra, then for every $\kappa>0$, there exists $\eps>0$ such that for any set $A\subset \Ball_E(0,1)$ satisfying $\ANC(\eps,\kappa)$ and $\delta^{-\kappa}\leq \Ncal(A,\delta)\leq\delta^{-\dim E+\kappa}$, one has $\Ncal(A+A,\delta)+\Ncal(AAA,\delta)\geq\delta^{-\eps}\Ncal(A,\delta)$.
We shall prove a slightly more technical growth statement, involving the tensor algebra $E\otimes E^\op$, where $E^\op$ denotes the algebra with the same linear structure as $E$ but with multiplication $(a,b) \mapsto ba$.
Note that the algebra $E \otimes E^\op$ acts naturally on $E$ by 
\[
\forall a,x \in E,\, \forall b\in E^\op, \quad (a\otimes b)x = axb.
\]

\begin{thm}[Sum-product in semisimple algebras]
\label{th:sumproduct}
Let $E$ be a finite-dimensional real semisimple algebra.
Given $\kappa > 0$, there exists $\eps = \eps(E,\kappa)$ such that the following holds for all $\delta > 0$ sufficiently small.
\begin{enumerate}
\item Let $A$ be a subset of $E$ satisfying $\ANC(\eps,\kappa)$ at scale $\delta$ and 
\item $\delta^{-\kappa} \leq \Ncal(A,\delta) \leq \delta^{- \dim E + \kappa}$.
\item Let $B \subset E \otimes E^\op$ be a subset satisfying $\ANC(\eps,\kappa)$ at scale $\delta$.
\end{enumerate}
Then there exists $f \in B$ such that 
\[
\Ncal(A + A,\delta) + \Ncal(A + fA,\delta) \geq \delta^{-\eps} \Ncal(A,\delta).
\]
\end{thm}

The theorem above is almost equivalent to the fact that one can obtain from $A$ a small ball in $E$ using a bounded number of sums and products.
This is the content of the proposition below, which we obtain as a simple application of \cite[Theorem~2.3]{HS2018}.
For a subset $A$ in an algebra $E$ and $s\in\N^*$, we let $\bracket{A}_s$ denote the set of elements in $E$ that can be obtained as sums of at most $s$ products of at most $s$ elements of $A$ or $-A$.

\begin{prop}[Bounded generation in semisimple algebras]
\label{pr:generateE}
Let $E$ be a finite-dimensional real semisimple algebra.
Given $\kappa > 0$ and $\eps_0>0$, there exists $\eps = \eps(E,\kappa,\eps_0) > 0$ and $s = s(E,\kappa,\eps_0) \geq 1$ such that the following holds for all $\delta > 0$ sufficiently small. 
If $A \subset \Ball(0,\delta^{-\eps})$ satisfies $\ANC(\eps,\kappa)$ at scale $\delta$ in $E$, then
\[\Ball(0,\delta^{\eps_0}) \subset \bracket{A}_s + \Ball(0,\delta).\]
\end{prop}
\begin{proof}
We consider the group $G=E^\times$ of invertible elements in $E$ and its action by multiplication on $V=E$.
By semisimplicity, we may decompose $E$ into a sum of non-trivial irreducible representations $E=\oplus_iV_i$.
Let $\pi_i\colon G\to\GL(V_i)$ denote the representation of $G$ on $V_i$.
By \cite[Theorem~2.3]{HS2018}, there is a neighbourhood $U$ of the identity in $G$ and constants $\eps = \eps(E,\kappa,\eps_0) > 0$ and $s = s(E,\kappa,\eps) \geq 1$ such that the following holds for all $\delta > 0$ sufficiently small. 
Let $A_0$ be a subset of $U$ and $A_1$ a subset of $\Ball_V(0,1)$.
Assume
\begin{enumerate}
\item for all $i = 1,\dotsc,k$, for all $\rho \geq \delta$, $\Ncal(\pi_i(A_0),\rho) \geq \delta^\eps \rho^{-\kappa}$,
\item for any linear subspace $W \subset V$ which is not $G$-invariant, there is $a \in A_0$ such that $d(a, \Stab_G(W)^\circ) \geq \delta^\eps$,
\item for any proper $G$-invariant linear subspace $W \subset V$, there is $a \in A_1$ such that $d(a,W) \geq \delta^\eps$. 
\end{enumerate}
Then 
\[
\Ball_V(0,\delta^{\eps_0}) \subset \bracket{A_0,A_1}_s + \Ball(0,\delta).
\]
Here, $\bracket{A_0,A_1}_s$ denotes the set of elements in $V$ that can be obtained as sums of at most $s$ products of at most $s$ elements of $A_0$ and elements  of $A_1 \cup (-A_1)$.
In the argument below, we apply this result with $\eps$ replaced by $O(\eps/\kappa)$.

Our set $A$ is not necessarily contained in the neighborhood $U$, but we may cover $A$ by translates of $U$ in $E$, and then, by the pigeonhole principle, there is $a \in A$ such that $A_0 = (A - a) \cap U$ satisfies
\[
\Ncal(A_0,\delta) \gg_U \delta^{O(\eps)} \Ncal(A,\delta).
\]
This set $A_0$ satisfies $\ANC(O(\eps),\kappa)$ at scale $\delta$.
This non-concentration condition applied to affine suspaces parallel to $\oplus_{j\neq i}V_j$ shows that the first condition above is verified.
Moreover, if $W\subset E$ is not $G$-invariant, then the algebra generated by $\Stab_G(W)$ is a proper subalgebra of $E$.
In particular, it is included in a proper affine subspace of $E$, and by $\ANC(O(\eps),\kappa)$, there must exist $a$ in $A_0$ such that $d(a,\Stab_G(W))\geq\delta^{O(\eps/\kappa)}$; so the second condition is also satisfied.
To conclude, take $A_1=A_0$, which satisfies the third condition with $\eps$ replaced by $O(\eps/\kappa)$.
\end{proof}

In short, Theorem~\ref{th:sumproduct} will follow from Proposition~\ref{pr:generateE} applied to the set $B$ in the tensor algebra $E\otimes E^\op$, and from the Plünnecke-Ruzsa inequality.

\begin{proof}[Proof of Theorem~\ref{th:sumproduct}]
For $K \geq 1$, define
\[
R_\delta(A,K) = \setbig{f \in E\otimes E^\op}{ \Ncal(A + f A, \delta) \leq K \Ncal(A,\delta)}.
\]
Let us show that $R_\delta(A,K)$ is almost stable under addition and multiplication.
By Ruzsa's covering lemma, if $f\in R_\delta(A,K)$, there exists a set $X_f$ such that $\Ncal(X_f,\delta)=O(K)$ and
\[
fA \subset A-A + X_f.
\]
Therefore, for $f_1,f_2$ in $\R_\delta(A,K)$, one has
\[
A+(f_1+f_2)A \subset A+f_1A+f_2A \subset 3A-2A + X_{f_1} + X_{f_2}.
\]
With the Plünnecke-Ruzsa inequality, this yields $\Ncal(A+(f_1+f_2)A,\delta)\leq K^{O(1)}\Ncal(A,\delta)$, i.e. $f_1+f_2$ is in $R_\delta(A,K^{O(1)})$.
Similarly, $f_1f_2\in R_\delta(A,K^{O(1)})$.
By induction, this implies that for $s\in\N$,
\[
\bracket{R_\delta(A,K)}_s + \Ball_{E\otimes E^\op}(0,\delta) \subset R_\delta(A,K^{O_s(1)}).
\] 
Now assume for a contradiction that $B \cup \{1\} \subset R_\delta(A,\delta^{-\eps})$.
Since $E$ is a semisimple algebra, $E \otimes E^\op$ is also one.
Thus, by Proposition~\ref{pr:generateE} applied to the set $B\subset E \otimes E^\op$, for any $\eps_0 > 0$, there is $s=s(E,\kappa,\eps_0) \geq 1$ such that
\[
\Ball_{E\otimes E^\op}(0,\delta^{\eps_0}) \subset \bracket{B}_s + \Ball(0,\delta)
\]
and therefore,
\[
\Ball_{E\otimes E^\op}(0,\delta^{\eps_0}) \subset \bracket{B}_s + \Ball(0,\delta) \subset R_\delta(A,\delta^{-O_s(\eps)}).
\] 
In particular, $\delta^{\eps_0} \in R_\delta(A,\delta^{-O_s(\eps)})$.
This certainly implies $\delta^{-\eps_0} \in R_\delta(A,\delta^{-O_s(\eps_0 + \eps)})$ and then, using once more stability of $R_\delta(A,K)$ under product,
\[
\Ball_{E\otimes E^\op}(0,1) \subset R_\delta(A,\delta^{-O_s(\eps_0 + \eps)}).
\]
If $\eps_0$ and $\eps$ are chosen small enough, this contradicts Lemma~\ref{lm:EEop1onA} below.
\end{proof}

We are left to show the next lemma, stating that if $A$ has $\ANC(\eps,\kappa)$ at scale $\delta$, then $\Ball_{E\otimes E^\op}(0,1)$ is not contained in $R_\delta(A,\delta^{-\eps})$.
\begin{lem}
\label{lm:EEop1onA}
Let $E= E_1 \oplus \dots \oplus E_r$ be a finite-dimensional real semisimple algebra decomposed as a direct sum of minimal two-sided ideals.
Write $\pi_j \colon E \to E_j$ for the corresponding projections.

Given $\kappa > 0$, there exists $\eps = \eps(E,\kappa) > 0$ such that the following holds for all $\delta>0$ sufficiently small.
Let $A \subset \Ball(0,\delta^{-\eps})$ be a subset of E. 
Assume
\begin{enumerate}
\item $\Ncal(A,\delta) \leq \delta^{-\dim E + \kappa}$
\item for each $j = 1,\dotsc, r$, $\max_{x\in E_j} \Ncal( A \cap \pi_j^{-1}(\Ball_{E_j}(x,\rho)),\delta) \leq \rho^\kappa \Ncal(A,\delta)$,
where $\rho=\delta^{\frac{\kappa}{\kappa+\dim E}}$.
\end{enumerate}
Then there exists $f \in \Ball_{E\otimes E^\op}(0,1)$ such that 
\[
\Ncal(A + fA, \delta) > \delta^{-\eps} \Ncal(A,\delta).
\]
\end{lem}
\begin{proof}
The image of $E \otimes E^\op$ in $\End(E)$ is equal to the image of $\bigoplus_{j = 1}^r E_j \otimes E_j^\op$. 
Let $f_j$, $j=1,\dots,r$ be a family of jointly independent random elements of $\Ball_{E_j\otimes E_j^\op}(0,1)$ distributed according to the Lebesgue measure on $E_j \otimes E_j^\op$, and set
\[
f = f_1 + \dots + f_r
\]
regarded as a random element of $\End(E)$.
In the following argument, probabilities and expectations are taken with respect to these random variables.
For each~$j$, since the algebra $E_j$ is simple, the action of $E_j \otimes E_j^\op$ on $E_j$ is irreducible. 
Hence $E_j \otimes E_j^\op(y) = E_j$ for any non-zero $y \in E_j$ and consequently,
\begin{equation}
\label{eq:fiymoinsx}
\forall \delta > 0,\, \forall x, y \in E_j,\quad \Probbig{ \norm{f_j(y) - x} \leq \delta} \ll \delta^{\dim E_j} \norm{y}^{-\dim E_j}.
\end{equation}

Consider the map
\[
\begin{array}{llcl}
\phi \colon & A \times A & \to & E\\
& (x,y) & \mapsto & x + fy
\end{array}
\]
The energy of the map $\phi$ at scale $\delta > 0$ is defined as
\[
\Ecal_\delta(\phi,A\times A)  = \Ncal\bigl( \set{(a,a',b,b') \in A \times A\times A\times A\ }{\norm{\phi(a,b) - \phi(a',b')} \leq \delta}, \delta\bigr).
\]
By the Cauchy-Schwarz inequality --- see also \cite[Lemma 12(i)]{He2016},
\[
\Ncal(\phi(A\times A),\delta) = \Ncal(A + fA,\delta) \geq \frac{\Ncal(A,\delta)^4}{\Ecal_\delta(\phi,A \times A)}.
\]
Taking expectations and applying Jensen's inequality, we find
\begin{equation}
\label{eq:JensenEnergy}
\Ebb\bigl[\Ncal(A + fA,\delta) \bigr] \geq \frac{\Ncal(A,\delta)^4}{\Ebb\bigl[\Ecal_\delta(\phi,A \times A) \bigr]}
\end{equation}
so it suffices to bound $\Ebb\bigl[\Ecal_\delta(\phi,A \times A) \bigr]$ from above.

For that, let $\tilde{A}$ be a maximal $\delta$-separated subset of $A$.
By \cite[Lemma~12(ii)]{He2016},
\[
\Ebb\bigl[\Ecal_\delta(\phi,A \times A)\bigr] \leq \sum_{x,y,x',y' \in \tilde{A}} \Probbig{ f(y'-y) \in \Ball(x - x',5\delta)}.
\]
Let $\rho = \delta^{\frac{\kappa}{\dim E+ \kappa}}$. 
We split the sum into two parts according to whether 
\[
\forall j = 1,\dotsc, r,\quad \norm{\pi_j(y' - y)} \geq \rho.
\] 
If this is the case, then \eqref{eq:fiymoinsx} implies 
\[
\Probbig{ f(y'-y) \in \Ball(x - x',5\delta)} \ll \delta^{\dim E} \rho^{-\dim E}.
\]
Otherwise, there is $j \in \{1,\dotsc,r\}$ such that $\pi_j(y') \in \Ball(\pi_j(y),\rho)$. For fixed $y$ the number of such $y'$ in $\tilde{A}$ is
\[
\# \bigl(\tilde{A} \cap \pi_j^{-1}(\Ball(\pi_j(y),\rho))\bigr) \ll \Ncal( A \cap \pi_j^{-1}(\Ball(\pi_j(y),\rho)),\delta) \leq \rho^\kappa \Ncal(A,\delta).
\]
Moreover for fixed $y,y'$ and $x$, we have
\[
\sum_{x' \in \tilde{A}} \Probbig{ f(y'-y) \in \Ball(x - x',5\delta)} \ll 1
\]
because the balls $\Ball(x',5\delta)$ have overlap multiplicity at most $O(1)$.
Putting these considerations together, we obtain
\begin{align*}
\Ebb\bigl[\Ecal_\delta(\phi,A \times A)\bigr] &\ll \delta^{\dim E} \rho^{-\dim E} \Ncal(A,\delta)^4 + \rho^\kappa \Ncal(A,\delta)^3 \\
&\leq \bigl( \delta^{\kappa} \rho^{-\dim E} + \rho^\kappa \bigr) \Ncal(A,\delta)^3\\
&\ll \delta^{\frac{\kappa^2}{\dim E+ \kappa}} \Ncal(A,\delta)^3
\end{align*}
Combined with \eqref{eq:JensenEnergy}, this finishes the proof of the lemma.
\end{proof}

\subsection{\texorpdfstring{$L^2$}{L2}-flattening}
Our goal is now to translate the sum-product theorem obtained above in terms of measures on the semisimple algebra $E$.
The result we obtain is an $L^2$-flattening lemma for additive and multiplicative convolutions of measures on $E$.
Statements of this form already appear implicitly in the work of Bourgain \cite{bourgain_evkt,Bourgain2010} on the Erd\H{o}s-Volkmann ring conjecture, and were later much popularized by their application to the spectral gap problem by Bourgain and Gamburd \cite{bourgaingamburd_sl2p,bourgaingamburd_su2}.
They are usually derived from the analogous combinatorial growth statement, via a decomposition of the measures into dyadic level sets, combined with an application of the Balog-Szemerédi-Gowers lemma.

Before we can state our result, we give a non-concentration condition for measures on $E$, analogous to the one given for subsets in the previous paragraph.

\begin{definition}[Affine non-concentration for measures]
Let $V$ be a Euclidean space, and $\eps,\kappa > 0$ two parameters. 
We say that a measure $\eta$ on $V$ satisfies $\ANC(\eps,\kappa)$ at scale $\delta$ if 
\begin{enumerate}
\item $\Supp\eta \subset \Ball(0,\delta^{-\eps})$;
\item for every $\rho \geq \delta$ and every proper affine subspace $W \subset V$, $\eta(W^{(\rho)}) \leq \delta^{-\eps} \rho^\kappa$.
\end{enumerate}
\end{definition}

In this paper, measures are often studied at some fixed small positive scale $\delta$.
For that reason, it is convenient to define the \emph{regularized measure} $\eta_\delta$ of a measure $\eta$ on $E$ at scale $\delta$ by
\[
\eta_\delta = \eta\pp P_\delta
\]
where $P_\delta=\frac{\1_{\Ball(0,\delta)}}{\abs{\Ball(0,\delta)}}$ is the normalized indicator function of the ball of radius $\delta$ centered at $0$.
The measure $\eta_\delta$ will be identified with its density with respect to the Lebesgue measure on $E$, and we write
\[
\norm{\eta}_{2,\delta} = \norm{\eta_\delta}_2.
\]

\begin{prop}[\texorpdfstring{$L^2$}{L2}-flattening]
\label{pr:l2flat}
Let $E$ be finite-dimensional semisimple algebra over $\R$.
Given $\kappa > 0$, there exists $\eps = \eps(E,\kappa)$ such that the following holds for all $\delta > 0$ sufficiently small.
Let $\eta$ be a probability measure on $E$ satisfying
\begin{enumerate}
\item $\eta$ is supported on $E \setminus S_E(\delta^\eps)$;
\item $\eta$ satisfies $\ANC(\eps,\kappa)$ at scale $\delta$ on $E$;
\item $\delta^{-\kappa + \eps} \leq \norm{\eta}_{2,\delta}^2 \leq \delta^{- \dim E + \kappa - \eps}$.
\end{enumerate}
Then,
\[
\norm{\eta \ff \eta \ff \eta \mm \eta \ff \eta \ff \eta}_{2,\delta} \leq \delta^{\eps} \norm{\eta}_{2,\delta}.
\]
\end{prop}

We wish to deduce this proposition from Theorem~\ref{th:sumproduct}.
A first useful observation is that the non-concentration condition for measures is closely related to non-concentration for subsets.

\begin{lem}\label{lm:ANC2set}
Given an Euclidean space $V$, and parameters $\epsilon > 0$ and $\kappa > 0$,
the following holds for all $\delta > 0$ sufficiently small.
\begin{enumerate}
\item If $A \subset V$ has $\ANC(\eps,\kappa)$ at scale $\delta$, then there is a measure supported on $A$ which has $\ANC(2\eps,\kappa)$ at scale $\delta$.
\item \label{it:ANC2set2} Let $\eta$ be a probability measure on $V$ satisfying $\ANC(\eps,\kappa)$ at scale $\delta$.
If $A \subset V$ is a subset such that $\eta(A) \geq \delta^\eps$ then there is a subset $A' \subset A$ which satisfies $\ANC(6\eps,\kappa)$ at scale $\delta$.
\end{enumerate}
\end{lem}
\begin{proof}
For the first item, let $\tilde{A}$ be a maximal $\delta$-separated subset of $A$.
The normalized counting measure on $\tilde{A}$ satisfies the desired property.
The second item is slightly more subtle.
Since the normalized restriction of $\eta$ to $A$ satisfies $\ANC(2\eps,\kappa)$, we may assume without loss of generality that $A = \Supp\eta$.
Let $i_{\min}$ be the largest integer such that $2^{i_{\min}} \leq \delta^{2\eps \dim V}$.
For every integer $i \geq i_{\min}$, set 
\[
A_{i,0} = \setBig{a \in A}{2^{i-1} < \frac{\eta(\Ball(a,2\delta))}{\abs{\Ball(0,\delta)}} \leq 2^i}.
\]
and then
\[
A_{-,0} = A \setminus \bigcup_{i \geq i_{\min}} A_{i,0}.
\]
Next, for every $i \geq i_{\min}$, set $A_i = A_{i,0}^{(\delta)}$ and also $A_- = A_{-,0}^{(\delta)}$.
By this construction,
\begin{equation}
\label{eq:A-+Ai}
\eta_\delta \ll  \delta^{2 \eps \dim V} \1_{A_-} + \sum_{i \geq i_{\min}} 2^i \1_{A_i} 
\end{equation}
and 
\begin{equation}
\label{eq:Aieta3delta}
\forall i \geq i_{\min},\quad 2^i \1_{A_i}  \ll \eta_{3\delta}.
\end{equation}

Note that $A_i$ is empty whenever $i \geq -\frac{\log \abs{\Ball(0,\delta)}}{\log 2} + 1$.
Thus, integrating \eqref{eq:A-+Ai} and recalling  $\Supp \eta  \subset B(0,\delta^{-\eps})$ from $\ANC(\eps,\kappa)$ for $\eta$, we obtain some $i \geq i_{\min}$ such that 
\(
2^i \abs{A_i} \geq \delta^\eps.
\)
Fix this $i$ and set $A'=A_{i,0}$.
If $W$ is a proper affine subspace in $V$ and $\rho\geq\delta$, we can bound using \eqref{eq:Aieta3delta} and $\ANC(\eps,\kappa)$ for $\eta$,
\begin{align*}
\Ncal(A'\cap W^{(\rho)},\delta)
	& \ll \delta^{-\dim V} \int_V \1_{A_i\cap W^{(\rho)}}(x)\dd x\\
	& \ll \delta^{-\dim V} \int_V 2^{-i}\1_{W^{(\rho)}}(x) \dd\eta_{3\delta}(x)\\
	& = \delta^{-\dim V} 2^{-i}\eta_{3\delta}(W^{(\rho)})\\
	& \leq \delta^{-\dim V} 2^{-i} \delta^{-\eps}\rho^\kappa
\end{align*}
and using the above lower bound on $2^i\abs{A_i}$, we get
\begin{align*}
\Ncal(A'\cap W^{(\rho)},\delta)
	& \ll \delta^{-\dim V} \abs{A_i} \delta^{-2\eps}\rho^\kappa\\
	& \ll \delta^{-2\eps} \rho^\kappa \Ncal(A_i,\delta)
\end{align*}
This shows that $A'$ satisfies $\ANC(3\eps,\kappa)$ at scale $\delta$.
\end{proof}

The next lemma is similar in spirit to the previous one.
Roughly speaking, given measures $\eta$ on $V$ and $\mu$ on $\GL(V)$ such that the convolution $\mu*\eta\mm\mu*\eta$ has large $L^2$-norm at scale $\delta$, we construct related subsets $A\subset V$ and $B\subset\GL(V)$ such that $A-BA$ is not much larger than $A$. 
This is the central part of the proof of Proposition~\ref{pr:l2flat}; it relies on the Balog-Szemerédi-Gowers lemma.

\begin{lem}
\label{lm:bsg}
Let $V$ be a Euclidean space and $\mu$ a probability measure on $\GL(V)$ such that 
\[
\forall g \in \Supp\mu,\quad \norm{g} + \norm{g^{-1}} \leq \delta^{-\eps}.
\]
Let $\eta$ be a probability measure on $\Ball_V(0,\delta^{-\eps})$ such that 
\[
\norm{\mu*\eta \mm \mu*\eta}_{2,\delta} > \delta^{\eps} \norm{\eta}_{2,\delta}.
\]
Then there exist a subset $A \subset \Ball_V(0,\delta^{-O(\eps)})$ and an element $g_1 \in \Supp\mu$ such that
\[
\delta^{-\dim V + O(\eps)} \norm{\eta}_{2,\delta}^{-2} \leq \Ncal(A,\delta) \leq \delta^{-\dim V - O(\eps)} \norm{\eta}_{2,\delta}^{-2}
\]
and 
\[
\mu\bigl( \setbig{g \in \GL(V)}{ \Ncal(A - gg_1^{-1}A,\delta) \leq \delta^{-O(\eps)}\Ncal(A,\delta)} \bigr) \geq \delta^{O(\eps)}.
\]
If moreover $\eta$ satisfies $\ANC(\eps,\kappa)$ in $V$ at scale $\delta$ for some $\kappa > 0$ then $A$ satisfies $\ANC(O(\eps),\kappa)$.
\end{lem}
\begin{proof}
We use the following rough comparison notation : for positive quantities $f$ and $g$, we write $f \lesssim g$ for $f \leq \delta^{-O(\eps)} g$ and $f \sim g$ for $f \lesssim g$ and $g \lesssim f$.

We have
\[
\norm{\mu*\eta_\delta \mm \mu*\eta_\delta}_2 \gtrsim \norm{\mu*\eta \mm \mu*\eta}_{2,\delta} \gtrsim \norm{\eta_\delta}_2.
\]
As in the proof of Lemma~\ref{lm:ANC2set}, we can approximate $\eta_\delta$ using dyadic level sets : 
there are $\delta$-discretized sets\footnote{A $\delta$-discretized set is a union of balls of radius $\delta$.} $(A_i)_{i\geq 0}$ in $\Ball_V(0,\delta^{-\eps})$ such that $A_i$ is empty for $i \gg \log\frac{1}{\delta}$ and
\begin{equation}
\label{eq:Aiapproxeta}
\eta_\delta \ll \sum_{i \geq 0} 2^i \1_{A_i} \lesssim \eta_{3\delta} + \1_{A_0}.
\end{equation}
By the pigeonhole principle, there are $i,j\geq 0$ such that
\begin{align*}
\norm{\eta_\delta}_2 &\lesssim \norm{\mu*\eta_\delta \mm \mu*\eta_\delta}_2\\
&\lesssim 2^{i+j}\norm{\mu*\1_{A_i} \mm \mu*\1_{A_j}}_2\\
&\lesssim 2^{i+j} \int_{\GL(V) \times \GL(V)} \norm{\1_{gA_i} \mm \1_{g'A_j}}_2 \dd(\mu \otimes \mu) (g,g').
\end{align*}
In the last inequality, we used $g*\1_{A_i} = \abs{\det{g}}^{-1}\1_{g A_i}$ and $\abs{\det g} \sim 1$ for all $g \in \Supp\mu$.
By the right-hand inequality in \eqref{eq:Aiapproxeta}, we have
\[
2^i\abs{A_i} \lesssim 1 \quad \text{and} \quad 2^j\abs{A_j}^{\frac{1}{2}} \lesssim \norm{\eta_\delta}_2
\]
and similarly
\[
2^j\abs{A_j} \lesssim 1 \quad \text{and} \quad 2^i\abs{A_i}^{\frac{1}{2}} \lesssim \norm{\eta_\delta}_2.
\]
By Young's inequality and the estimate on $\det g$, we have for all $g,g' \in \Supp\mu$, 
\[2^{i+j}  \norm{\1_{gA_i} \mm \1_{g'A_j}}_2  \lesssim \norm{\eta_\delta}_2.\]
Thus, by the pigeonhole principle again, there exists $g_0 \in \Supp\mu$ and a set $B_0 \subset \Supp\mu$ such that $\mu(B_0)\gtrsim 1$ and for all $g \in B_0$,
\[
\norm{\eta_\delta}_2 \gtrsim 2^{i+j} \norm{\1_{g_0A_i} \mm \1_{gA_j}}_2 \gtrsim \norm{\eta_\delta}_2.
\]
By the above estimates, this implies
\begin{align*}
\norm{\1_{g_0A_i} \mm \1_{gA_j}}_2^2 & \gtrsim 2^{-2i-2j}\norm{\eta_\delta}_2^2\\
& \gtrsim 2^{-i-j}\abs{A_i}^{\frac{1}{2}} \abs{A_j}^{\frac{1}{2}}\\
& \gtrsim \abs{A_i}^{\frac{3}{2}}\abs{A_j}^{\frac{3}{2}}\\
& \sim \abs{g_0A_i}^{\frac{3}{2}}\abs{gA_j}^{\frac{3}{2}}.
\end{align*}
By the Balog-Szemerédi-Gowers lemma \cite[Theorem~6.10]{Tao2008}, for each $g \in B_0$ there are $\delta$-discretized subsets $A_g \subset A_i$ and $A'_g \subset A_j$ such that
\[
\abs{A_g} \sim \abs{A_i},\,\abs{A'_g} \sim \abs{A_j} \text{, and}\quad  \Ncal(g_0A_g - gA'_g,\delta)\lesssim \Ncal(g_0A_g,\delta)^{\frac{1}{2}} \Ncal(gA'_g,\delta)^{\frac{1}{2}}.
\]
Set $X=A_i\times A_j$ and $X_g=A_g\times A'_g\subset X$ and write
\begin{align*}
\iint \abs{X_{g_1}\cap X_{g_2}}\dd\mu(g_1)\dd\mu(g_2)  & = \iint \int \1_{X_{g_1}}(x)\1_{X_{g_2}}(x) \dd x\dd\mu(g_1)\dd\mu(g_2) \\
& = \int \left(\int \1_{X_g(x)}\dd\mu(g)\right)^2 \dd x\\
& \geq \frac{1}{\abs{X}} \left(\iint \1_{X_g}(x)\dd x\dd\mu(g)\right)^2\\
& \gtrsim \abs{X}.
\end{align*}
This shows that there exists $g_1$ and $B_1\subset B_0$ such that $\mu(B_1)\gtrsim 1$ for all $g$ in $B_1$, $\abs{X_{g_1}\cap X_{g}}\gtrsim \abs{X}$.
Equivalently,
\[
\forall g \in B_1,\quad \abs{A_{g_1} \cap A_g} \sim \abs{A_{g_1}} \sim \abs{A_g}\text{ and } \abs{A'_{g_1} \cap A'_g} \sim \abs{A'_{g_1}} \sim \abs{A'_g}.
\]
For subsets $A,A' \in V$, write $A \approx A'$ if $\Ncal(A-A',\delta) \lesssim \Ncal(A,\delta)^{\frac{1}{2}} \Ncal(A',\delta)^{\frac{1}{2}}$.
Ruzsa's triangle inequality~\cite[Lemma 2.6]{TaoVu} is still valid for covering numbers at scale $\delta$,
so if $A\approx A'$ and $A'\approx A''$, then $A\approx A''$, and moreover, if $A\approx A'$ for some sets $A$ and $A'$, then $A\approx A$ and $A'\approx A'$.

The above shows that for every $g\in B_0$, $g_0A_g\approx gA_g'$.
This implies $g_0A_g\approx g_0A_g$, and since $g$ is $\delta^{-\eps}$-Lipschitz, $A_g\approx A_g$.
Therefore, for $g_1\in B_0$ and $g\in B_1$ as above, we find $A_{g_1}\approx A_{g_1}\cap A_g\approx A_g$.
Similarly, $A'_{g_1}\approx A'_g$,.
Finally
\[
g_0A_{g_1} \approx g_0A_g \approx gA'_g \approx gA'_{g_1} \approx g g_1^{-1}g_0A_{g_1},
\]
showing that $A = g_0 A_{g_1}$ has all the desired properties.

For last assertion, note that $\eta(A_i)\gtrsim 1$. By the proof of Lemma~\ref{lm:ANC2set}\ref{it:ANC2set2}, $A_i$ satisfies $\ANC(O(\eps),\kappa)$ and hence so do $A_{g_1}$ and $A$.
Note also that 
\[
\Ncal(A,\delta)\sim \Ncal(A_{g_1},\delta)\sim \Ncal(A_i,\delta) \sim \delta^{-\dim V}\abs{A_i} \sim \delta^{-\dim V}\norm{\eta}_{2,\delta}^2.
\]
\end{proof}

To prove Proposition~\ref{pr:l2flat} we use the above lemma for the action of $E\otimes E^\op$ on $E$, and then apply the sum-product theorem in $E$.

\begin{proof}[Proof of Proposition~\ref{pr:l2flat}]
Let $\mu$ be the image measure of $\eta\otimes\eta$ in $\GL(E)$, so that $\mu * \eta \mm \mu * \eta = \eta \ff \eta \ff \eta \mm \eta \ff \eta \ff \eta$.
We argue by contradiction:
Assuming
\[
\norm{\mu*\eta \mm \mu*\eta}_{2,\delta} > \delta^{\eps} \norm{\eta}_{2,\delta},
\]
we shall construct sets $A,B$ satisfying all assumptions of Theorem~\ref{th:sumproduct}, with $\kappa$ and $\eps$ replaced by $\kappa/2$ and $O(\eps)$ but violating its conclusion.
By Lemma~\ref{lm:bsg} there is a subset $A \subset E$ satisfying $\ANC(O(\eps),\kappa)$ at scale $\delta$ and an element $g_1 \in \Supp\mu$ such that 
\[
\delta^{-\kappa + O(\eps)}\leq \Ncal(A,\delta) \leq \delta^{- \dim E + \kappa - O(\eps)}.
\]
and
\[
\mu\bigl( \setbig{g \in \GL(V)}{ \Ncal(A - gg_1^{-1}A,\delta) \leq \delta^{-O(\eps)}\Ncal(A,\delta)} \bigr) \geq \delta^{O(\eps)}.
\]
By definition of $\mu$, we may write $g_1=a_1\otimes b_1$, and the above inequality becomes
\begin{equation}\label{baba}
(\eta * a_1^{-1})\otimes (b_1^{-1} *\eta)\bigl( \setbig{(a,b) \in E \times E}{ \Ncal(A - aAb,\delta) \leq \delta^{-O(\eps)}\Ncal(A,\delta)}  \bigr) \geq \delta^{O(\eps)}.
\end{equation}
Since $a_1,b_1 \not\in S_E(\delta^\eps)$, the measures $\eta * a_1^{-1}$ and $b_1^{-1} *\eta$ satisfy $\ANC(O(\eps),\kappa)$ at scale $\delta$.
Moreover, Lemma~\ref{lm:tensorANC} below shows that $(\eta * a_1^{-1})\dot\otimes (b_1^{-1} *\eta)$ satisfies $\ANC(O(\eps),\frac{\kappa}{2})$.
By Lemma~\ref{lm:ANC2set}\ref{it:ANC2set2}, there exists a subset $B\subset\Supp((\eta * a_1^{-1})\dot\otimes (b_1^{-1} *\eta))$ satisfying $\ANC(O(\eps),\frac{\kappa}{2})$.
Equation~\eqref{baba} shows that $\Ncal(A+fA,\delta)\leq\delta^{-O(\eps)}\Ncal(A,\delta)$ for all $f$ in $B$.
Since $\eta$ is supported on $E\setminus S_E(\delta^\eps)$, one also has $\Ncal(fA,\delta)\geq\delta^{O(\eps)}\Ncal(A,\delta)$, and so by Plünnecke's inequality one also has $\Ncal(A+A,\delta)\leq\delta^{-O(\eps)}\Ncal(A,\delta)$ and in turn.
\[
\Ncal(A+A,\delta) + \Ncal(A+fA,\delta) \leq 
\delta^{-O(\eps)}\Ncal(A,\delta).
\]
Thus, $A$ and $B$ violate the conclusion of Theorem~\ref{th:sumproduct} with parameters $\kappa/2$ and $O(\eps)$.
This yields the desired contradiction, provided $\eps$ is chosen small enough.
\end{proof}

\begin{lem}
\label{lm:tensorANC}
Let $V_1$ and $V_2$ be finite-dimensional linear spaces.
For each $i = 1,2$, let $\eta_i$ be a measure on $V_i$ and denote by $\eta_1 \dot\otimes \eta_2$ the image measure of $\eta_1 \otimes \eta_2$ by the natural bilinear map $V_1 \times V_2 \to V_1 \otimes V_2$.

Given two parameters $\eps,\kappa > 0$, the following holds for $\delta > 0$ sufficiently small.
If $\eta_1$ and $\eta_2$ both satisfy $\ANC(\eps,\kappa)$ at scale $\delta$, then $\eta_1 \dot\otimes \eta_2$ satisfies $\ANC(2\eps,\frac{\kappa}{2})$ in $V_1 \otimes V_2$ at scale $\delta$.
\end{lem}
\begin{proof}
Let $v_1$ and $v_2$ be independent random variables taking values respectively in $V_1$ and $V_2$ and distributed according to $\eta_1$ and $\eta_2$.
To establish $\ANC(2\eps,\frac{\kappa}{2})$ for $\eta_1 \dot\otimes \eta_2$, it is enough to show that for any linear form $\phi \in (V_1 \otimes V_2)^*$ with $\norm{\phi} = 1$, any $t \in \R$ and any $\rho \geq \delta$, we have
\begin{equation}
\label{eq:tensorANCP0}
\Probbig{\abs{\phi(v_1 \otimes v_2) - t} < \rho} \ll \delta^{-\eps}\rho^{\kappa/2}.
\end{equation}

Note that $(V_1 \otimes V_2)^* = V_1^* \otimes V_2^*$.
Hence, letting $(\psi_1,\dotsc,\psi_d)$ be a orthonormal basis of $V_1^*$, we can write $\phi \in (V_1 \otimes V_2)^*$ as 
\[\phi = \psi_1 \otimes \phi_1 + \dots + \psi_d \otimes \phi_d\]
where $\phi_1, \dotsc, \phi_d \in V_2^*$ are uniquely determined.
Moreover,
\begin{equation}
\label{eq:normphi}
1 = \norm{\phi}^2 = \norm{\phi_1}^2 + \dots + \norm{\phi_d}^2.
\end{equation}

On the one hand, when $v_2$ is fixed, the map
\[
v_1 \mapsto \phi(v_1 \otimes v_2) = \sum_{i=1}^d \psi_i(v_1) \phi_i(v_2)
\]
is the linear form $  \sum_{i=1}^d \phi_i(v_2) \psi_i \in V_1^*$, which has norm $\sum_{i=1}^d \abs{\phi_i(v_2)}^2$.
Thus, by independence of $v_1$ and $v_2$ and property $\ANC(\eps,\kappa)$ for $\eta_1$, we can estimate the conditional probability
\begin{equation}
\label{eq:tensorANCP1}
\ProbcondBig{\abs{\phi(v_1 \otimes v_2) - t} < \rho}{\sum_{i=1}^d \abs{\phi_i(v_2)}^2 \geq \rho^{1/2}} \leq  \delta^{-\eps}\rho^{\kappa/2}.
\end{equation}

On the other hand, by property $\ANC(\eps,\kappa)$ for $\eta_2$, for each $i = 1,\dotsc,d$,
\[
\Probbig{\abs{\phi_i(v_2)} \leq \rho^{1/2}\norm{\phi_i}} \leq \delta^{-\eps} \rho^{\kappa/2}.
\]
Hence, recalling \eqref{eq:normphi},
\begin{equation}
\label{eq:tensorANCP2}
\ProbBig{\sum_{i=1}^d \abs{\phi_i(v_2)}^2 < \rho} \leq d\delta^{-\eps} \rho^{\kappa/2}.
\end{equation}
Inequalities \eqref{eq:tensorANCP1} and \eqref{eq:tensorANCP2} together imply \eqref{eq:tensorANCP0} and finish the proof of the lemma.
\end{proof}

\subsection{Fourier decay}
To prove Theorem~\ref{thm:fourier} we apply the $L^2$-flattening Proposition~\ref{pr:l2flat} repeatedly.
The measures we obtain are images of tensor powers $\eta^k$ under polynomial maps $E^k \to E$, and we need to compare their Fourier decay to that of simple multiplicative convolutions of $\eta$.
This is the content of the next lemma.
This technique will also be useful to weaken slightly the assumptions of Theorem~\ref{thm:fourier}, see Corollary~\ref{cor:fourier} below.

\begin{lem}
\label{lm:eta3mmeta3}
Let $E$ be any real associative algebra, and let $\eta$ be a measure on $E$ with $\eta(E) \leq 1$.
Let $\mu = \eta * \eta * \eta \mm \eta * \eta * \eta$ then for any integer $m \geq 1$,
\[
\forall \xi \in E^*,\quad \abs{\widehat{\eta^{*3m}}(\xi)}^{2^m} \leq \widehat{\mu^{*m}}(\xi).
\]
\end{lem}
\begin{proof}
By \cite[Lemma B.6]{HLL2021}, if $\eta$, $\eta'$, $\eta''$ are probability measures on $E$, then the Fourier transform of $\eta * (\eta' \mm \eta') * \eta''$ takes non-negative real values and moreover,
\[
\forall \xi \in E^*,\quad \abs{(\eta * \eta' * \eta'')^{\wedge}(\xi)}^{2} \leq \bigl(\eta * (\eta' \mm \eta') * \eta''\bigr)^{\wedge}(\xi).
\]
By a simple scaling argument we see that the same holds when $\eta$, $\eta'$, $\eta''$ are finite measures with total mass $\eta(E), \eta'(E), \eta''(E) \leq 1$.  
Using this inequality $m$ times with measure $\eta'=\eta^{*3}$, so that $\mu=\eta'\mm\eta'$, we get
\begin{align*}
\hhat{\mu^{*m}}(\xi) & \geq \abs{\widehat{\mu^{*(m-1)}*\eta^{*3}}(\xi)}^2
 \geq \dots
 \geq \abs{\widehat{\eta^{*3m}}(\xi)}^{2^m}.
\end{align*}
\end{proof}

We shall also need a lemma on Fourier decay for multiplicative convolutions of measures with small $L^2$-norm.
In the case where $E=\R$, such bounds originate in the work of Falconer~\cite{falconer} on projection theorems, and appear explicitly in Bourgain~\cite[Theorem~7]{Bourgain2010}.
The result below is taken from \cite[Lemma 2.9]{HS2019}.

\begin{lem}
\label{lm:L2Fourier}
Let $E$ be a finite-dimensional real associative algebra with unit.
The following holds for any parameters $\kappa > 0$ and $\eps > 0$ and any scale $\delta > 0$ small enough. 
Let $\eta$ and $\nu$ be probability measures on $E$. Assume 
\begin{enumerate}
\item $\norm{\eta}_{2,\delta}^2 \leq \delta^{-\kappa}$,
\item $\Supp\eta\subset \Ball(0,\delta^{-\eps})$ and $\Supp\nu \subset \Ball(0,\delta^{-\eps})$,
\item for every proper affine subspace $W \subset E$, $\nu(W^{(\delta)}) \leq \delta^{2\kappa}$.
\end{enumerate}
Then for $\xi \in E^*$ with $\delta^{-1+\eps} \leq \norm{\xi} \leq \delta^{-1-\eps}$,
\[
\abs{\widehat{\eta \ff \nu}(\xi)} \leq \delta^{\frac{\kappa}{\dim E + 3} - O(\eps)}.
\]
\end{lem}

We can finally derive Theorem~\ref{thm:fourier}.

\begin{proof}[Proof of Theorem~\ref{thm:fourier}]
\underline{First case: $\eta_1=\dots=\eta_s=\eta$.}\\
For a measure $\eta$ satisfying $\NC(\eps,\kappa,\tau)$ at scale $\delta$, we write $\ess(\eta)$ to denote the essential part of $\eta$, defined as a measure on $E$ satisfying
\begin{enumerate}
\item $\ess(\eta) \leq \eta$ and $\ess(\eta)(E) \geq \eta(E) - 3 \delta^\tau$,
\item $\ess(\eta)$ is supported on $\Ball(0,\delta^{-\eps}) \setminus S_E(\delta^\eps)$,
\item $\ess(\eta)$ satisfies $\ANC(\eps,\kappa)$ at scale $\delta$.
\end{enumerate}

The second and third conditions from $\NC_0$ are invariant under translation, so that if $\mu$ is a measure satisfying $\NC_0(\eps,\kappa,\tau)$ and $\nu$ any measure supported on $\Ball_E(0,\delta^{-\eps})$, then $\mu\pp\nu$ always satisfies $\NC_0(2\eps,\kappa,\tau)$.
Therefore, if $\eta$ and $\eta'$ satisfy $\NC(\eps,\kappa,\tau)$ at scale $\delta$, then $\eta \pp \eta'$ satisfy $\NC(O(\eps),\kappa,\frac{\tau}{2})$ at scale $\delta$, with essential part $\ess(\eta\pp\eta')=\ess(\eta)\pp\ess(\eta')$.
Similarly, $\eta \mm \eta'$ and $\eta \ff \eta'$ satisfy $\NC(O(\eps),\kappa,\frac{\tau}{2})$.
We may therefore define inductively $\eta_0 = \ess(\eta)$, and for $k \geq 0$,
\[
\eta_{k+1} = \ess\bigl( \eta_k^{*3} \mm \eta_k^{*3} \bigr),
\]
to get, for each $k \geq 0$,
\begin{enumerate}
\item $\eta_k(E) \geq 1 - \delta^{\frac{\tau}{O_k(1)}}$,
\item $\eta_k$ is supported on $\Ball(0,\delta^{-O_k(\eps)}) \setminus S_E(\delta^{O_k(\eps)})$,
\item $\eta_k$ satisfies $\ANC(O_k(\eps),\kappa)$ at scale $\delta$.
\end{enumerate}
Note that $\ANC(O_k(\eps),\kappa)$ implies
\[
\norm{\eta_k}_{2,\delta}^2 \leq \delta^{- \dim E + \kappa - O_k(\eps)}.
\]
Set $\kappa'=\frac{\kappa}{3}$.
By Proposition~\ref{pr:l2flat} applied with $\kappa'$ instead of $\kappa$, there exists $\eps_1 = \eps_1(E,\kappa')$
such that, provided $\eps>0$ is small enough, we have for each $0 \leq k \leq \left\lceil \frac{\dim E}{\eps_1}  \right\rceil$, either $\norm{\eta_k}^2_{2,\delta} \leq \delta^{-\kappa'}$ or 
\[\norm{\eta_{k+1}}_{2,\delta}^2 \leq \delta^{\eps_1} \norm{\eta_k}_{2,\delta}^2.\]
Hence there exists $s \leq \left\lceil \frac{\dim E}{\eps_1} \right\rceil$ such that 
\[
\norm{\eta_s}^2_{2,\delta} \leq \delta^{-\kappa'}.
\]
By Lemma~\ref{lm:L2Fourier} applied with $\kappa'$ instead of $\kappa$, for $\xi \in E^*$ with $\delta^{-1 + \eps} \leq \norm{\xi} \leq \delta^{-1 - \eps}$, 
\[
\absbig{\hhat{\eta_s^{*2}}(\xi)} \leq \delta^{\frac{\kappa'}{O(1)} -O(\eps)}
	\leq \delta^{\frac{\kappa'}{O(1)}}
	\leq \delta^{\tau},
\]
provided $\eps$ is chosen sufficiently small.
Now, a first application of Lemma~\ref{lm:eta3mmeta3} to $\mu=\eta_{s-1}^{*3}\mm\eta_{s-1}^{*3}$ with $m=2$ yields
\[
\hhat{\eta_s^{*2}}(\xi)+\delta^{\frac{\tau}{O_s(1)}}
	\geq \hhat{\mu^{*2}}(\xi)
	\geq \abs{\hhat{\eta_{s-1}^{*2\cdot 3}}(\xi)}^{2^2}.
\]
A second application of the same lemma to $\mu_1=\eta_{s-2}^{*3}\mm\eta_{s-2}^{*3}$ with $m=2\cdot 3$ gives
\[
\hhat{\eta_{s-1}^{*2\cdot 3}}(\xi) + \delta^{\frac{\tau}{O_s(1)}}
	\geq \hhat{\mu_1^{*2\cdot 3}}(\xi)
	\geq \abs{\hhat{\eta_{s-2}^{*2\cdot 3^2}}(\xi)}^{2^{2\cdot 3}}
\]
and repeating this process $s$ times, we finally obtain
\[
\absbig{\hhat{\eta_0^{*2 \cdot 3^s}}(\xi)}^{O_s(1)} \leq \absbig{\hhat{\eta_s^{*2}}(\xi)} + \delta^{\frac{\tau}{O_s(1)}}
\leq \delta^\tau + \delta^{\frac{\tau}{O_s(1)}} \leq \delta^{\frac{\tau}{O_s(1)}}.
\]
This allows to conclude:
\[
\absbig{\hhat{\eta^{*2 \cdot 3^s}}(\xi)} \leq \absbig{\hhat{\eta_0^{*2 \cdot 3^s}}(\xi)} + O_s(\delta^\tau) \leq \delta^{\eps \tau},
\]
provided $\eps$ is sufficiently small.
This proves the theorem, with parameter $s(E,\kappa)=2\cdot 3^s$.

\noindent{\underline{General case}}\\
To deduce the general case from the previous one, we follow \cite[Proof of Theorem~B.3]{HLL2021}.
In short, one applies the previous case to the measures
\[
\eta_\lambda = \lambda_1(\eta_1\mm\eta_1) + \dots + \lambda_s(\eta_s\mm\eta_s),
\]
where $\lambda=(\lambda_1,\dots,\lambda_s)$ in $\R_+^s$ is such that $\lambda_1+\dots+\lambda_s\leq 1$.
The Fourier decay for $\eta_1*\dots*\eta_s$ can be deduced from that of $\eta_\lambda*\dots*\eta_\lambda$ for every $\lambda$ using the fact that Fourier coefficients of $\eta_\lambda*\dots*\eta_\lambda$ can be written as polynomials in $\lambda$ whose coefficients are essentially Fourier coefficients of $\eta_1*\dots*\eta_s$.
The reader is referred to \cite{HLL2021} for details.
\end{proof}

We conclude this section by showing that the conclusion of Theorem~\ref{thm:fourier} still holds if the non-concentration assumption is only satisfied for some additive convolution of the measures $\eta_i$, $i=1,\dots,s$.
This will be useful when we study Fourier decay of random walks on linear groups.

\begin{coro}
\label{cor:fourier}
Let $E$ be a normed finite-dimensional semisimple algebra over $\R$.
Given $D\in\N^*$ and $\kappa > 0$, there exists $s = s(E,\kappa) \in \N$ and $\eps = \eps(E,\kappa,D) > 0$ such that for any parameter $\tau \in {(0, \eps \kappa)}$ the following holds for any scale $\delta > 0$ sufficiently small.

If $\eta_i$, $i=1,\dots,s$ are probability measures on $E$ such that each $\eta_i^{\pp D}$ satisfies $\NC(\eps,\kappa,\tau)$ at scale $\delta$,
then for all $\xi \in E^*$ with $\delta^{-1 + \eps} \leq \norm{\xi} \leq \delta^{-1 - \eps}$,
\[
\abs{(\eta_1 * \dotsm * \eta_s)^{\wedge}(\xi)} \leq \delta^{\eps \tau}.
\]
\end{coro}
\begin{proof}
Let $\xi\in E^*$ with $\delta^{-1+\eps}\leq\norm{\xi}\leq\delta^{-1-\eps}$.
Since all the measures $\eta_i^{\pp D}\mm\eta_i^{\pp D}$, $i=1,\dots,s$ satisfy $\NC(\eps,\kappa,\tau)$ at scale $\delta$, Theorem~\ref{thm:fourier} shows that
\[
\abse{\big((\eta_1^{\pp D}\mm\eta_1^{\pp D}) * \dotsm * (\eta_s^{\pp D}\mm\eta_s^{\pp D})\big)^{\wedge}(\xi)} \leq \delta^{\eps \tau}.
\]
Applying \cite[Lemma~B.6]{HLL2021} repeatedly $s$ times, we see that
\begin{align*}
& \abse{\big((\eta_1^{\pp D}\mm\eta_1^{\pp D}) * \dotsm * (\eta_s^{\pp D}\mm\eta_s^{\pp D})\big)^{\wedge}(\xi)}\\
& \qquad\qquad \geq \abse{\big((\eta_1^{\pp D}\mm\eta_1^{\pp D}) * \dotsm * (\eta_{s-1}^{\pp D}\mm\eta_{s-1}^{\pp D})*\eta_s\big)^{\wedge}(\xi)}^{2D}\\
& \qquad\qquad \geq \dots\\
& \qquad\qquad \geq \abs{(\eta_1 * \dotsm * \eta_s)^{\wedge}(\xi)}^{(2D)^s}
\end{align*} 
so that 
\[
\abs{(\eta_1 * \dotsm * \eta_s)^{\wedge}(\xi)} \leq \delta^{\frac{\eps \tau}{(2D)^s}}.
\]
\end{proof}

\section{Non-concentration for random walks on semisimple groups}
\label{sec:nonconcentration}

In this section we consider a probability measure $\mu$ on $\SL_d(\Z)$, and we prove some non-concentration property for the law of the associated random walk, viewed as a measure on the algebra generated by $\mu$.  
Let $\Gamma$ be the group generated by the support of $\mu$ and $G$ be the Zariski closure of $\Gamma$ in $\SL_d(\R)$.
We assume that $G$ is semisimple and Zariski connected.

\bigskip

Let $E$ denote the $\R$-linear span of $G$ in $\Mat_d(\R)$, which is also the subalgebra generated by $G$ in $\Mat_d(\R)$.
Since $G$ is semisimple, one may decompose $E$ into a direct sum of irreducible $G$-modules.
This gives a decomposition of $E$ into minimal left ideals, so that by the fundamental theorem of semisimple rings \cite[\S 117]{VanderWaerden}, $E$ is a semisimple algebra.
Let 
\begin{equation}\label{tesum}
E = E_1 \oplus \dots \oplus E_r
\end{equation}
be the decomposition of $E$ into simple factors, i.e. into minimal two-sided ideals.
For $j = 1,\dotsc,r$, let $\pi_j \colon E \to E_j$ denote the corresponding projections.
Consider the top Lyapunov exponent associated to $\mu$ on each of the factors $E_j$, defined by
\[
\lambda_1(\mu,E_j) = \lim_{n \to +\infty} \frac{1}{n} \int \log \norm{\pi_j(g)} \dd \mu^{*n}(g).
\]
In order to study the law at time $n$ of the random walk, we shall use the rescaling automorphism $\phi_n\colon E\to E$ defined by
\begin{equation}
\label{eq:phin}
\phi_n(g) = \sum_{j=1}^r e^{-n \lambda_1(\mu,E_j)} \pi_j(g).
\end{equation}

Recall that by Furstenberg's theorem \cite{Furstenberg1963} on the positivity of the Lyapunov exponent, one has $\lambda_1(\mu,E_j)\geq 0$ with equality if and only if $\pi_j(G)$ is compact.
After reordering the factors, we may assume that $\lambda_1(\mu,E_j) > 0$ if and only if $j \leq s$ for some integer $s \leq r$. 
Let $E' = E_1 \oplus \dots \oplus E_s$ and $\pi' \colon E \to E'$ the corresponding projection.
Finally, for $n\geq 1$ we define
\[
\mu_{n} = (\pi' \circ \phi_n)_*(\mu^{*n}).
\]
The goal of this section is as follows.

\begin{prop}[Non-concentration]
\label{pr:nc2}
Let $\mu$ be a probability measure on $\SL_d(\Z)$ having a finite exponential moment.
Let $G$ denote the algebraic group generated by $\mu$.
Assume that $G$ is semisimple and Zariski connected, and denote by $E\subset \Mat_d(\R)$ the algebra generated by $G$.
Writing $D = \dim E$, there exists $\kappa = \kappa(\mu) > 0$ such that for any $\eps > 0$ there exists $\tau > 0$ such that
$\mu_{n}^{\pp D}$ satisfies $\NC(\eps,\kappa,\tau)$ at scale $e^{-n}$ in $E'$ for all $n$ sufficiently large.
\end{prop}

The readers can easily convince themselves that $\mu_n$ does not satisfy $\NC(\eps,\kappa,\tau)$, especially the non-concentration condition near singular matrices.
Hence taking an additive convolution power is necessary.

\subsection{Non-concentration near affine subspaces}
In this subsection we show that if $\mu$ is a probability measure on $\SL_d(\Z)$ generating a connected semisimple algebraic group $G$, the law at time $n$ of the random walk associated to $\mu$ is not concentrated near proper affine subspaces of the algebra generated by $\mu$.

We introduce a quasi-norm adapted to the random walk on the algebra $E$ generated by $\mu$.
Given an element $g$ in $E$, we write $g=\sum_{i=1}^r g_i$ according to the direct sum decomposition \eqref{tesum} and set
\[
\qnorm{g} = \max_{1\leq i\leq s}\norm{g_i}^{\frac{1}{\lambda_1(\mu,E_i)}}.
\]
Note that $\abs{g}=0$ if and only if $g$ lies in the sum
\[
E_0:=E_{s+1}\oplus\dots\oplus E_r
\]
of all compact factors.
We denote by $\dtil$ the quasi-distance on $E$ given by $\dtil(x,y)=\qnorm{x-y}$.
For instance, if $W$ is any affine subspace of $E$, we write
\[
\dtil(g,W) = \inf_{w\in W}\qnorm{g-w}.
\]
Our goal is the following proposition.

\begin{prop}[Affine non-concentration on $E$]
\label{anc}
Let $\mu$ be a probability measure on $\SL_d(\Z)$ having a finite exponential moment.
Let $G$ denote the algebraic group generated by $\mu$.
Assume that $G$ is semisimple and Zariski connected, and denote by $E\subset \Mat_d(\R)$ the algebra generated by $G$.

There exists $\kappa=\kappa(\mu)>0$ such that for every $n \geq 0$ and  $\rho \geq e^{-n}$, for every affine hyperplane $W \subset E$ such that $W-W\supset E_0$,
\[
\mu^{*n}\bigl(\setbig{g \in G}{\dtil(g,W) < \rho \min_{j \in J_W} \qnorm{\pi_j(g)}} \bigr) \ll \rho^\kappa
\]
where $J_W = \set{1 \leq j \leq r}{V_j \not\subset W - W}$.
\end{prop}

\begin{remark}
In general, it is not possible to replace the minimum $\min_{j\in J_W}\qnorm{\pi_j(g)}$ by $\qnorm{g}$.
This can be seen for example by taking $G=G_1\times G_1$ and $\mu=\mu_1\otimes\mu_1$; in other words, the random walk is the direct product of two independent copies of a random walk on $G_1$.
By the central limit theorem for random matrix products~\cite[Theorem~5.1, page 121]{BougerolLacroix}, the probability to obtain at time $n$ an element $g=(g_1,g_2)$ such that $\norm{g_1}\leq e^{-\sqrt{n}}\norm{g_2}$ has a positive limit $c$. 
Therefore, for large $n$,
\[
\mu^{*n}\bigl(\setbig{g=(g_1,g_2) \in G}{\norm{g_1} < e^{-\sqrt{n}}\norm{g}} \bigr) \geq \frac{c}{2}.
\]
and taking $W=\{0\}\times \Span_\R(G_1)$, we find
\[
\mu^{*n}\bigl(\setbig{g \in G}{\dtil(g,W) < e^{-\sqrt{n}}\qnorm{g}}) \geq \frac{c}{2}.
\]
\end{remark}


\subsubsection{The case of a simple algebra}
For clarity, we first explain the proof of Proposition~\ref{anc} when the algebra $E$ generated by $G$ is simple.
In that case, the quasi-norm is a norm on $E$ and $\min_{j\in J_W}\qnorm{\pi_j(g)}=\norm{g}^{\frac{1}{\lambda_1}}$.
The key result in the proof is the following proposition, which we shall later apply to the irreducible action of $G\times G$ on $E$.

\begin{prop}
\label{ncirr}
Let $\mu$ be a probability measure on $\SL_d(\Z)$ with a finite exponential moment.
Assume that the algebraic group $G$ generated by $\mu$ is Zariski connected and acts irreducibly on $V=\R^d$.
There exists $\kappa=\kappa(\mu)$ such that for every $v\in V$, and any affine hyperplane $W \subset V$,
\[
\mu^{*n}\bigl(\set{g\in G}{d(gv,W)\leq\rho\norm{gv}}\bigr) \ll \rho^\kappa.
\]
\end{prop}
\begin{proof}
\underline{First step: escape from affine subvarieties.}\\
We claim that there exists $c>0$ such that for every affine map $f$ on $E$ that is not identically zero on $G$,
\[
\mu^{*n}\bigl(\set{g\in G}{f(g)=0}\bigr) \ll e^{-cn}.
\]
Indeed, by \cite[Lemme~8.5]{bq1}, the group $G$ is semisimple.
So the desired inequality is a particular case of \cite[Proposition~3.7]{HS2019}, whose proof is a combination of the spectral gap property modulo prime integers~\cite{SGV} and the Lang-Weil estimates on the number of points on algebraic varieties in finite fields.

\noindent\underline{Second step: a small neighborhood via a Diophantine property.}\\
Let us show that there exist $C,c>0$ such that for every non-zero polynomial map $f$ of degree at most $1$ on $G$,
\[
\mu^{*n}\bigl(\set{g\in G}{\abs{f(g)}\leq e^{-Cn}\norm{f}}\bigr) \ll e^{-cn}.
\]
In the above, a polynomial map $f$ on $G$ is simply the restriction to $G$ of a polynomial map on $E$; it is said to have degree at most $D$ if it is the restriction of a polynomial map on $E$ of degree at most $D$.
We endow the finite-dimensional space of polynomial maps of degree at most $1$ with a fixed norm $\norm{\cdot}$, say $\norm{f}=\sup_{g\in G\cap B_E(1,1)}\abs{f(g)}$.
By the large deviation principle (see Theorem~\ref{thm:LargeD} below), there exists $c>0$ such that for all $n$ large enough, $\mu^{*n}(\{g\in G\ |\ \norm{g}>e^{2n\lambda_1(\mu,V)}\})\leq e^{-cn}$.
Therefore, to prove the desired inequality, it suffices to show that for $C\geq 0$ large enough, the subset $A_n \subset \Mat_d(\Z)$ defined as
\[
A_n=\setbig{g\in \Gamma}{\abs{f(g)}\leq e^{-Cn}\norm{f} \text{ and } \norm{g}\leq e^{2n\lambda_1(\mu,V)}}
\]
is included in $G\cap\ker\psi$ for some affine map $\psi:E\to\R$ not identically zero on $G$.

Suppose for a contradiction that this is not the case.
Letting $k=\dim\R_{\leq 1}[G]$ be the dimension of the space of polynomial maps on $G$ of degree at most $1$, we may choose $g_1,\dots,g_k$ in $A_n$ such that the linear map
\[
\begin{array}{lccc}
L \colon & \R_{\leq 1}[G] & \to & \R^k\\
& \psi & \mapsto & (\psi(g_1),\dots,\psi(g_k))
\end{array}
\]
 is bijective.
Since it has integer coefficients and norm at most $e^{C_0n}$, we get $\norm{L^{-1}}\leq e^{C_1n}$.
In particular, $\norm{f}\leq e^{C_1n}\norm{Lf}=e^{C_1n}\max_{1\leq i\leq k}\abs{f(g_i)} \leq e^{C_1n}e^{-Cn}\norm{f}$, which is the desired contradiction if $C>C_1$.

\noindent\underline{Third step: distance to proper subspaces.}\\
We claim that there exist $C,c>0$ such that for every $v\in V$ and every affine hyperplane $W \subset V$,
\[
\mu^{*n}(\{g\in G\ |\ d(gv,W)\leq e^{-Cn}\norm{v}\}) \ll e^{-cn}.
\]
Indeed, let $\phi_W:V\to\R$ be an affine map such that $\ker \phi_W=W$, and consider the affine map on $G$ given by
\[
f_{v,W}(g)=\frac{\phi_W(gv)}{\norm{\phi_W}}.
\]
Note that $\abs{f_{v,W}(g)} \asymp d(gv,W)$.
Let $B=\Ball_{G}(1,1)$ denote the unit ball centered at the identity in $G$.
Note that $\norm{v}\asymp \norm{f_{v,W}} = \sup_{g\in B} \abs{f_{v,W}(g)}$ within constants independent of $v$ and $W$.
Indeed, otherwise, we may find $v_n$ and $W_n$ such that $\sup_{g\in B} d(gv_n,W_n)\to 0$.
Extracting subsequences if necessary, we may assume that $v_n\to v$ and $W_n\to W$; then for every $g\in B$, $d(gv,W)=\lim d(gv_n,W_n)=0$.
This implies that $G\cdot v\subset W$ and contradicts the assumption that $G$ acts irreducibly on $V$.
The desired inequality therefore follows from the previous step.

\noindent\underline{Fourth step: scaling.}\\
First observe that increasing $C$ slightly, we can assume that for every $v\in V$ and every affine hyperplane $W \subset V$,
\[
\mu^{*n}\bigl(\set{g\in G}{d(gv,W)\leq e^{-Cn}\norm{gv}}\bigr) \ll e^{-cn}.
\]
Indeed, by the large deviation estimate, 
\[ \mu^{*n}\bigl(\set{g\in G}{\norm{gv} \leq e^{2\lambda_1 n}\norm{v} }\bigr) \geq 1-e^{-cn}\]
where $\lambda_1 = \lambda_1(\mu,V)$ is the top Lyapunov exponent of $\mu$.
To conclude, let $\kappa=\frac{c}{C}$, where $C,c>0$ are the constants obtained above.
Choose $m\in\N^*$ such that $\rho=e^{-Cm}$ and write
\begin{align*}
& \mu^{*n}\bigl(\set{g\in G}{d(gv,W)\leq\rho\norm{gv}}\bigr)\\
& \qquad\qquad = \int \mu^{*m}\bigl(\set{g\in G}{d(gg_1v,W)\leq e^{-Cm}\norm{gg_1v}}\bigr) \dd\mu^{*(n-m)}(g_1)\\
& \qquad\qquad \leq e^{-cm} = \rho^\kappa.
\end{align*}
\end{proof}

\begin{proof}[Proof of Proposition~\ref{anc}, case where $E$ is simple]
For $x\in E$, let $L_x:E\to E$ and $R_x:E\to E$ denote the left and right multiplication by $x$, respectively.
Given a probability measure $\mu$ on $G$, we define a probability measure $\tmu$ on $\GL(E)$ by
\[
\tmu = \frac{1}{2}L_*\mu + \frac{1}{2}R_*\mu.
\]
The group generated by $\tmu$ is isomorphic to $G\times G$ and acts irreducibly on $E$.
Moreover, in an appropriate basis, the elements of $\Supp\tmu$ have integer coefficients, so we may apply Proposition~\ref{ncirr} to $\tmu$, with vector $v=1_{E}$, the unit of $E$.
Note that if $\bar{g}$ is a random element distributed according to $\tmu^{*n}$, then $\bar{g}\cdot 1_{E}$ has law $\mu^{*n}$, and therefore we find, uniformly over all affine hyperplanes $W \subset E$,
\[
\mu^{*n}\bigl(\set{g\in G}{d(g,W)\leq\rho\norm{g}}\bigr) \ll \rho^\kappa,
\]
which is exactly the content of Proposition~\ref{anc} in the case where $E$ is simple.
\end{proof}

\subsubsection{General case}
The proof of Proposition~\ref{anc} in the general case follows the same strategy as in the simple case, but the argument becomes slightly more technical, because the norm on $E$ is replaced by a quasi-norm, and $E$ contains proper ideals.

To state the appropriate generalization of Proposition~\ref{ncirr}, we consider a probability measure $\mu$ on $\SL_d(\Z)$ with some finite exponential moment, and let $G$ be the algebraic group generated by $\mu$.
We assume that $G$ is Zariski connected and that the space $V=\R^d$ can be decomposed into a sum of irreducible representations of $G$:
\[
V = V_1 \oplus \dots \oplus V_r.
\]
We denote by $\pi_j\colon V\to V_j$, $j=1,\dots,r$ the corresponding projections.
To define a quasi-norm on $V$, we fix $\alpha = (\alpha_1,\dotsc,\alpha_s)$ an $s$-tuple of positive real numbers, where $s$ is some fixed integer $1\leq s\leq r$, and set
\[
\qnorm{v} = \qnorm{v}_\alpha = \max_{1 \leq j \leq s} \norm{\pi_j(v)}^{\alpha_j}.
\]
For example, $\qnorm{v} = 0$ if and only if $v \in V_0:=V_{s+1} \oplus \dots \oplus V_r$.
The quasi-distance associated to $\qnorm{\mybullet}$ on $V$ is given by
\[
\dtil(v,w) = \dtil_\alpha(v,w) = \qnorm{v - w}_\alpha.
\]
It satisfies a weak form of the triangle inequality:
\[
\forall u,v,w \in V,\quad \dtil(u,w) \ll_\alpha \dtil(u,v) + \dtil(v,w).
\]
Given a subset $W \subset V$, and $v\in V$, we define the distance from $v$ to $W$ by
\[
\dtil(v,W) = \inf_{w \in W} \dtil(v,w).
\]

\begin{remark}
In all our applications, we shall take $s$ so that $V_0 = V_{s+1}\oplus\dots\oplus V_r$ is the sum of all compact factors and
\[
\alpha_j = \frac{1}{\lambda_1(\mu,V_j)} \quad\text{for}\ j=1,\dotsc,s
\]
to obtain a quasi-norm adapted to the random walk associated to $\mu$, same as the one defined in the introduction.
However, the proof works in the more general setting of any choice of $s$ and $\alpha$.
\end{remark}

In the remainder of this subsection, $s$ and $\alpha$ are fixed and the implied constants in all Landau and Vinogradov notations may depend on $d$ and $\alpha$.
\begin{prop}
\label{pr:ncaff2}
Assume that $G$ is Zariski connected and that the linear span of $G$ in $\End(V)$ contains $\pi_j$ for $j = 1, \dotsc, s$.
Then there exists $\kappa = \kappa(\mu,\alpha) > 0$ such that for any $v \in V$ and any affine hyperplane $W \subset V$ with $V_0 \subset W - W$,
\[
\forall n \geq 0,\; \forall \rho \geq e^{-n},\quad \mu^{*n}\bigl(\setbig{g \in G}{\dtil(gv,W) < \rho \min_{j \in J_W} \qnorm{\pi_j(gv)}} \bigr) \ll \rho^\kappa
\]
where $J_W = \set{1 \leq j \leq r}{V_j \not\subset W - W}$.
\end{prop}

\begin{remark}
The requirement that the linear span of $G$ contain $\pi_j$, $j=1,\dots,s$ is here to exclude examples such as $V=V_1\oplus V_1$, with $G$ acting irreducibly on $V_1$.
Indeed, in that case the diagonal subspace $W=\{(v_1,v_1)\ ;\ v_1\in V_1\}$ is stable under $G$, so the proposition cannot hold.
\end{remark}

In the proof, we will use Lemma~\ref{lm:norm2} and Lemma~\ref{lm:qn}, whose proofs will be given right after.

\begin{proof}
\underline{First and second step: spectral gap and Diophantine property.}\\
Arguing exactly as in the proof of Proposition~\ref{ncirr} we obtain that there exist $C,c>0$ such that for every polynomial map $f$ of degree at most $1$ on $G$,
\[
\mu^{*n}\bigl(\set{g\in G}{\abs{f(g)}\leq e^{-Cn}\norm{f}}\bigr) \ll e^{-cn}.
\]
As before, the norm on polynomial maps of degree at most $1$ is defined by $\norm{f}=\sup_{g\in B_G(1,1)}\abs{f(g)}$.
\\
\noindent
\underline{Third step: distance to proper subspaces}.\\
We claim that there exist $C_1,c>0$ such that for every $v\in V$ and every affine hyperplane $W$ such that $W-W\supset V_0$,
\[
\mu^{*n}\bigl(\setbig{g\in G}{\dtil(gv,W)\leq e^{-C_1n}\min_{j\in J_W}\qnorm{\pi_j(v)}}\bigr) \ll e^{-cn}.
\]
To prove this, Lemma~\ref{lm:norm2} below shows that it is enough to show that if $B$ is some large ball in $G$, then
\[
\mu^{*n}\bigl(\setbig{g\in G}{\dtil(gv,W)\leq e^{-C_1n}\sup_{h \in B} \dtil(hv,W)}\bigr) \ll e^{-cn}.
\]
Now let $\phi_W \colon V \to \R$ be an affine map such that $\ker\phi_W = W$, and denote by $l_W$ the linear part of $\phi_W$; by Lemma~\ref{lm:qn} below, the distance to $W$ for the quasi-norm is given by
\[
\forall v\in V,\quad 
 \dtil(v,W) \asymp \min_{i : l_W(u_i) \neq 0} \abse{\frac{\phi_W(v)}{l_W(u_i)}}^{\alpha_{j(i)}},
\]
where $(u_i)_{1\leq i\leq d}$ is an orthonormal basis compatible with the quasi-norm and $u_i\in V_{j(i)}$ for $i=1,\dots,d$.
Therefore, if $g$ satisfies $\dtil(gv,W)\leq e^{-C_1n}\sup_{h \in B} \dtil(hv,W)$, there must exist $i$ such that
\[
\abse{\frac{\phi_W(gv)}{l_W(u_i)}}^{\alpha_{j(i)}} \ll_\alpha e^{-C_1n}\sup_{h \in B}\abse{\frac{\phi_W(hv)}{l_W(u_i)}}^{\alpha_{j(i)}}
\]
whence
\[
\abs{\phi_W(gv)} \ll_\alpha e^{-\frac{C_1n}{\alpha_{j(i)}}}\sup_{h\in B}\abs{\phi_W(hv)}.
\]
If $\frac{C_1}{\alpha_{j(i)}}>C$, the previous step applied to the affine map $f\colon g\mapsto\phi_W(gv)$ shows that the $\mu^{*n}$-measure of such points is bounded above by $e^{-cn}$, so the desired statement is proved.

\noindent\underline{Fourth step: scaling}\\
Note that there are $C_2 = C_2(\mu,\alpha) > 1$ and $c = c(\mu) > 0$ such that for any vector $v \in V$ and any affine hyperplane $W \subset V$ with $V_0 \subset W - W$,
\begin{equation}
\label{eq:ncaff2step2}
\forall n \geq 0,\quad \mu^{*n}\bigl( \setbig{g \in G}{ \dtil(gv,W) \leq e^{-C_2 n}\min_{j \in J_W} \qnorm{\pi_j(gv)}  } \bigr) \ll e^{-cn}.
\end{equation}
This readily follows from the previous step, and from the fact that, by the exponential moment assumption, there are $C_3 = C_3(\mu) > 1$  and $c = c(\mu) > 0$ such that 
\[
\mu^{*n}\bigl( \set{g\in G}{ \norm{g} \geq e^{C_3n} }\bigr) \ll e^{-cn}.
\]
Noting that for any $j = 1,\dotsc,s$, $\qnorm{\pi_j(gv)} \leq \norm{g}^{\alpha_j} \qnorm{\pi_j(v)}$, we obtain \eqref{eq:ncaff2step2} by taking $C_2 = C_1 + (\max_{1\leq j\leq s} \alpha_j) C_3$.

Finally, given $e^{-n}\leq \rho \leq 1$, set $m = \left\lfloor\frac{-\log \rho}{C_2}\right\rfloor$.
Writing $\mu^{*n} = \mu^{*m} * \mu^{*(n-m)}$ and using the fact that \eqref{eq:ncaff2step2} holds uniformly in $v$, we find
\begin{align*}
&\mu^{*n} \bigl(\setbig{g \in G}{\dtil(gv,W) < \rho \min_{j \in J_W} \qnorm{\pi_j(gv)} } \bigr) \\
\qquad & \leq  \int_G  \mu^{*m} \bigl(\setbig{g \in G}{\dtil(ghv,W) \leq e^{-C_2 m} \min_{j \in J_W} \qnorm{\pi_j(ghv)} } \bigr) \dd \mu^{*(n-m)}(h)\\
\qquad & \ll e^{-cm}\\
\qquad & \ll  \rho^{c/C_2}.
\end{align*}
This finishes the proof of Proposition~\ref{pr:ncaff2}.
\end{proof}

We are left to show the two technical lemmas on quasi-norms and distances to hyperplanes that we used in the proof.

\begin{lem}
\label{lm:norm2}
Assume that the linear span of $G$ in $\End(V)$ contains $\pi_j$, for $j = 1, \dotsc, s$.
Then there exists a ball $B\subset G$ such that for any affine hyperplane $W \subset V$ with $V_0 \subset W- W$, and any $v\in V$,
\[
\min_{j \in J_W} \qnorm{\pi_j(v)} \ll \sup_{g \in B} \dtil(gv,W).
\]
\end{lem}
\begin{proof}
In the particular case $r = 1$ (irreducible case), one may assume that the quasi-norm is equal to the Euclidean norm.
So the desired inequality with $B=\Ball_G(1,1)$ has already been proved in the third step of the proof of Proposition~\ref{ncirr}.

For the general case, first observe that by working in the quotient space $V/V_0$, we may assume that $V_0 = \{0\}$, that is, $V = V_1 \oplus \dots \oplus V_s$.
Then assume for a contradiction that for arbitrarily large $R$ and arbitrarily small $c > 0$, there exist $v\in V$ and $W \subset V$ such that
\begin{equation}
\label{eq:dgvW}
\forall g \in \Ball_G(1,R),\quad \dtil(gv,W) \leq c \min_{j \in J_W} \qnorm{\pi_j(v)}.
\end{equation}

Fix $j\in\{1,\dots,s\}$.
Let us construct a basis of $V_j$ such that
\[
\normbig{v_{j,1} \wedge \dotsm \wedge v_{j,\dim V_j}} \gg \norm{\pi_j(v)}^{\dim V_j}
\]
and
\begin{equation}
\label{eq:vj1}
\forall i=1,\dots,\dim V_j,\qquad v_{j,i} \in \Ball_G(1,1) \pi_j(v) - \pi_j(v).
\end{equation}
For that, we proceed iteratively.
Assuming $v_{j,1},\dots,v_{j,i}$ have been constructed, we know from the irreducible case applied in $V_j$ with vector $v=\pi_j(v)$ and subspace $W=\pi_j(v)+\Span(v_{j,1},\dots,v_{j,i})$, that there exists $g$ in $\Ball_G(1,1)$ such that $d(g\pi_j(v)-\pi_j(v),\Span(v_{j,k};k\leq i))\gg\norm{\pi_j(v)}$.
So we set $v_{j,i+1}=g\pi_j(v)-\pi_j(v)$.
By construction, the vectors $v_{j,i}$, $i\geq 1$ are linearly independent, and in the end, we get a basis for $V_j$ with the desired property.
Concatenate these bases to get a basis $(u_1,\dotsc, u_d)$ of $V$, which has the property that 
\begin{equation}
\label{eq:wedgeui}
\norm{u_1 \wedge \dotsm \wedge u_d} \gg \prod_{j=1}^{s} \norm{\pi_j(v)}^{\dim V_j}. 
\end{equation}

Let $W_0 = W - W$ denote the direction of $W$. 
By assumption, for $j=1,\dots,s$, there exist constants $\beta_{j,k} \in \R$ and elements $g_k$ in $G$ such that
\[
\pi_j = \sum_k \beta_{j,k}g_k.
\]
Set $R$ large enough so that for all $k$, $\Ball_G(1,1)g_k  \subset \Ball_G(1,R)$.
Fix $i\in\{1,\dots,d\}$.
From \eqref{eq:vj1} we may write $u_i=g\pi_j(v)-\pi_j(v)$ for some $g\in\Ball_G(1,1)$, so
\[
u_i = \sum_k \beta_{j,k}(gg_k-g_k)v
\]
and \eqref{eq:dgvW} allows us to bound
\[
\dtil(u_i,W_0) \leq 2c \sum_{k} \abs{\beta_{j,k}} \min_{j \in J_W} \qnorm{\pi_j(v)}.
\]
For each $i= 1,\dotsc,d$, let $w_i\in W_0$ be such that
\[
\dtil(u_i,w_i) \ll c \min_{j \in J_W} \qnorm{\pi_j(v)},
\]
where the involved constant depends on the numbers $\beta_{j,k}$.
Using the assumption that $\bigoplus_{j \not\in J_W} V_j \subset W_0$, after adjusting $w_i$, we can moreover ensure that 
\[w_i - u_i \in \bigoplus_{j \in J_W} V_j.\]
We can bound
\begin{equation}
\label{eq:wedgewiui}
\normbig{w_1 \wedge \dotsm \wedge w_d  -u_1 \wedge \dotsm \wedge u_{d}}
\ll c^{\min_j\frac{1}{\alpha_j}}\prod_{j=1}^s\norm{\pi_j(v)}^{\dim V_j}.
\end{equation}
Indeed, developing the first wedge product using $w_i = u_i + (w_i - u_i)$ and then decomposing each vector along $V_1 \oplus \dots \oplus V_s$ and further developing the sum, we can express 
$w_1 \wedge \dotsm \wedge w_d  -u_1 \wedge \dotsm \wedge u_{d}$ as a sum of wedge products of $d$ vectors of the following types
\begin{enumerate}
\item (first type) $\pi_j(w_i-u_i)$ with $j \in J_W$, or
\item (second type) $\pi_j(u_i)$ with $1 \leq j \leq s$.
\end{enumerate} 
In each wedge product, the first type appears at least once and the product is zero unless $\pi_j$ appears exactly $\dim V_j$ times. 
We can bound vectors of the first type by 
\[\norm{\pi_j(w_i-u_i)} \leq \dtil(w_i,u_i)^{\frac{1}{\alpha_j}}  \ll c^{\frac{1}{\alpha_j}}  \norm{\pi_j(v)}\]
and vectors of the second type by
\[\norm{\pi_j(u_i)} \ll \norm{\pi_j(v)}.\]
This proves \eqref{eq:wedgewiui}.

To conclude, we choose $c$ be to small enough so that \eqref{eq:wedgewiui} combined with~\eqref{eq:wedgeui} implies $w_1 \wedge \dotsm \wedge w_d \neq 0$ contradicting the condition that $W_0$ is a proper linear subspace of $V$.
\end{proof}

The second lemma is an elementary computation using the definition of the quasi-norm.
It is instructive to convince oneself with a picture that the lemma holds when the quasi-norm is simply the euclidean norm on $\R^d$.

\begin{lem}
\label{lm:qn}
Let $(u_i)_{1\leq i \leq d}$ be a union of orthonormal bases of each of the $V_j$, $j=1,\dots,r$.
For $i = 1,\dotsc,d$, denote by $j(i)$ the unique integer such that $u_i \in V_{j(i)}$.

Let $v \in V$ and let $W \subset V$ be an affine hyperplane with with $V_0\subset W - W$.
Let $\phi_W \colon V \to \R$ be an affine map such that
\[W = \set{v \in V}{ \phi_W(v) = 0}.\]
Let $l_W \colon V \to \R$ denote the linear part of $\phi_W$.
We have for any $v \in V$,
\[ \dtil(v,W) \asymp \min_{i : l_W(u_i) \neq 0} \abse{\frac{\phi_W(v)}{l_W(u_i)}}^{\alpha_{j(i)}}.\]
\end{lem}
\begin{proof}
Note that $l_W(u_i) \neq 0$ implies that $j(i) \in J_W$ and $J_W \subset \{1,\dotsc,s\}$ because $V_0 \subset W - W$.
It follows that $\alpha_{j(i)}$ is defined and positive.

For any $i \in \{1,\dotsc, d\}$ with $l_W(u_i) \neq 0$, we have $v - \frac{\phi_W(v)}{l_W(u_i)}u_i \in W$.
Hence 
\[\dtil(v,W) \leq \abse{\frac{\phi_W(v)}{l_W(u_i)}}^{\alpha_{j(i)}}.\]
Let  $u \in V$ be such that $v - u \in W$. 
Write $u = \sum_{i = 1}^d x_i u_i$. 
Then
\[\phi_W(v) = \phi_W(v-u) + l_W(u) = \sum_{i = 1}^d x_i l_W(u_i).\]
It follows that there exists $i$ with $l_W(u_i) \neq 0$ and such that
\[\abs{x_i} \geq \frac{1}{d} \abse{\frac{\phi_W(v)}{l_W(u_i)}}.\]
This allows to conclude since
\[\dtil(v,v - u) = \qnorm{u} \geq \norm{\pi_{j(i)}(u)}^{\alpha_{j(i)}} \geq \abs{x_i}^{\alpha_{j(i)}}.\] 
\end{proof}

%
%
To conclude, we explain how to obtain Proposition~\ref{anc} from Proposition~\ref{pr:ncaff2}.
The argument is essentially the same as the one used in the particular case where $E$ is simple.

\begin{proof}[Proof of Proposition~\ref{anc}, general case]
For $x\in E$, let $L_x:E\to E$ and $R_x:E\to E$ denote the left and right multiplication by $x$, respectively.
Then, define
\[
\begin{array}{rccc}
L \colon &  E & \to & \End E\\
& x & \mapsto & L_x
\end{array}
\quad\mbox{and}\quad
\begin{array}{rccc}
R \colon &  E & \to & \End E\\
& x & \mapsto & R_x
\end{array}
\]
Given a probability measure $\mu$ on $G$, we define a probability measure $\tmu$ on $\GL(E)$ by
\[
\tmu = \frac{1}{2}L_*\mu + \frac{1}{2}R_*\mu.
\]
The group $\bar{G}$ generated by $\tmu$ is isomorphic to $G\times G$ and the decomposition of $E$ into irreducible $\bar{G}$-submodules is simply the decomposition into simple ideals $E=\oplus_j E_j$.
By definition, the algebra generated by $G$ contains the unit $1_{E_j}$ of $E_j$ for each $j$.
It follows that the linear span of $\bar{G}$ contains all projections $\pi_j \colon E\to E_j$.
Moreover, in an appropriate basis, the elements of $\Supp\tmu$ have integer coefficients, so we may apply Proposition~\ref{pr:ncaff2} to $\tmu$, with vector $v=1_{E}$.
Note that if $\bar{g}$ is a random element distributed according to $\tmu^{*n}$, then $\bar{g}\cdot 1_{E}$ has law $\mu^{*n}$, and therefore we obtain $\kappa >0$ such that uniformly over all affine hyperplanes $W \subset E$ with $W-W\supset E_0$,
\[
\forall n \geq 0,\; \forall \rho \geq e^{-n},\quad \mu^{*n}\bigl(\setbig{g\in G}{\dtil(g,W) \leq \rho\min_{j\in J_W}\qnorm{\pi_j(g)} }) \ll \rho^\kappa.
\]
\end{proof}

\subsection{Non-concentration at singular matrices}
As in the previous paragraph, $\mu$ denotes a probability measure on $\SL_d(\Z)$.
We assume that the algebraic group $G$ generated by $\mu$ is semisimple and connected, and let $E$ be the algebra generated by $G$ in $\Mat_d(\R)$.
Recall that for $x\in E$, we defined $\det_E(x)$ to be the determinant of the map $E \to E$, $y \mapsto xy$.
Note that $\det_E$ is a homogeneous polynomial function on $E$ of degree equal to $D=\dim E$.
Recall also that
\[
\mu_n = (\pi'\circ\phi_n)_*(\mu^{*n}),
\]
where $\pi' \colon E \to E'$ is the projection to the direct sum $E'=  E_1 \oplus \dots \oplus E_s$ of all simple ideals with non-zero Lyapunov exponent, and $\phi_n \colon E \to E$ is the scaling map defined in \eqref{eq:phin}.
As before, we write $\pi_j:E\to E_j$, $j=1,\dots,s$ for the projection to the simple factors.

\begin{lem}
\label{lm:detE'}
Given $\omega > 0$ there exists $c = c(\mu,\omega) > 0$ such that the following holds.
\[\forall n \geq 0,\, \forall y \in E',\quad \mu_n^{\pp D}\bigl(\setbig{x \in E'}{\abs{\det\nolimits_{E'}(x-y)} \leq e^{-\omega n}} \bigr) \ll e^{-c n}.\]
\end{lem}
Note that for all $x$ in $E'$,
\(
\det\nolimits_{E'}(x) = \prod_{j= 1}^s \det\nolimits_{E_j}(\pi_j(x))\)
and hence for every $n \geq 0$, and every $x\in E$,
\[
\det\nolimits_{E'}(\pi'\circ \phi_n(x)) = \prod_{j= 1}^s e^{- (\dim E_j) \lambda_1(\mu,E_j) n} \det\nolimits_{E_j}(\pi_j(x)).
\]
This immediately reduces the proof of Lemma~\ref{lm:detE'} to the following.
\begin{lem}
\label{lm:detEj}
Given $\omega > 0$ there exists $c = c(\mu,\omega) > 0$ such that the following holds for every $j = 1,\dotsc,s$, all $n \geq 0$ and all $y \in E_j$,
\[(\mu^{*n})^{\pp D}\bigl(\setbig{x \in E}{\abs{\det\nolimits_{E_j}(\pi_j(x) - y)} \leq e^{(\dim E_j) \lambda_1(\mu,E_j) n -\omega n}} \bigr) \ll e^{-c n}.\]
\end{lem}
The idea is to apply \cite[Proposition 3.2]{HS2019}, where the case where $E$ is simple was treated.
However, upon projecting to a simple factor, the random walk might no longer be defined with integer coefficients: simple factors of $E$ are only defined over a number field.
So we cannot apply \cite[Proposition 3.2]{HS2019} as it is stated. 
Nevertheless, we can remark that, in the proof of \cite[Proposition 3.2]{HS2019}, \cite[Lemma 3.13]{HS2019} holds more generally for the projected random walk from $E$ to each $E_j$ and then the rest of the proof of \cite[Proposition 3.2]{HS2019} for a projected random walk is identical.

Here is the detailed proof.
We need two ingredients from \cite{HS2019}.
For a probability measure $\mu$ on a semisimple Lie group $G$ and a finite-dimensional linear representation $(\rho,V)$ of $G$ over $\R$, recall that
\[\lambda_1(\mu,V) = \lim_{n \to +\infty} \frac{1}{n} \int_G \log \norm{\rho(g)} \dd \mu^{*n}(g)\]
denotes the top Lyapunov exponent associated to the random walk induced on $V$.
By semisimplicity $V$ is a sum of irreducible sub-representations.
The sum of  irreducible sub-representations of same top Lyapunov exponent is a sum of isotypical components.
For $\lambda \in \R$, we will denote by $p_\lambda \colon V \to V$ the $G$-equivariant projection onto 
\[
\sum_{\substack{V' \subset V, \text{ irreducible}\\ \lambda_1(\mu ,V') \geq \lambda}} V'.
\]
We also write $\R[G]_{\leq D}$ for the set of polynomial maps of degree at most $D$ on $G$, i.e. restrictions to $G$ of polynomial maps of degree at most $D$ on $E$.
We fix a norm on $\R[G]_{\leq D}$, for instance $\norm{f}=\sup_{g\in\Ball_G(1,1)}\abs{f(g)}$.
The following is \cite[Proposition 3.17]{HS2019}.
\begin{prop}
\label{pr:Prop317}
Let $\mu$ be a probability measure on $\SL_d(\Z)$ having a finite exponential moment. 
Let $G$ denote the Zariski closure of the subgroup generated by $\Supp(\mu)$ in $\SL_d(\R)$.
Assume that $G$ is semisimple and Zariski connected.
Given $D \geq 1$, $\lambda \geq 0$, and $\omega > 0$, there is $c = c(\mu,D,\lambda,\omega) > 0$ such that the following holds for every $f \in \R[G]_{\leq D}$.
\[\forall n \geq 0,\quad \mu^{*n}\bigl( \set{g \in G}{ \abs{f(g)} \leq e^{(\lambda - \omega)n} \norm{p_\lambda(f)} } \bigr) \ll e^{-cn}.\]
Here $p_\lambda \colon \R[G]_{\leq D} \to \R[G]_{\leq D}$ is defined as above.
\end{prop}

The following is \cite[Lemma 3.18]{HS2019}. For $k \geq 1$ and a measure $\mu$ on $G$, $\mu^{\otimes k} = \mu \otimes \dotsm \otimes \mu$ denotes the product measure on $G^k = G \times \dotsm \times G$.
Again, $\R[G^k]_{\leq D}$ denotes the space of restrictions to $G^k$ of polynomial functions of degree at most $D$ on $E^k$.

\begin{lem}
\label{lm:Lemma318}
Let $V$ be a Euclidean space.
Let $\mu$ be a Borel probability measure  on $\SL(V)$ having a finite exponential moment.
Let $G$ denote the Zariski closure of the subgroup generated by $\Supp(\mu)$ in $\SL_d(\R)$.
Assume that $G$ is Zariski connected, is not compact and acts irreducibly on $V$.

Let $E$ denote the $\R$-span of $G$ in $\End(V)$ and $k\geq\dim E$ an integer.
Let $D\geq 1$ an integer and $f \in \R[E]_{\leq D}$ be such that its homogeneous part $f_D$ of degree $D$ does not vanish on $E$.
Define $F \in \R[G^k]_{\leq D}$ to be the polynomial function
\[
\forall (x_1,\dotsc,x_k) \in G^k,\quad F(x_1,\dotsc,x_k) = f(x_1 + \dots + x_k).
\]
Then we have
\[
p_{D\lambda_1(\mu,V)}(F) \neq 0
\]
where $p_{D\lambda_1(\mu,V)} \colon  \R[G^k]_{\leq D} \to  \R[G^k]_{\leq D}$ denotes the projection to the sum of irreducible $G^k$-subrepresentations $M \subset  \R[G^k]_{\leq D}$ with $\lambda_1(\mu^{\otimes k},M) \geq D\lambda_1(\mu,V)$.
\end{lem}

\begin{remark}
The conclusion of the above lemma can be improved to
\[
\norm{p_{D\lambda_1(\mu,V)}(F)}_{\R[G^k]_{\leq D}} \gg_{\mu,D,k} \norm{f_D}_{\R[E]_{\leq D}}.
\]
Indeed, it is enough to check it when $f=f_D$ is homogeneous,
and then one may assume $\norm{f_D}_{\R[E]_{\leq D}}=1$.
The left-hand side is a positive continuous function of $f_D$, so it admits a uniform positive lower bound on the unit sphere $\norm{f_D}_{\R[E]_{\leq D}}=1$.
This shows the desired lower bound.
\end{remark}

\begin{proof}[Proof of Lemma \ref{lm:detEj}]
Fix $j \in\{ 1,\dotsc,s\}$. 
Remember that $E_j$ is a simple algebra over $\R$. 
Using Wedderburn's structure theorem, we can find a real vector space $V_j$ and an irreducible faithful linear representation $E_j \to \End(V_j)$.
It is easy to see that $\lambda_1(\mu,E_j) = \lambda_1\bigl({\pi_j}_*\mu, V_j\bigr)$.
The Zariski closure of the subgroup generated by $\Supp({\pi_j}_*\mu)$ is precisely $\pi_j(G)$. It spans $E_j$, is Zariski connected, acts irreducibly on $V_j$ and is not compact.  
Thus, we may apply Lemma~\ref{lm:Lemma318} to ${\pi_j}_*\mu$ with $D=D_j$ and $k=D$.

Let $y \in E_j$ and consider the polynomial function $f \in \R[E_j]$, $f(x) = \det_{E_j}(x-y)$.
The degree of $f$ is $D_j = \dim E_j$ and its degree $D_j$ homogeneous part is $\det_{E_j}$.
Recall $D = \dim E$. 
Consider $F \in \R[\pi_j(G)^D]_{\leq D_j}$ defined as
\[\forall x_1,\dotsc,x_D \in \pi_j(G),\quad F(x_1,\dotsc,x_D) = f(x_1 + \dots + x_D).\]
By Lemma~\ref{lm:Lemma318} and the remark that follows it
\[
\normbig{p_{D_j\lambda_1(\mu,E_j)}(F)}_{\R[\pi_j(G)^D]_{\leq D_j}} \gg \norm{\det\nolimits_{E_j}}_{\R[E_j]_{\leq D_j}} \gg_E 1.
\]

The linear map $\Theta_j \colon \R[\pi_j(G)^D] \to \R[G^D]$ obtained by precomposing $(\pi_j, \dotsc, \pi_j)$ is injective and sends irreducible $\pi_j(G)^D$-subrepresentations to irreducible $G^D$-subrepresentations.
Moreover, for any irreducible $\pi_j(G)^D$-subrepresentation $M \subset \R[\pi_j(G)^D]$, we have 
\[\lambda_1\bigl(({\pi_j}_*\mu)^{\otimes D},M\bigr) = \lambda_1\bigl(\mu^{\otimes D}, \Theta_j(M) \bigr).\]
It follows that 
\[\normbig{p_{D_j\lambda_1(\mu,E_j)}(F\circ(\pi_j, \dotsc, \pi_j))}_{\R[G^D]_{\leq D_j}} \gg_E 1.\]
Then we obtain Lemma~\ref{lm:detEj} by applying Proposition~\ref{pr:Prop317} to the measure $\mu^{\otimes D}$ and the polynomial function $F\circ(\pi_j, \dotsc, \pi_j) \in \R[G^D]_{\leq D_j}$.
\end{proof}


\subsection{Proof of Proposition~\ref{pr:nc2}}

In order to obtain the required non-concentration properties for the measure $\mu_n$, we shall use the basic large deviation estimates for matrix products that have already been used in the proof of Proposition~\ref{pr:ncaff2}.
The statement below is taken from Boyer~{\cite[Theorem A.5]{Boyer}, which generalizes previous results of Le Page~{\cite{LePage}} and Bougerol~{\cite[Theorem V.6.2]{BougerolLacroix}}.

\begin{thm}[Large deviation estimates]
\label{thm:LargeD}
Let $\mu$ be a Borel probability measure on $\GL_d(\R)$ having a finite exponential moment. 
For any $\omega > 0$, there is $c = c(\mu,\omega) > 0$, such that  the following holds.
\begin{enumerate}
\item \label{it:LargeDn} For all $n \geq 1$,
\[\mu^{*n} \Bigl(\setBig{ g \in \Gamma}{\abse{\frac{1}{n}\log \norm{g} - \lambda_1(\mu,\R^d)} \geq \omega} \Bigr)  \ll_\omega e^{-cn}.\]
\item \label{it:LargeDnc}  Assume further that the group generated by $\Supp(\mu)$ acts irreducibly on $\R^d$. For all $n \geq 1$ and all $v \in \R^d \setminus \{0\}$,
\[\mu^{*n} \Bigl(\setBig{ g \in \Gamma}{\abse{\frac{1}{n}\log \frac{\norm{gv}}{\norm{v}} - \lambda_1(\mu,\R^d)} \geq \omega} \Bigr) \ll_\omega e^{-cn}.\]
\end{enumerate}
\end{thm}
To prove Proposition~\ref{pr:nc2}, we shall only need the first item; the second item will be used later in Section~\ref{sec:granulation}.

\begin{proof}[Proof of Proposition~\ref{pr:nc2}]
Note that condition $\NC(\eps,\kappa,\tau)$ was defined for algebras endowed with a norm, and not with a quasi-norm.
However, for some constants $\alpha,\beta>0$, we have, for every $v\in E'$,
\[
\left\{ \begin{array}{ll}
\norm{v} \leq {\qnorm{v}}^{\,\alpha} & \text{ if }\norm{v}\geq 1\\
\norm{v} \leq {\qnorm{v}}^{\,\beta} & \text{ if }\norm{v} <1.
\end{array}
\right.
\]
So if some measure satisfies condition $\NC(\eps,\kappa,\tau)$ for the quasi-norm $\qnorm{\mybullet}$ on $E'$, then it satisfies $\NC(\alpha\eps,\frac{\kappa}{\beta},\tau)$ for the usual norm $\norm{\cdot}$ on $E'$.
It is therefore sufficient to check the non-concentration properties of $\mu_n$ for the quasi-distance $\dtil$.

For that, let $\eps>0$ be some small parameter.
By Theorem~\ref{thm:LargeD}\ref{it:LargeDn} applied to each ${\pi_j}_*\mu$, there exists $\tau = \tau(\mu,\eps) > 0$ such that 
\[
\mu_n\bigl( \setbig{g \in E'}{\forall j = 1,\dotsc,s,\ \qnorm{\pi_j(g)}\geq e^{-\eps n}}\bigr) \geq 1 - e^{-\tau n}.
\]
Let $\nu_0$ be the restriction of $\mu_n$ to such $g$ and write
\[
\mu_n = \nu_0 + \nu_1
\]
so that $\nu_1(E)\leq e^{-\tau n}$.
By Proposition~\ref{anc}, there exists $\kappa = \kappa(\mu) > 0$ such that for any affine hyperplane $W \subset E$ with $E_0 \subset W - W$,
\[
\forall \rho \geq e^{-n},\quad \mu^{*n}\bigl( \setbig{g \in G}{ \dtil(g,W) < \rho \min_{j \in J_W}\qnorm{\pi_j(g)}} \bigr) \ll \rho^\kappa.
\]
By definition of $\phi_n$ and of the quasi-norm $\abs{\cdot}$ on $E$, we have 
\(
\abs{\phi_n(g)} = e^{-n} \abs{g}
\)
for every $g\in G$ and therefore, for every affine hyperplane $W \subset E'$,
\[
\forall \rho \geq e^{-n},\quad \mu_n\bigl( \setbig{g \in E'}{ \dtil(g,W) < \rho \min_{j \in J_W} \qnorm{\pi_j(g)}} \bigr) \ll \rho^\kappa.
\]
By definition of $\nu_0$, this implies
\begin{equation}\label{subs}
\forall \rho \geq e^{-n},\quad \nu_0\bigl( \setbig{g \in E'}{ \dtil(g,W) < \rho e^{-\eps n}} \bigr) \ll \rho^\kappa
\end{equation}
and this inequality is still valid for any convolution $\nu_0\pp\eta$, where $\eta$ is a finite measure with $\eta(E)\leq 1$.
On the other hand, Lemma~\ref{lm:detE'} shows that for some $\tau_1>0$,
\begin{equation}\label{dt}
\forall y \in E',\quad \mu_n^{\pp D}\bigl(\setbig{x \in E'}{\abs{\det\nolimits_{E'}(x-y)} \leq e^{-\eps n}} \bigr) \ll e^{-\tau_1 n}.
\end{equation}
Let $\eta_0$ be the restriction of $\nu_0\pp\mu_n^{\pp(D-1)}$ to $\Ball_{E'}(0,e^{2\eps n})$, and write
\[
\mu_n^{\pp D} = \eta_0 + \eta_1.
\]
By Theorem~\ref{thm:LargeD}\ref{it:LargeDn}, we have $\eta_1(E')\leq e^{\tau n}$ for some $\tau_2=\tau_2(\eps)>0$, and by equations \eqref{subs} and \eqref{dt}, the measure $\eta_0$ satisfies $\NC_0(2\eps,\frac{\kappa}{2},\tau)$ with $\tau=\min(\tau_1,\tau_2)$.
\end{proof}

\section{Fourier spectrum of the random walk}
\label{sec:fourier}

Let $\mu$ be a probability measure on $\GL_d(\Z)$.
Denote by $\Gamma \subset \GL_d(\Z)$ the subgroup generated by $\Supp(\mu)$ and $G \subset \GL_d(\R)$ the Zariski closure of $\Gamma$ in $\GL_d(\R)$.
Under the assumption that $\mu$ has a finite exponential moment and that $G$ is semisimple, we want to show some Fourier decay property for the measure $\mu^{*n}$ on the algebra $E$ generated by $G$.

\subsection{Connected case}
The result we need about Fourier decay for random walks is particularly transparent and easy to prove when the algebraic group $G$ generated by $\mu$ is Zariski connected.
So we first explain this particular case.
Recall that $\phi_n\colon E\to E$ is the rescaling automorphism given by \eqref{eq:phin}, that $\pi'\colon E\to E'$ denotes the projection to the direct sum of all non-compact factors in $E$, and that for any integer $n\geq 1$, we let
\[
\mu_n = (\pi'\circ\phi_n)_*(\mu^{*n})
\]
be the image of $\mu^{*n}$ after rescaling and projection to $E'$.
The proof of the Fourier decay for $\mu_n$ will be a consequence of the results of Section~\ref{sec:sumprod} for multiplicative convolutions on semisimple algebras, and of the multiplicative structure of $\mu_n$ simply expressed as
\[
\forall m,n,\quad \mu_{n+m} = \mu_m*\mu_n.
\]
We will denote by ${E'}^*$ the space of linear forms on $E'$ over the real numbers.

\begin{thm}[Fourier decay for random walks in $E'$]
\label{thm:decayconnected}
Assume that $G$ is semisimple and Zariski connected, and that $\mu$ has a finite exponential moment.
Then there exists $\alpha_0 = \alpha_0(\mu) > 0$ such that for every $\alpha_1 \in (0,\alpha_0)$, there exists $c_0 = c_0(\mu, \alpha_1) > 0$ such that for all $n$ sufficiently large, for all $\xi \in {E'}^*$ with
\[
e^{\alpha_1 n} \leq \norm{\xi} \leq e^{\alpha_0 n}
\]
the following estimate on the Fourier transform of $\mu^{*n}$ holds:
\[
\abse{\widehat{\mu_n}(\xi)} \leq e^{-c_0n}.
\]
\end{thm}

We let $E'$ act on ${E'}^*$ on the right by
\[\forall \xi \in {E'}^*,\; \forall x,y \in E',\quad (\xi \cdot x)(y) = \xi(xy).\]
Moreover, we let $E$ act on ${E'}^*$ via $\pi'$.

\begin{proof}
Let $D=\dim E'$, $\eps=\eps(E',\kappa,D)$ and $s=s(E',\kappa)$ be the quantities given by Corollary~\ref{cor:fourier}.
By Proposition~\ref{pr:nc2}, given $\alpha_1\in(0,1)$, there exists $\kappa>0$ such that for any $\eps>0$, there exists $\tau>0$ such that $\mu_n^{\pp D}$ satisfies $\NC(\frac{\alpha_1\eps}{2},\kappa,\tau)$ at scale $e^{-n}$ in $E'$ for all $n$ sufficiently large.
This formally implies that $\mu_n^{\pp D}$ satisfies $\NC(\eps,\kappa,\tau)$ at all scales $\delta\in[e^{-n},e^{-\frac{\alpha_1n}{2}}]$.

Without loss of generality, we may of course assume that $\tau\in(0,\eps\kappa)$.
Let $\xi\in {E'}^*$ be such that $e^{\frac{\alpha_1 n}{2}}\leq\norm{\xi}\leq e^{n}$.
Taking $\delta=\norm{\xi}^{-1}$, we have $\delta\in [e^{-n},e^{-\frac{\alpha_1 n}{2}}]$ so $\mu_n^{\pp D}$ satisfies $\NC(\eps,\kappa,\tau)$ at scale $\delta$.
Therefore, Corollary~\ref{cor:fourier} shows that $\mu_{sn}=\mu_n*\dots*\mu_n$ satisfies
\[
\abs{\widehat{\mu_{sn}}(\xi)} \leq e^{-\eps\tau n}.
\]
This shows the desired property if $n\in s\Z$.
In general, take $\alpha_0=\frac{1}{4s}$.
For $n$ large and $\xi\in{E'}^*$ such that $e^{\alpha_1n}\leq\norm{\xi}\leq e^{\alpha_0n}$, write $n=sm+r$, with $0\leq r<s$ and
\[
\widehat{\mu_n}(\xi) = \int_G \widehat{\mu_{sm}}(\xi \cdot x)\dd\mu_r(x).
\]
Then, observe from the exponential moment assumption that outside of a set of $\mu_r$-measure at most $e^{-cn}$, one has $e^{-\frac{\alpha_1 n}{2}}\norm{\xi} \leq \norm{\xi \cdot x}\leq e^{\frac{n}{2s}}\norm{\xi}$ and so
\[
e^{\frac{\alpha_1m}{2}} \leq e^{\frac{\alpha_1n}{2}}\leq\norm{\xi \cdot x}\leq e^{\frac{n}{2s}}e^{\frac{n}{4s}} \leq e^{m}.
\]
For such $\xi \cdot x$, we may bound
\[
\abse{\widehat{\mu_{sm}}(\xi \cdot x)} \leq e^{-\eps\tau m} \leq e^{-\frac{\eps\tau n}{2s}}
\]
whence
\[
\abs{\widehat{\mu_n}(\xi)} \leq e^{-\frac{\eps\tau n}{2s}} + e^{-cn} \leq e^{-c_0n}
\]
with $c_0=\min(\frac{c}{2},\frac{\eps\tau}{4s})$.
\end{proof}

\subsection{Disconnected case}
As before, $\mu$ denotes a probability measure on $\GL_d(\Z)$, and $G$ the algebraic group generated by $\mu$.
We still assume that $\mu$ has a finite exponential moment and that $G$ is semisimple but no longer that it is Zariski connected.
The identity component $G^\circ$ is then a finite index subgroup in $G$.
We now write $\tE$ for the subalgebra generated by $G$ in $\Mat_d(\R)$.
As before, we decompose
\begin{equation*}
\tE = \tE_1 \oplus \dots \oplus \tE_r
\end{equation*}
into simple ideals.
The rescaling automorphism $\phi_n\colon \tE\to \tE$ is now defined by
\begin{equation}
\label{eq:phin2}
\phi_n(g) = \sum_{j=1}^r e^{-n \lambda_1(\mu,\tE_j)} \pi_j(g)
\end{equation}
where $\lambda_1(\mu,\tE_j)$ denotes the top Lyapunov exponent associated to $\mu$ on each of the factors $\tE_j$.
Also, we assume that $\lambda_1(\mu,\tE_j) = 0$ if and only if $j > s$ and denote by $\pi' \colon \tE \to \tE'=\tE_1\oplus\dots\oplus\tE_s$ the projection to the non-compact factors.

\begin{example}
When $G$ is not Zariski connected, we shall write $\tE$ for the algebra generated by $G$, and let $E$ denote the algebra generated by the identity component $G^\circ$ of $G$.
Let $a_0=\begin{pmatrix}1 & 1\\ 0 & 1\end{pmatrix}$, $a_1=\begin{pmatrix} 1 & 0\\ 1 & 1\end{pmatrix}$ and $w=\begin{pmatrix}0 & 1\\-1 & 0\end{pmatrix}$.
Then define by blocks
\( A_0=\begin{pmatrix} w & 0\\
			0  & a_0
\end{pmatrix}
\)
and
\(
A_1=\begin{pmatrix} w & 0\\
			0  & a_1
\end{pmatrix}
\)
in $\SL_4(\Z)$
and set
\[
\mu = \frac{1}{4}(\delta_{A_0}+\delta_{A_1}+\delta_{A_0^{-1}}+\delta_{A_1^{-1}}).
\]
One has $G\simeq (\Z/4\Z)\times\SL_2(\R)$ and $\tE\simeq\C\times M_2(\R)$.
On the other hand, the algebra generated by $G^\circ$ is $E\simeq\R\times M_2(\R)$ if one identifies $\C\simeq\{\begin{pmatrix}a & b\\-b & a\end{pmatrix}\ ;\ a,b\in\R\}$ and $\R\simeq\R 1$.
The law $\mu^{*n}$ of the random walk at time $n$ is supported by $E$ if $n$ is even, and by $A_0E\simeq i\R\times M_2(\R)$ if $n$ is odd.
It is always supported on a proper subspace of $\tE$.
\end{example} 

To overcome this issue, we shall use the algebra $E\subset\tE$ generated by the identity component $G^\circ$ in $G$.
The group $G^\circ$ has finite index in $G$ and we let
\[
F= G/G^\circ.
\]
With a slight abuse of notation, we identify $F$ with a set of representatives in $G$ and write $G$ as a disjoint union
\[
G = \bigsqcup_{\gamma\in F} \gamma G^\circ.
\]
Any measure $\nu$ on $G$ can then be decomposed uniquely in the form
\[
\nu = \sum_{\gamma\in F} \gamma_*\nu_\gamma
\]
where each $\nu_\gamma$ is a measure on $E$.
Finally, we let $E'=\pi'(E)$, and for $n\geq 1$ and $\gamma\in F$,
\[
\mu_{n,\gamma} = (\pi' \circ \phi_n)_*[(\mu^{*n})_\gamma].
\]
Fourier decay for (integer coefficient) random walks on non-connected semisimple groups can be stated as follows.

\begin{thm}[Fourier decay for random walks in $E'$]
\label{thm:decaymun}
Let $\mu$, $G$, $G^\circ$ and $F$ be as above.
Then there exists $\alpha_0 = \alpha_0(\mu) > 0$ such that for every $\alpha_1 \in (0,\alpha_0)$, there exists $c_0 = c_0(\mu, \alpha_1) > 0$ such that for all $n$ sufficiently large, all $\gamma\in F$ and $\xi \in {E'}^*$ with
\[
e^{\alpha_1 n} \leq \norm{\xi} \leq e^{\alpha_0 n}
\]
the following estimate on the Fourier transform of $\mu^{*n}$ holds:
\[
\abse{\widehat{\mu_{n,\gamma}}(\xi)}
\leq e^{-c_0 n}.
\]
\end{thm}

One can show that the above theorem is still valid under the assumption that the measure $\mu$ is supported on the group $\GL_d(\overline{\Q})$ of matrices with algebraic coefficients.
It seems a difficult problem to prove the same statement without any such assumption on the support of $\mu$.

\subsection{Induced random walk on the identity component}
In order to prove Theorem~\ref{thm:decaymun}, we shall use the induced random walk on $G^\circ$, whose definition is given below.
Since by definition $G^\circ$ is connected, this will allow us to use the results of Section~\ref{sec:nonconcentration}.
The drawback is that we can no longer use the simple identity $\mu_{sn}=\mu_n*\dots*\mu_n$; so we shall have to write $\mu_{sn}$ as a weighted sum of convolutions related to the induced measure $\mu^\circ$ on the identity component, which makes the argument more technical.
The argument is identical to the one given in~\cite[Appendix~B]{HLL2021}, but we include it for completeness.

\bigskip

Let $(g_n)_{n \geq 1}$ be a sequence of independent random variables distributed according to $\mu$.
Consider the return times to $G^\circ$, 
\[
\tau(1) = \inf \{\, n \geq 1 \mid g_n \dotsm g_1 \in G^\circ \,\}
\]
and recursively for $m \geq 2$,
\[
\tau(m) = \inf \{\, n > \tau(m-1) \mid g_n \dotsm g_1 \in G^\circ \,\}.
\]
Those are the return times of a Markov chain on the finite space $G/G^\circ$, so that for every $m \geq 1$, $\tau(m)$ is almost surely finite.
In fact, by Kac's formula \cite[Lemma 5.4]{BenoistQuint}
\[
\Ebb[\tau(1)] = [G:G^\circ].
\]
The random variables $(g_{\tau(m)}\dots g_{\tau(m-1)+1})_{m\geq 0}$ are independent and identically distributed with law $\mu^\circ$, the law of $g_{\tau(1)} \dotsm g_1$.
Note that $\mu^\circ$ is a probability measure on $G^\circ$ and has the following properties \cite[Lemmas~4.40 and 4.42]{Aoun2011}.

\begin{lem}
\label{lm:mucirc}
Let $\mu$ be a probability measure on a real algebraic group $G$ and $\mu^\circ$ the induced measure on the identity component $G^\circ$.
Let $T=[G:G^\circ]$.
If $\mu$ admits some finite exponential moment, then:
\begin{enumerate}
\item The measure $\mu^\circ$ has some finite exponential moment;
\item For every $\omega > 0$, there exists $c = c(\mu,\omega) > 0$ such that for all $m$ sufficiently large, 
\(
\Pbb\bigl[ \abs{\tau(m) - T m} \geq \omega m\bigr] \leq e^{-cm}.
\)
\end{enumerate}
\end{lem}

In order to prove Theorem~\ref{thm:decaymun}, we shall need to relate the random walk defined by $\mu$ and the one defined by $\mu^\circ$. 
For that, we introduce, for $m \geq 1$ and $l \geq 1$, the law $\nu_l$ of the random variable
\[
g_{\tau(m)} \dotsm g_1 \quad \text{conditional to the event} \quad \tau(m) = l.
\]
Naturally, $\nu_l$ is also the law of the variable $g_l \dotsm g_1$ conditional to $\tau(m)=l$.
On the one hand, we may relate the measures $\nu_l$ to $(\mu^\circ)^{*m}$ with the formula
\begin{equation}
\label{eq:plnul}
(\mu^\circ)^{*m} = \sum_{l \in \N} p_l \nu_l.
\end{equation}
where $p_l = \Pbb[\tau(m) = l]$.
Here, we are hiding the dependency of $\nu_l$ and $p_l$ on $m$ in order to make notation less cumbersome. 

On the other hand, writing $l_1 + \dotsb + l_s + k =  n$ for some natural integers $n,s$ and $l_1,\dots,l_s$, we have
\begin{equation}
\label{eq:RWcondRT}
\mu^{*n} = \sum_{l_1 + \dotsb + l_s + k = n} p_{l_1} \dotsm p_{l_s} \mu^{*k} * \nu_{l_s} * \dotsm * \nu_{l_1} + ((\Pbb[\tau(sm) > n]))
\end{equation}
where the notation $((t))$ for some positive quantity $t$ means some unspecified positive measure of total mass at most $t$.
These two formulae will allow us to use the non-concentration properties of $(\mu^\circ)^{*m}$ to prove some Fourier decay estimate for $\mu^{*n}$.

Before we derive Theorem~\ref{thm:decaymun}, we note that the scaling automorphism $\phi_m^\circ$ on $E$ associated to $\mu^\circ$ is simply given by $\phi_n^\circ=\phi_{mT}$, where $T=[G:G^\circ]$.
This readily follows from the fact that if $\tE_i$ is any simple ideal in $\tE$ and $V$ any $G^\circ$-irreducible submodule of $\tE_i$, then $\lambda_1(\mu^\circ,V) = T\lambda_1(\mu,\tE_i)$.

\begin{proof}[Proof of Theorem~\ref{thm:decaymun}]
Let $\alpha_1 > 0$ be a given small number.
Since the algebraic group generated by $\mu^\circ$ is connected, Proposition~\ref{pr:nc2} applies to the induced random walk on $G^\circ$.
We let $\kappa = \kappa(\mu^\circ) > 0$ be the constant given by that proposition.
Let $D=\dim E$ and $s = s(E,\kappa) \geq 1$ and $\eps = \eps(E,\kappa,D) > 0$ be the constants given by Corollary~\ref{cor:fourier}.

Given $\alpha_1>0$, Proposition~\ref{pr:nc2} shows that for all $m$ large enough, the measure
\[
(\pi'\circ \phi_{mT})_*\bigl( ((\mu^\circ)^{*m})^{\pp D} \mm ((\mu^\circ)^{*m})^{\pp D} \bigr)
\]
satisfies $\NC(\frac{\alpha_1\eps}{2},\kappa,\tau)$ in $E'$ at scale $e^{-m}$ for some $\tau > 0$.
This implies that the same measure satisfies $\NC(\frac{\eps}{2},\kappa,\tau)$ in $E'$ at all scales $\delta \in {[e^{-m},e^{-\alpha_1 m}]}$.
Without loss of generality, we may assume that $\tau < \kappa \eps/2$ and $\tau < \eps /2$.

Let $\omega = \omega(\mu,\alpha_1)$ be a constant whose value is to be determined later.
Fix $n \geq 1$ large, and set $m = \left\lfloor (1-2\omega) \frac{n}{Ts} \right\rfloor$, where $T=[G:G^\circ]$. 
Everything below is true for $n$ sufficiently large (larger than some $n_0$ depending on $\mu$ and $\alpha_1$).
The letter $c$ denotes a small positive constant, whose value may vary from one line to the other, depending on $\mu$ and $\alpha_1$ but independent of $n$.

By Lemma~\ref{lm:mucirc}, we have
\[\Pbb[\tau(sm) > n - \omega n] \leq e^{-cn}\]
and 
\[\Pbb[\tau(sm) < n - 3 \omega n ] \leq e^{-cn}.\]
Put 
\[\Lcal = \{\, l \in \N \mid p_l \geq e^{-\frac{\alpha_1\tau}{2D}m} \,\}.\]
We can bound 
\[
\sum_{(l_1,\dotsc,l_s) \not\in \Lcal^s} p_{l_1} \dotsm p_{l_s} \leq sn e^{-\frac{\alpha_1\tau}{2D}m} \leq e^{-cn}.
\]
Thus, \eqref{eq:RWcondRT} becomes
\[
\mu^{*n} = \sum_{\substack{l_1,\dotsc,l_s \in \Lcal,\, \omega n \leq k \leq 3 \omega n \\ l_1 + \dotsb + l_s + k = n}}  p_{l_1} \dotsm p_{l_s} \mu^{*k} * \nu_{l_s} * \dotsm * \nu_{l_1}  + ((e^{-cn})).
\]
Let $\gamma\in F$.
To finish the proof of the theorem, it suffices to establish an upper bound of the form $e^{-cn}$ for the quantity
\begin{align*}
I_{l_1,\dotsc,l_s,k}(\xi)
:=&  \int_{\gamma G^\circ} e\bigl(\xi \circ \pi' \circ \phi_n(\gamma^{-1}g)\bigr) \dd \bigl( \mu^{*k} * \nu_{l_s} * \dotsm * \nu_{l_1} \bigr) (g)\\
=& \iint_{g \in \gamma G^\circ} e\bigl(\xi \circ \pi' \bigl( \phi_{n-smT}(\gamma^{-1}g) \phi_{smT}(h)\bigr)\bigr) \dd \mu^{*k}(g) \dd(\nu_{l_s} * \dotsm * \nu_{l_1})(h)\\
=& \int_{\gamma G^\circ}  \bigl(  (\pi'\circ\phi_{smT})_*(\nu_{l_s} * \dotsm * \nu_{l_1}) \bigr)^\wedge \bigl( \xi \cdot \phi_{n-smT} (\gamma^{-1}g) \bigr) \dd \mu^{*k}(g)
\end{align*}
uniformly for all $l_1,\dotsc,l_s \in \Lcal$ and $\omega n \leq k \leq 3 \omega n$ with $l_1 + \dotsb + l_s + k = n$.

\smallskip

First, we claim that uniformly for all $l\in\Lcal$, the measure
\[
(\pi' \circ \phi_{mT})_* \bigl(\nu_l^{\pp D} \mm \nu_l^{\pp D}\bigr)
\]
satisfies $\NC(\eps,\kappa,\tau/2)$ in $E'$ at all scales $\delta \in {[e^{-m},e^{-2\alpha_1 m}]}$, provided that $m \geq 1$ is large enough.
Indeed, developing $((\mu^\circ)^{*m})^{\pp D} \mm ((\mu^\circ)^{*m})^{\pp D}$ using \eqref{eq:plnul}, we see that for any $l\geq 1$,
\[
((\mu^\circ)^{*m})^{\pp D} \mm ((\mu^\circ)^{*m})^{\pp D} = p_l^{2D} \bigl(\nu_l^{\pp D} \mm \nu_l^{\pp D}\bigr) + ((1)).
\]
Observe that given two measures $\eta,\eta'$ such that $\eta=\delta^\sigma\eta'+((1))$, if $\eta$ satisfies $\NC(\eps,\kappa,\tau)$, then $\eta'$ satisfies $\NC(\eps+\sigma,\kappa,\tau-\sigma)$.
Therefore, the inequality $p_l^{2D} \geq e^{-\alpha_1 \tau m}$ for $l \in \Lcal$ together with the fact that the left-hand side rescaled by $\pi' \circ \phi_{mT}$ satisfies $\NC(\frac{\eps}{2},\kappa,\tau)$ in $E'$ at all scales $\delta \in {[e^{-m},e^{-\alpha_1 m}]}$ show our claim.
By Corollary~\ref{cor:fourier}, this implies, for all $\zeta \in (E')^*$ such that $e^{2\alpha_1 m}\leq \norm{\zeta}\leq e^{m}$,
\[
\bigl\lvert \bigl(  (\pi'\circ\phi_{smT})_*(\nu_{l_s} * \dotsm * \nu_{l_1}) \bigr)^\wedge(\zeta) \bigr\rvert \leq e^{- \frac{\alpha_1 \eps \tau}{(2D)^s} m} \leq e^{-cn}.
\]
Note that for any $g \in \gamma G^\circ$,
\begin{equation}
\label{eq:xigammag}
\norm{\xi} \norm{g^{-1}}^{-1} \ll \norm{\xi \cdot \phi_{n-smT} (\gamma^{-1}g) } \ll \norm{\xi} \norm{\phi_{n-smT}} \norm{g}.
\end{equation}
On the one hand, we have $0 \leq n-smT \leq 3\omega n$.
Hence, there exists a constant $C = C(\mu) \geq 1$ such that 
\[
\norm{\phi_{n-smT}} \leq e^{C\omega n}.
\]
On the other hand, using the assumption that $\mu$ has a finite exponential moment and Markov's inequality, we can find a constant $C = C(\mu) \geq 1$ such that for any $k \geq 1$, the $\mu^{*k}$-measure of the set of $g \in \Gamma$ such that 
\begin{equation}
\label{eq:goodmuk}
\norm{g} \leq e^{Ck} \quad \text{and}\quad \norm{g^{-1}} \leq  e^{Ck}
\end{equation}
is at least $1 - e^{-k}$.

Set $\alpha_0 = \frac{1}{4Ts}$ and let $\xi \in (E')^*$ be such that $e^{\alpha_1 n} \leq \norm{\xi} \leq e^{\alpha_0 n}$. Using $k \leq 3 \omega n$, we have, for any $g \in \Supp(\mu^{*k})$ satisfying \eqref{eq:goodmuk},
\[
e^{(\alpha_1 - 4C\omega)n} \leq \norm{\xi \cdot \phi_{n-smT} (\gamma^{-1}g)} \leq e^{(\alpha_0 + 5C\omega)n}.
\]
With the choice $\omega = \min\{\frac{\alpha_1}{8C}, \frac{1}{20CTs}\}$, we can guarantee that this implies
\[
e^{\alpha_1 m} \leq e^{\alpha_1 n/2} \leq \norm{\xi \cdot \phi_{n-smT} (\gamma^{-1}g) } \leq e^m.
\]
Putting everything together, we obtain
\[\abs{I_{l_1,\dotsc,l_s,k}(\xi)} \leq e^{-cn} + e^{-k} \leq e^{-c n} +e^{-\omega n}.\]
for all $l_1,\dotsc,l_s \in \Lcal$, $\omega n \leq k \leq 3 \omega n$ with $l_1 + \dotsb + l_s + k = n$.
This concludes the proof of the theorem.
\end{proof}

\section{From Fourier decay to granular structure}
\label{sec:wiener}

As in the previous section, $\mu$ denotes a probability measure on $\GL_d(\Z)$ and we study the random walk associated to $\mu$ on $\Tbb^d$, with starting distribution $\nu\in\Pcal(\Tbb^d)$.
The law of the walk at time $n$ is $\nu_n=\mu^{*n}*\nu$.
The goal of this section is to show that if $\nu_n$ has a large Fourier coefficient, then the starting distribution $\nu$ must have some strong concentration property.

\subsection{Concentration statement for the random walk}

In order to state the main proposition of this section, we need to set up some notation.
As before, $G$ denotes the algebraic subgroup generated by $\mu$, $E\subset\Mat_d(\R)$ denotes the algebra generated by the identity component $G^\circ$ of $G$, $F$ denotes the finite group $G/G^\circ$ and $T = \# F$.

Changing notation slightly, we now consider a decomposition of $E$
\[
E = E_0\oplus E_1\oplus\dots\oplus E_r
\]
into maximal sums of minimal ideals with same Lyapunov exponent for the action of $\mu^\circ$.
We assume that the summands are ordered so that
\[
\lambda_1(\mu^\circ,E_1) > \dots > \lambda_1(\mu^\circ,E_r) > 0 = \lambda_1(\mu^\circ,E_0).
\]
Here, $E_0$ is eventually trivial.

The group $G$ acts naturally on the space $V = \R^d$ and for $1 \leq j \leq r$, we let $V_i$ be the sum of all simple $G^\circ$-submodules $W \subset V$ such that $\lambda_1(\mu^\circ,W) = \lambda_1(\mu^\circ,E_i)$.
Equivalently, $V_i$ is also the sum of all simple $G$-submodules $W \subset V$ such that $\lambda_1(\mu,W) = \frac{1}{T}\lambda_1(\mu^\circ,E_i)$.
We have
\[
V=V_0\oplus V_1 \oplus \dots \oplus V_r.
\]
Let $\pi_i \colon V \to V_i$ denote the corresponding projection.
Define a quasi-norm on $V$ by
\[
\qnorm{v} = \max_{0\leq i\leq s} \norm{\pi_i(v)}^{\frac{1}{\lambda_1(\mu,V_i)}}
\]
where by convention
\[
\norm{\pi_0(v)}^{\frac{1}{0}} = \norm{\pi_0(v)}^{+\infty} = \left\{
\begin{array}{cl}
0 & \mbox{if}\ \norm{\pi_0(v)}\leq 1\\
+\infty & \mbox{otherwise}.
\end{array}
\right.
\]

This induces a quasi distance on $\Tbb^d$. For $x, y \in \Tbb^d$, define  
\[\dtil(x,y) = \left\{ 
\begin{array}{ll}
\qnorm{v - w} & \quad \text{if there are lifts $v \in V$ of $x$ and $w \in V$ of $y$ such that $\norm{v-w}\leq \frac{1}{2}$,}\\ 
1 & \quad \text{otherwise.}
\end{array}\right.
\] 
Neighborhoods of subsets of $\Tbb^d$ with respect to this quasi-distance will be denoted by $\Nbdtil(\mybullet,\mybullet)$.
Finally, for a rational subspace $W \subset V$, we let $W \mymod \Z^d$ denote its projection in $\Tbb^d$, which is a subtorus.
\begin{prop}
\label{pr:RWwiener}
Let $\mu$ be a probability measure on $\GL_d(\Z)$ with some finite exponential moment and $\nu$ be a Borel probability measure on $\Tbb^d$.
If the algebraic group $G$ generated by $\mu$ is semisimple, then there exist $C=C(\mu)\geq 0$ and $\tau > 0$ such that the following holds.

Assume that for some $t\in(0,\frac{1}{2})$,
\[
\abs{\hhat{\mu^{*n}*\nu}(a_0)} \geq t
\quad\text{for some } a_0 \in \Z^d \text{ and } n \geq C\log\frac{\norm{a_0}}{t}.
\]
Then, there exists $\gamma\in F$ such that, denoting
\[
W = (a_0\gamma E)^\perp
\]
there exists a finite subset $X \subset \Tbb^d$ such that
\[
(X - X)\cap \Nbdtil(W \mymod \Z^d,e^{-(1- 2 \tau)n}) = \{0\}
\]
and
\[
\nu \bigl(X + \Nbdtil(W \mymod \Z^d,e^{-(1-\tau)n}) \bigr) \geq t^{O(1)}.
\]
\end{prop}


The proof of Proposition~\ref{pr:RWwiener} is in two steps: First, using the results of Section~\ref{sec:fourier}, one shows that the inequality $\abs{\hhat{\mu^{*n}*\nu}(a_0)}\geq t$ implies that $\mu^{*n}*\nu$ has many large Fourier coefficients (reducing slightly the value of $n$) and then, one applies a Fourier analysis lemma originating in the work of Bourgain, Furman, Lindenstrauss and Mozes \cite[Proposition~7.5]{BFLM}.
We start with the statement and proof of a general version of that lemma adapted to our needs.

\subsection{A quantitative version of Wiener's lemma}

Wiener's lemma in harmonic analysis states that a measure $\nu$ on the torus $\Tbb^d$ is atom-free if and only if its Fourier series tends to zero in density, i.e. given $t>0$, the proportion of vectors $a\in\Z^d$ in a large ball $B(0,N)$ such that $\abs{\hat{\nu}(a)}\geq t$, tends to zero as $N$ goes to infinity.
In their paper \cite{BFLM}, Bourgain, Furman, Lindenstrauss and Mozes observed that this statement could be made quantitative: If $B(0,N)$ contains a proportion at least $s>0$ of integer vectors satisfying $\abs{\hat{\nu}(a)}\geq t$, then there exists a ball $B= B(x,\frac{1}{N})$ of radius $\frac{1}{N}$ in $\Tbb^d$ such that $\nu(B)\gg (st)^3$, where the involved constant depends only on $d$.
In order to later be able to use the quasi-norm adapted to a random walk, we need to generalize this statement.
It turns out to be most convenient to formulate the lemma in terms of convex sets and polar pairs.

\bigskip

We will need to generalize slightly the notation $\Ncal(\mybullet,\delta)$ of covering number.
Instead of covering a set by balls, we will use translates of a convex body.
Recall that a \emph{convex body} is a compact convex set in $\R^d$, symmetric about $0$, i.e. such that $B=-B$, and containing $0$ in its interior.
Given a convex body $B \subset\R^d$, and $A \subset V$ a bounded non-empty subset, we define the \emph{covering number of $A$ by $B$} by
\[
\Ncal(A,B) = \min \setBig{N \geq 1}{ \exists x_1, \dotsc,x_N \in V,\, A \subset \bigcup_{i=1}^N (x_i+B)}.
\]
We shall also say that $A$ is \emph{$B$-separated}, if $(A - A)\cap B = \{0\}$.
Let us briefly list some useful properties of covering numbers.
These may be used without explicit mention in the rest of the paper; the elementary proofs are left to the reader.
Notice that if $\norm{\cdot}$ is the norm on $V$ for which $B$ is the unit ball, then $\Ncal(\cdot,B)$ is simply the covering number at scale $1$ for the distance associated to the norm.
\begin{itemize}
\item Let $f \colon V \to W$ be a linear map to another Euclidean space $W$. Then, for any set $A\subset V$ and any convex body $B\subset W$,
\begin{equation}
\label{eq:NfAfB}
\Ncal\bigl(f(A),f(B)\bigr) \leq \Ncal(A,B)
\end{equation}
with equality if $f$ is a linear isomorphism.
\item Let $B' \subset V$ be another convex body, then 
\[\Ncal(A,B') \leq \Ncal(A,B) \Ncal(B,B').\]
\item The previous point combined with John's ellipsoid theorem~\cite[Theorem~2A, page 87]{schmidt_da} shows that for any convex body $B$ there is an ellipsoid $E$ such that for all non-empty subsets $A \subset V$,
\[
\Ncal(A,B) \asymp \Ncal(A,E)
\]
where the implied constant in the $\asymp$ notation depends only on $\dim V$. 
In particular,
\[
\Ncal(A,2B) \asymp \Ncal(A,B)
\]
within constants depending only on $\dim V$.
\item In $A$, maximal $B$-separated subsets have cardinality at least $N(A,B)$.
\item If $B$ is symmetric and $A$ is $2B$-separated then $\Ncal(A,B) \geq \# A$.
\item For any set $A\subset V$ and any convex body $B\subset W$,
\[
\Ncal(A,2B) \asymp \Ncal(A,B)
\]
where the implied constant depends only on $\dim V$.
\item Let $f \colon V \to W$ be a surjective linear map between Euclidean spaces.
Let $B,C \subset V$ be convex bodies and let $A \subset C$ be a subset of $C$.
We have
\begin{equation}
\label{NfANfC}
\frac{\Ncal\bigl(f(A),f(B)\bigr)}{\Ncal\bigl(f(C),f(B)\bigr)} \gg  \frac{\Ncal(A,B)}{\Ncal(C,B)}
\end{equation}
where the implied constant in the $\gg$ notation depends only on $\dim V$. 
\end{itemize}

Let $\R^d$ be endowed with the usual scalar product.
Given a convex body $C \subset \R^d$, its \emph{polar set} $C^*$ is defined by
\[
C^* = \set{x \in \R^d}{\forall y\in C,\ \bracket{x,y} \leq 1}.
\]
If $C\supset B(0,2)$, then $C^*\subset B(0,\frac{1}{2})$ and we naturally identify $C^*$ with its projection to $\Tbb^d$.
The quantitative version of Wiener's lemma that we need is given by the proposition below.

\begin{prop}
\label{pr:polarwiener}
Let $\nu$ be a probability measure on $\Tbb^d$ and write for $t > 0$
\[A_t= \set{a \in \Z^d}{ \abs{\hat\nu(a)} \geq t}.\]
Assume that for some convex bodies $ B\subset C\subset\R^d$ containing $B(0,1)$, we have,
for some $c_0 \in \Z^d$ and some $s > 0$
\begin{equation}
\label{eq:AtcapC}
\Ncal(A_t\cap (c_0+C), B) \geq s\cdot \frac{\abs{C}}{\abs{B}}.
\end{equation}
Then there exists a $B^*$-separated subset $X \subset \Tbb^d$ such that
\[\nu(X + C^*) \gg_d s^{3/2}t^6.\]
\end{prop}
\begin{proof}
The implied constants in the Vinogradov notation in this proof depend only on $d$.
We shall need two auxiliary functions; the first one corresponds to the pair of convex sets $(C,C^*)$, the second to $(B,B^*)$:
\begin{enumerate}[label=(\arabic*)]
\item There exists a smooth function $\psi \colon \Tbb^d \to \R_{\geq 0}$ such that
\begin{enumerate}[label=(\alph*)]
\item $\int_{\Tbb^d} \psi = 1$,
\item $\psi \ll \frac{1}{\abs{C^*}}\1_{C^*}$,
\item $\hhat\psi \gg \1_{2C \cap \Z^d}$.
\end{enumerate}
\item There exists a smooth function $\phi \colon \Tbb^d \to \R_{\geq 0}$ such that
\begin{enumerate}
\item $\phi \gg \1_{B^*}$,
\item $\hhat\phi$ is real and positive and $\hhat\phi \ll \frac{1}{\abs{B}^2}\1_{B}\mm \1_{B} \leq \frac{1}{\abs{B}} \1_{2B}$.
\end{enumerate}
\end{enumerate}
One obtains $\psi$ by taking any smooth symmetric bump function supported on $\frac{1}{16}C^*$ with integral $\int \psi =1$.
The third property follows from the fact that for every $\xi\in 2C \cap \Z^d$ and $x\in\frac{1}{16}C^*$, one has $\bracket{\xi,x}\leq\frac{1}{8}$ and hence $\Re\bigl( e(\bracket{\xi,x}) \bigr) \geq \frac{1}{2}$.
The function $\phi$ can be given explicitly by the formula
\(
\phi(x) =\absBig{ \frac{1}{\abs{B}}\sum_{a\in\frac{1}{8}B\cap\Z^d} e(\bracket{a,x})}^2
\)
for all $x$ in $\Tbb^d$.
The second item is immediate by definition of $\phi$, and the first one follows from the fact that by Minkowski's first theorem on convex bodies, one has $\#(\frac{1}{8}B\cap\Z^d)\gg\abs{B}$.

Pick a maximal $2B$-separated subset $A' \subset A_t \cap (c_0+ C)$ such that all coefficients $\hat\nu(a)$, $a \in A'$ fall in the same quadrant of $\C$.
One still has $\# A' \gg \Ncal(A_t\cap (c_0 + C), B)$ and moreover
\[
\absBig{\sum_{a \in A'} \hat\nu(a)} \geq \frac{t \# A'}{\sqrt{2}}.
\]
By the Cauchy-Schwarz inequality,
\begin{equation*}
\sum_{a, b \in A'} \hat\nu(a -b) = \int_{\Tbb^d} \absbig{\sum_{a\in A'}e(\bracket{a,x})}^2 \dd \nu(x)
\geq \absBig{\int_{\Tbb^d} \sum_{a\in A'}e(\bracket{a,x}) \dd \nu(x)}^2
\gg t^2 (\# A')^2. 
\end{equation*}
Hence, there exists a translate $A$ of $A'$ such that $A \subset A' - A' \subset 2C$ and
\begin{equation*}
\absBig{\sum_{a\in A} \hat\nu(a)} \gg t^2 \# A
\end{equation*}
and 
\begin{equation}
\label{eq:AgeqBC}
\# A = \#A' \gg s\cdot\frac{\abs{C}}{\abs{B}}.
\end{equation}

Consider the function $f \colon \Tbb^d \to \R$ defined by
\[
\forall x \in \Tbb^d,\quad  f(x) = \sum_{a\in A} e(\bracket{a,x}).
\]
On the one hand, using the definition of $f$, the properties of $\phi$ and the fact that $A$ is $2B$-separated, one has, for any $y \in \Tbb^d$,
\begin{align*}
\int_{y+B^*} \abs{f}^2 &= \int_{\Tbb^d} \1_{B^*}(x - y) \abs{f(x)}^2 \dd x\\
& \ll  \sum_{a_1,a_2 \in A}  \int \phi(x-y) e(\bracket{a_1- a_2,x})\dd x\\
&\leq  \sum_{a_1,a_2 \in A} \hat\phi(a_1 - a_2)\\
&\ll  \frac{1}{\abs{B}} \sum_{a_1,a_2 \in A}  \1_{2B}(a_1 - a_2)\\
&\ll  \frac{\# A}{\abs{B}}.
\end{align*}
On the other hand, from the properties of $\psi$ and of those of $A$,
\[
\absBig{\int_{\Tbb^d} f \dd(\nu\pp\psi)} = \absBig{\sum_{a \in A} \hat\nu(a) \hat\psi(a)} \gg  \absBig{\sum_{a \in A} \hat\nu(a)} \gg t^2 \# A.
\] 

Let $(y_i)_{i\in I}$ be a maximal family of $(4B^*)$-separated points in $\Tbb^d$.
Then the translates $(y_i + B^*)_{i\in I}$ are disjoint and have a total volume $\gg 1$.
By Fubini's theorem, 
\[
\int \dd x \sum_{i \in I} \int_{x + y_i + B^*} f \dd (\nu\pp\psi) = \sum_{i \in I}\abs{y_i+B^*} \int_{\Tbb^d} f \dd (\nu\pp\psi).
\]
Hence, translating all $y_i$ by some $x\in\Tbb^d$ if necessary, we may assume that
\begin{equation}
\label{eq:sumI}
\sum_{i \in I} \absBig{\int_{y_i + B^*} f \dd (\nu\pp\psi)} \gg t^2 \# A.
\end{equation}
By the Cauchy-Schwarz inequality, for each $i \in I$, 
\begin{align*}
\absBig{\int_{y_i + B^*} f \dd (\nu\pp\psi)} & \leq \sqrt{\int_{y_i+B^*} \abs{f}^2} \sqrt{\int_{y_i+B^*} (\nu\pp\psi)^2}\\
&\ll \sqrt{\frac{\# A}{\abs{B}}} \sqrt{(\nu\pp\psi)(y_i+B^*) \max_{y_i +B^*}\nu\pp\psi}\\
&\ll \nu(y_i+B^* + C^*) \sqrt{\frac{h_i \# A \abs{C}}{\abs{B}}} 
\end{align*}
where $h_i = \frac{\abs{C^*}\max_{y_i+B^*}\nu\pp\psi}{\nu(y_i+B^* + C^*)}$.
Recalling \eqref{eq:AgeqBC} and \eqref{eq:sumI}, we obtain some constant $L=L(d) >1$ depending only on $d$ such that
\[
\sum_{i \in I} \nu(y_i+B^* + C^*) h_i^{1/2} \geq \frac{s^{1/2}t^2}{L}.
\]
On the other hand, since for every $x$ in $\Tbb^d$, one has $(\nu\pp\psi)(x) \ll \frac{\nu(x+C^*)}{\abs{C^*}}$, so
\[
h_i = \frac{\abs{C^*}\max_{y_i+B^*}\nu\pp\psi}{\nu(y_i+B^* + C^*)} \ll \frac{\max_{x\in y_i+B^*}\nu(x+C^*)}{\nu(y_i+B^*+C^*)} \leq 1
\]
and therefore, increasing the value of $L$ if necessary, we may assume that
\[
\forall i,\quad h_i \leq L.
\]
Finally, $(y_i)_{i \in I}$ is $4B^*$-separated so
\[
\sum_{i \in I} \nu(y_i+B^* + C^*) \leq \sum_{i \in I} \nu(y_i+2\cdot B^*) \leq 1
\]
and we may set $J = \set{i \in I}{h_i \geq \frac{st^4}{4L^2}}$ to find
\[
\sum_{i \in J} \nu(y_i + B^* + C^*) \geq \frac{s^{1/2}t^2}{2L^{3/2}}.
\]

For each $i \in J$, fix $x_i \in y_i+B^*$ such that
\[
(\nu \pp \psi)(x_i)= \max_{y_i + B^*} \nu \pp \psi
\]
and let
\[
X = \{x_i\ ;\ i\in J\}.
\]
Since the family $(y_i)_{i\in I}$ is $4B^*$-separated, the set $X$ is $B^*$-separated.
For the second property, note that for each $i$ in $J$,
\[
\nu(x_i + C^*) \geq \abs{C^*} (\nu \pp \psi)(x_i) = h_i \nu(y_i+B^*+C^*)
\gg st^4 \nu(y_i+B^*+C^*).
\]
so that
\[
\nu(X+C^*)  = \sum_{i \in J} \nu(x_i+C^*) \gg st^4 \sum_{i \in J}\nu(y_i+B^*+C^*) \gg s^{3/2}t^6.
\]
\end{proof}

\subsection{Proof of Proposition~\ref{pr:RWwiener}}
To prove Proposition~\ref{pr:RWwiener}, we follow the same pattern as in \cite{HS2019}. The only difference here is that we need to find the correct polar pairs $(B,B^*)$ and $(C,C^*)$ to apply Proposition~\ref{pr:polarwiener}.
Before we start the proof, we record to elementary lemmas from~\cite{HS2019}.

\begin{lem}[{\cite[Lemma~4.3]{HS2019}} Additive structure of large Fourier coefficients]
\label{lm:addstruct}
Let $\mu$ be a Borel probability measure on $\SL_d(\Z)$ and $\nu$ a Borel probability measure on $\Tbb^d$.
If 
\[ \abs{\widehat{\mu * \nu}(a_0)} \geq t_0 > 0,\]
then for any integer $k \geq 1$, the set
\[A = \bigl\{ g \in \Mat_d(\Z) \mid \abs{\hat{\nu}(a_0g)} \geq t_0^{2k}/2 \bigr\}\]
satisfies
\[\bigl(\mu^{\pp k} \mm \mu^{\pp k}\bigr) (A) \geq \frac{t_0^{2k}}{2}.\]
\end{lem}

\begin{lem}[{\cite[Lemma~4.4]{HS2019}} Regularity from Fourier decay]
\label{lm:regFourier}
Given $D \geq 1$ and $\alpha > 0$ sufficiently small, there exist constants $c = c(D,\alpha) > 0$ and $C_1 = C_1(D,\alpha) > 0$ such that the following holds for all $0 < \delta < c t$.
Let $\eta$ be a Borel measure on $\R^D$, of total mass $\mu(\R^D) \leq 1$. 
Let $A$ be a subset of $\R^D$. Assume
\begin{enumerate}
\item $\Supp(\eta) \subset B(0, \delta^{-\alpha})$,
\item for all $\xi \in \R^D$ with $\delta^{-\alpha} \leq \norm{\xi} \leq \delta^{-1- \alpha}$, $\abs{\hat\eta(\xi)} \leq \norm{\xi}^{-C_1}$,
\item $\eta(A) \geq t$.
\end{enumerate}
Then there exists $a \in \R^D$ such that 
\[\Ncal(A \cap B(a, \delta^\beta), \delta) \geq c t^{D + 1} \Bigl( \frac{\delta^\beta}{\delta} \Bigr)^D,\]
where $\beta = (2D + 1) \alpha $.
\end{lem}

We are ready to prove the main proposition of this section.

\begin{proof}[{Proof of Proposition~\ref{pr:RWwiener}}]
We shall use Lemma~\ref{lm:regFourier} with $D = \dim E'$ and
\[
\alpha \leq \frac{\min_{1\leq i\leq r}\lambda_1(\mu,V_i)}{3(2D+1)\max_i\lambda_1(\mu,V_i)}
\quad\mbox{so}\quad
\beta \leq \frac{\min_{1\leq i\leq r}\lambda_1(\mu,V_i)}{3\max_{1\leq i\leq r}\lambda_1(\mu,V_i)}.
\]
Let $C_1=C_1(D,\alpha)$ be as in the lemma.
Take $\alpha_0 = \alpha_0(\mu)$ as in Theorem~\ref{thm:decaymun} with the additional condition that
\[
\alpha_0\beta <\min_{1\leq i\leq r}\lambda_1(\mu,V_i)
\]
and set $\alpha_1 = \frac{\alpha \alpha_0}{2}$ and $c_0=c_0(\mu, \alpha_1)$ as in Theorem~\ref{thm:decaymun}. 
Set also $k_0= \bigl\lceil\frac{\alpha_0 C_1}{c_0}\bigr\rceil$ and $k = Tk_0$, where $T=\# F$.

Assume
\[
\abs{\widehat{\mu^{*n}*\nu}(a_0)} \geq t.
\]
By Lemma~\ref{lm:addstruct}, there is a subset $A \subset \Mat_d(\Z)$ such that 
\[
\forall g \in A,\quad \abs{\hat\nu(a_0 g)} \geq \frac{t^{2k}}{2}
\]
and
\[\bigl((\mu^{*n})^{\pp k} \mm (\mu^{*n})^{\pp k}\bigr) (A) \geq \frac{t^{2k}}{2}.\]
Note that $\mu^{*n}$ is supported on $\bigcup_{\gamma\in F} \gamma E$.
We can cover $\mu^{*n}$ by its restrictions $(\mu^{*n})_{| \gamma E}$ to each subspace $\gamma E$.
Thanks to the choice of $k$ and the commutativity of additive convolutions, there exists $\gamma\in F$ such that
\[
\left((\mu^{*n})_{| \gamma E}\right)^{\pp k_0} \pp (\text{a probability measure}) (A) \geq \frac{t^{2k}}{2T^{2k}}.
\]
Hence for some $x_1 \in \Mat_d(\Z)$, we have
\[
\left((\mu^{*n})_{| \gamma E}\right)^{\pp k_0}(x_1 + A) \geq \frac{t^{2k}}{2T^{2k}}.
\]
Let $\eta'$ be the pushforward of $\left((\mu^{*n})_{| \gamma E}\right)^{\pp k_0}$ under the map
\(
g  \mapsto  (\pi'\circ \phi_{n})(\gamma^{-1}g)
\)
and let
\[
A' = (\pi'\circ \phi_{n}) \bigl(E\cap\gamma^{-1}(x_1+A)\bigr) \subset E'
\]
so that
\begin{equation}
\label{eq:etapAp}
\eta'(A') \gg t^{2k}.
\end{equation}

Lemma~\ref{lm:regFourier} will be used at scale $\delta = e^{-\frac{\alpha_0 n}{2}}$.
By the definition of $\mu_{n,\gamma}$ in Section~\ref{sec:fourier}, we have $\eta' = \mu_{n,\gamma}^{\pp k_0}$.
By Theorem~\ref{thm:decaymun}, for all $\xi \in E'^*$ with $\delta^{-\alpha}=e^{\alpha_1 n} \leq \norm{\xi} \leq e^{\alpha_0 n} = \delta^{-2}$, we have
\[ \abs{\widehat{\eta'}(\xi)} \leq e^{-k_0c_0 n} \leq \norm{\xi}^{-C_1}.\]
In view of the large deviation principle for $\mu^{*n}$, we may replace $\eta'$ by its restriction to $B(0,\delta^{-\alpha})$ while maintaining \eqref{eq:etapAp} and the conclusion of  Theorem~\ref{thm:decaymun}.
Thus by Lemma~\ref{lm:regFourier} applied to $\eta'$ and $A'$, there exists $x_2 \in \Ball_{E'}(0, \delta^{-\alpha})$ such that 
\begin{equation}
\label{eq:Ba2capA}
\Ncal\bigl(A'\cap B(x_2,\delta^{\beta}),\delta\bigr) \gg t^{O(1)} \delta^{-D(1-\beta)}.
\end{equation}
Now define convex bodies in $E$ by
\[
C_0 = \phi_{-n}\big(\Ball_{E'}(0,\delta^{\beta})\times \Ball_{E_0}(0,R)\big)
\quad\mbox{and}\quad
B_0 = \phi_{-n}\big(\Ball_{E'}(0,\delta)\times \Ball_{E_0}(0,R)\big)
\]
where $R = O_\mu(k)$ is a constant large enough so that $\gamma^{-1}A \subset E' \times \Ball_{E_0}(0,R)$. 
Note that $C_0\supset B_0$ and, since we took $\alpha_0\beta <\min_{1\leq i\leq r}\lambda_1(\mu,V_i)$,
\[
B_0\supset\Ball_E(0,1).
\]
Then inequality~\eqref{eq:Ba2capA} implies that for some $x_3$ in $E$,
\[
\Ncal(\gamma^{-1}A\cap (C_0 +x_3),B_0) \gg t^{O(1)}\delta^{-D(1-\beta)} \asymp t^{O(1)}\cdot\frac{\abs{C_0}}{\abs{B_0}}.
\]
Indeed, with
\[
\begin{array}{cccc}
f_1 \colon & E & \to & E'\\
& x & \mapsto & \pi'\circ\phi_{n}(\gamma^{-1}x_1+x)
\end{array}
\]
one has $f_1(\gamma^{-1}A) \supset A'$, $\pi'\circ \phi_{n}(B_0) = \Ball_{E'}(0,\delta)$, and taking $x_3 \in E'$ such that $\pi'\circ\phi_n(\gamma^{-1}x_1+x_3)=x_2$, $f_1(C_0 + x_3) = \Ball_{E'}(x_2,\delta^{\beta})$.
The choice of $R$ guarantees that $f_1\bigl(\gamma^{-1}A \cap (C_0 + x_3)\bigr) = f_1(\gamma^{-1}A) \cap f_1(C_0 + x_3)$. One concludes using the inequality~\eqref{eq:NfAfB} on $f_1$.

Now let
\[
C_1 = a_0\gamma C_0 
\quad\mbox{and}\quad
B_1 = a_0\gamma B_0 
\]
and apply \eqref{NfANfC} to $f\colon x\mapsto a_0\gamma x$ to obtain, with $c_0=a_0\gamma x_3$,
\[
\frac{\Ncal\bigl(a_0 A\cap(C_1+c_0),B_1\bigr)}{\Ncal(C_1,B_1)} \gg \frac{\Ncal\bigl(\gamma^{-1}A\cap(C_0+x_3),B_0\bigr)}{\Ncal(C_0,B_0)}
	\gg t^{O(1)}
\]
whence
\[
\Ncal\bigl(a_0 A\cap(C_1+c_0),B_1\bigr)\geq t^{O(1)} \frac{\abs{C_1}_{a_0\gamma E}}{\abs{B_1}_{a_0\gamma E}}.
\]
Since $C_1\subset B_1$ and $B_1$ contains a ball of radius $1$ in $a_0\gamma E$, we may set
\[
C = C_1 + \Ball_{\R^d}(0,2)
\quad\mbox{and}\quad
B = B_1 + \Ball_{\R^d}(0,2)
\]
to get convex bodies in $\R^d$ containing $\Ball_{\R^d}(0,2)$ such that
\[
\Ncal\bigl(a_0 A\cap(C+c_0),B\bigr)\geq t^{O(1)} \frac{\abs{C}}{\abs{B}}.
\]
Since $\hat\nu(a_0g)\geq t^{O(1)}$ for every $g\in A$, Proposition~\ref{pr:polarwiener} shows that there exists a $B^*$-separated subset $X\subset\Tbb^d$ such that
\[
\nu(X+C^*) \geq t^{O(1)}.
\]

To conclude, it remains to describe the sets $B^*$ and $C^*$.
For that, first consider a decomposition of the space of linear forms on $\R^d$ into irreducible components under the right action of $G^\circ$
\[
(\R^d)^* = V' = V^{(1)}\oplus\dots\oplus V^{(r)}
\]
and write $p_i$, $i=1,\dots,k$ for the corresponding projections.
Since $a_0\gamma$ is an integer vector, and each $V^{(i)}$ is defined over a number field,
there exists a constant $C > 0$ such that for each $i$ such that $p_i(a_0\gamma)\neq 0$, one has
\[
\norm{a_0}^{-C} \ll \norm{p_i(a_0\gamma)} \ll \norm{a_0}.
\]
Therefore, for any $\eps>0$, we may choose $C_\eps\geq 0$ such that $n\geq C_\eps\log\frac{\norm{a_0}}{t}$ implies, for $i=1,\dots,k$,
\[
p_i(a_0\gamma)=0
\quad\mbox{or}\quad
e^{-\eps n} \leq \norm{p_i(a_0\gamma)} \leq e^{\eps n}.
\]
Thus, if  $p_i(a_0\gamma)\neq 0$ and $\lambda_1(\mu,V^{(i)})\neq 0$, then
\[
p_i(a_0\gamma) B_0 \subset \Ball_{V^{(i)}}(0, e^{(\lambda_1(\mu,V^{(i)})+\eps) n}\delta)
\]
and
\[
p_i(a_0\gamma) C_0 \supset \Ball_{V^{(i)}}(0, e^{(\lambda_1(\mu,V^{(i)})-\eps) n}\delta^{\beta}).
\]
Now consider the decomposition of $V=\R^d$ according to Lyapunov exponents
\[
V = V_0\oplus V_1\oplus\dots\oplus V_r,
\]
where for $i=0,\dots,r$, $V_i$ is the sum of all irreducible $G$-submodules of $V$ with Lyapunov exponent $\lambda_1(\mu,E_i)$.
Since $W = (a_0\gamma E)^\perp$ is a submodule, it can be written
\[
W = W_0\oplus\dots\oplus W_r,
\quad\mbox{where}\quad W_i=W\cap V_i,\quad i=0,\dots,r.
\]
An elementary computation based on the above observations shows that for some compact subset $A_0\subset V_0$ containing $\Ball_{V_0}(0,\frac{1}{2})$, (we identify subsets of $\Ball_V(0,\frac{1}{2})$ with their projections in $\Tbb^d$)
\[
B^* \supset A_0 \times \prod_{1\leq i\leq r} \set{v_i\in V_i}{d(v_i,W_i)\leq e^{-(\lambda_1(\mu,V_i)+\eps)n}\delta^{-1}\ \mbox{and}\ \norm{v_i}\leq\frac{1}{2}}
\]
and
\[
C^* \subset A_0 \times \prod_{1\leq i\leq r} \set{v_i\in V_i}{d(v_i,W_i)\leq e^{-(\lambda_1(\mu,V_i)-\eps)n}\delta^{-\beta}\ \mbox{and}\ \norm{v_i}\leq\frac{1}{2}}.
\]
Recalling $\delta=e^{-\frac{\alpha_0 n}{2}}$ and setting $\tau=\frac{\alpha_0}{5\max_i\lambda_1(\mu,V_i)}>0$, we may choose $\eps>0$ small enough so that
\[
e^{-(\lambda_1(\mu,V_i)+\eps)n}\delta^{-1} = e^{-(\lambda_1(\mu,V_i)+\eps-\frac{\alpha_0}{2})n} \geq e^{-(1-2\tau)\lambda_1(\mu,V_i)n}
\]
and
\[
e^{-(\lambda_1(\mu,V_i)-\eps)n}\delta^{-\beta} = e^{-(\lambda_1(\mu,V_i)-\eps-\frac{\beta\alpha_0}{2})n} \leq e^{-(1-\tau)\lambda_1(\mu,V_i)n}.
\]
Finally, since $A_0$ can be covered by a bounded number of translates of $\Ball_{V_0}(0,\frac{1}{2})$, we may assume $A_0\subset\Ball_{V_0}(0,\frac{1}{2})$, and then $B^*\supset \Nbdtil(W \mymod \Z^d,e^{-(1- 2 \tau)n})$ while $C^*\subset \Nbdtil(W \mymod \Z^d,e^{-(1- \tau)n})$.
\end{proof}

For some technical reason, we shall have to work on a union of tori $\Tbb^d \times F$, where $\Gamma$ acts diagonally.
For a measure $\nu$ on $\Tbb^d \times F$ and $a \in \Z^d$, we write $\hhat{\mu^{*n}*\nu}(a,1)$ for the Fourier coefficient at frequency $a$ of the restriction of $\nu$ to $\Tbb^d \times \{1\}$ viewed as a measure on $\Tbb^d$.

A more careful look at the proof gives us the following slightly more precise version of Proposition~\ref{pr:RWwiener}.
\begin{coro}
\label{cr:RWwiener}
Let $\mu$ be a probability measure on $\GL_d(\Z)$ with finite exponential moment.
Assume the algebraic group $G$ generated by $\mu$ is semisimple and write $F = G/G^\circ$.
Let $\nu$ be a Borel probability measure on $\Tbb^d \times F$.
Then there exist $C = C(\mu)\geq 0$ and $\tau > 0$ such that the following holds.

Assume that for some $t\in(0,\frac{1}{2})$,
\[
\absbig{\hhat{\mu^{*n}*\nu}(a_0,1)} \geq t
\quad\text{for some } a_0 \in \Z^d \text{ and } n \geq C\log\frac{\norm{a_0}}{t}.
\]
Then, there exists $\gamma\in F$ such that, denoting
\[
W = (a_0\gamma E)^\perp
\]
there exists a finite subset $X \subset \Tbb^d \times \{\gamma^{-1}\}$ such that
\[
(X - X)\cap \Nbdtil(W \mymod \Z^d,e^{-(1- 2 \tau)n}) = \{0\}
\]
and
\[
\nu \bigl(X + \Nbdtil(W \mymod \Z^d,e^{-(1-\tau)n}) \bigr) \geq t^{O(1)}.
\]
\end{coro}
Here, of course, the addition on $\Tbb^d \times \{\gamma^{-1}\}$ is defined for the torus coordinate.
\begin{proof}[Proof sketch]
Decomposing the measure $(\mu^{*n}*\nu)_{\Tbb^d\times 1}$ as
\[
(\mu^{*n}*\nu)_{|\Tbb^d\times\{1\}}
	= \sum_{\gamma\in F} \mu^{*n}_{|\gamma G^\circ}*\nu_{|\Tbb^d\times\{\gamma^{-1}\}}
\]
we find that for some $\gamma$, one has, up to a constant depending only on $[G:G^\circ]$,
\[
\absbig{(\mu^{*n}_{|\gamma G^\circ}*\nu_{|\Tbb^d\times\{\gamma^{-1}\}})^\wedge(a_0)} \gg t.
\]
Let $k=\dim E'$.
Lemma~\ref{lm:addstruct} shows that there exists a set $A\subset\Mat_d(\Z)$ such that
\[
\forall g\in A,\qquad \abs{\widehat{\nu_{|\Tbb^d\times\{\gamma^{-1}\}}}(a_0g)} \gg t^{2k}.
\]
and
\[
\left((\mu^{*n}_{|\gamma G^\circ})^{\boxplus k} \boxminus (\mu^{*n}_{|\gamma G^\circ})^{\boxplus k}\right)(A) \gg t^{2k}
\]
Reasoning as in the proof of Proposition~\ref{pr:RWwiener}, we deduce that $\nu_{|\Tbb^d\times\{\gamma^{-1}\}}$ has many large Fourier coefficients in $a_0\gamma E$ and therefore must be concentrated around a finite subset of well-separated translates of neighborhoods of $W=(a_0\gamma E)^\perp$.
\end{proof}

\section{Concentration and unstability of the random walk}
\label{sec:granulation}

In this section, we finally prove the main result of the paper.
We consider a probability measure $\mu$ on $\GL_d(\Z)$ and the associated random walk on the torus $\Tbb^d$, starting from a point $x_0\in\Tbb^d$.
Letting $\Gamma$ be the group generated by $\Supp\mu$ and $G$ the Zariski closure of $\Gamma$, we assume that $G$ is semisimple as an algebraic group, and we show --- in a quantitative way --- that if the law $\mu^{*n}*\delta_{x_0}$ of the random walk is not exponentially close to the Haar measure on $\Tbb^d$, then the starting point $x_0$ is exponentially close to a proper closed invariant subset.

\bigskip

But let us first introduce a new space on which it is convenient to study the random walk, especially to overcome issues related to being Zariski disconnected.
As before, $G^\circ$ denotes the identity component of $G$, and $F=G/G^\circ$.
The subalgebra of $\Mat_d(\R)$ generated by $G^\circ$ is denoted by $E$.
In order to keep track of the coset modulo $G^\circ$, we let
\[
Y_0 = \Tbb^d \times F
\]
and let $\Gamma$ act on $Y_0$ diagonally.

Let $W_0$ be a rational $G^\circ$-invariant subspace of $V = \R^d$.
We can define a factor of $Y_0\to Y$ by
\[
Y = \bigsqcup_{\gamma \in F} V/ (\gamma W_0+\Z^d) \times \{\gamma\}.
\] 
The action of $\Gamma$ on $Y$ is defined in the obvious way.
This way, the natural projection $Y_0 \to Y$ is $\Gamma$-equivariant.

Let $V_0$ denote the sum of all compact factors of $G$ in $V=\R^d$, that is, the sum of irreducible subrepresentations $W \subset V$ such that $\lambda_1(\mu,W)=0$.
Given $a_0\in\Z^d$ for which the random walk $\mu^{*n}*\delta_{x_0}$ has large Fourier at $a_0$, we shall set 
\[
W_0 = V_0 + (a_0E)^\perp
\]
and study the random walk on the space $Y$ associated to this $W_0$.

In the introduction, we defined the quasi-distance adapted to the random walk on $V=\R^d$ and on $\Tbb^d$.
Similarly, we can define a quasi-distance on each of the tori $V/ (\Z^d + \gamma W_0)$ in $Y$.
Together, theses quasi-distances define a quasi-distance on $Y$ by the formula
\begin{equation*}
\dtil\bigl( (x_1,\gamma_1),  (x_2,\gamma_2) \bigr)= \begin{cases} \dtil(x_1,x_2) & \text{if } \gamma_1 = \gamma_2,\\ +\infty &\text{otherwise.}\end{cases}
\end{equation*}

Below we will state a slightly more precise version of Theorem~\ref{thm:finali}.
We fix a $G$-invariant Euclidean norm on $V_0$ and write $\Ball_{V_0}(0,R)$ for the closed ball in $V_0$ with radius $R > 0$ with respect to this norm. 
Given some parameter $Q > 0$, we note that
\[
\Ball_{V_0}(0,Q)+\bigcup_{q\leq Q} \frac{1}{q} \Z^d \subset V=\R^d
\]
is $\Gamma$-invariant.
As a consequence, the set
\[
Z_{Q} =\bigsqcup_{\gamma \in F} \bigl(\Ball_{V_0}(0,Q)+\bigcup_{q\leq Q} \frac{1}{q} \Z^d \mymod (\gamma W_0 + \Z^d)\bigr)\times \{\gamma\}
\]
is a $\Gamma$-invariant closed subset of $Y$.

\begin{thm}
\label{thm:final}
Assume that $\mu$ has a finite exponential moment and the algebraic group $G$ is semisimple.
Then for every $\lambda\in(0,1)$, there exist $C=C(\mu,\lambda)\geq 0$ such that the following holds.

Given $a_0 \in \Z^d$, let $W_0=V_0+(a_0E)^\perp$ and $Y$ be defined as above.
For any $x_0\in \Tbb^d$, if
\begin{equation}
\label{eq:finalassp}
\absbig{(\hhat{\mu^{*n}*\delta_{x_0}})(a_0)} \geq t \quad \text{ for some } t\in(0,\frac{1}{2}) \text{ and } n \geq C\log\frac{\norm{a_0}}{t},
\end{equation}
then there is $\gamma_0 \in F$ such that writing $y_0 = (x_0 \mod \gamma_0 W_0, \gamma_0) \in Y$, we have
\[ \dtil(y_0,Z_Q) \leq e^{-\lambda n} \quad \text{ for some }  Q \leq \left(\frac{\norm{a_0}}{t}\right)^C.\]
\end{thm}
To obtain Theorem~\ref{thm:finali} from this theorem, it suffices to lift $y_0$ and $Z_Q$ to $Y_0=\Tbb^d\times F$ and then project to $\Tbb^d$.


We proceed to the proof of Theorem~\ref{thm:final}. We fix the meaning of $a_0 \in \Z^d$, $W_0 \subset V$,  $x_0 \in \Tbb^d$ as in the statement.
By the pigeonhole principle, \eqref{eq:finalassp} implies that there is $\gamma_0 \in F$ such that 
\begin{equation}
\label{eq:finalassp2}
\absbig{\hhat{\mu^{*n}*\delta_{(x_0,\gamma_0)}}(a_0,1)} \geq \frac{t}{\# F},
\end{equation}
where the notation $\hhat{\mu^{*n}*\delta_{(x_0,\gamma_0)}}(a_0,1)$ is defined in the paragraph preceding Corollary~\ref{cr:RWwiener}.
This choice of $\gamma_0$ determines $y_0 \in Y$. We fix this $y_0$ for the rest of the proof.

\subsection{Bootstrapping concentration}
In order to prove Theorem~\ref{thm:final}, we start from the granulation estimate obtained in the previous section as Proposition~\ref{pr:RWwiener}.
The first step is then to run backwards the random walk to increase the concentration.
\begin{prop}[High concentration]
\label{high}
Assume \eqref{eq:finalassp}.
Given $\eta>0$, there exists $n_1\asymp_\eta\log\frac{\norm{a_0}}{t}$ and $\rho>0$ with $\abs{\log\rho}\asymp n_1$ such that for some $y\in Y$,
\[
\mu^{*(n-n_1)}*\delta_{y_0}(\tilde\Ball(y,\rho)) \geq \rho^\eta.
\]
\end{prop}
\begin{proof}
Using \eqref{eq:finalassp2} and Corollary~\ref{cr:RWwiener} and observing that $(a_0\gamma E)^\perp = \gamma^{-1} (a_0E)^\perp$, we obtain that for $n_0\geq \log\frac{\norm{a_0}}{t}$, there exists an $e^{-(1-2\tau)n_0}$-separated subset $X_0$ contained in a single torus in $Y$ such that
\[
\mu^{*(n-n_0)}*\delta_{y_0}\left( \Nbdtil(X_0,e^{-(1-\tau)n_0} ) \right) \geq t^{C_0}.
\]

Increasing $C_0$ if necessary, we may also assume that $\#X_0\leq e^{C_0n_0}$.
Fix some large $k\in\N$
and then $\eps>0$ such that
\(
2k\eps<1.
\)
Starting with
\[
m_0=\lfloor\frac{\tau n_0}{2d}\rfloor,\quad
r_0=e^{-(1-2\tau)n_0},
\quad\text{and}\quad
\rho_0=e^{-(1-\tau)n_0},
\]
we apply Lemma~\ref{bootstrap} below $k$ times successively.
This yields integers $m_i$, and scales $r_i>\rho_i$, defined inductively by
\[
\left\{
\begin{array}{l}
r_{i+1} = e^{-m_i(1+\eps)}r_i\\
\rho_{i+1} = e^{-m_i(1-\eps)}\rho_i\\
m_{i+1} = \lfloor m_i(1-\frac{2\eps}{d})\rfloor\\
\end{array}
\right.
\]
and at each step, an $r_i$-separated set $X_i$ such that $\#X_i\leq\#X_0$ and
\[
\mu^{*(n-n_0-m_0-\dotsb-m_i)}*\delta_{y_0}\bigl( \Nbdtil(X_i,\rho_i)\bigr) \geq \left(\frac{t^{C_0}}{2}\right)^{d^i}.
\]
Notice that by induction on $i$, one always has $e^{(d+1)m_i}\rho_i\leq r_i$, so that Lemma~\ref{bootstrap} may indeed be applied.
Moreover, choosing $n_0\asymp\log\frac{\norm{a_0}}{t}$ large enough (the involved constant will depend on $k$, $\tau$, $C_0$, etc.), we may ensure that all $m_i$ are large enough so that the error term $e^{-cm_i}$ from that lemma is always small compared to $\left(\frac{t^{C_0}}{2}\right)^{d^i}$.
Set $n_1=n_0+m_0+\dots+m_k$ and $\rho=\rho_k$.
One has $\#X_k\leq\#X_0$ and
\[
\mu^{*(n-n_1)}*\delta_{y_0}\bigl( \Nbdtil(X_k,\rho)\bigr)  \geq \left(\frac{t^{C_0}}{2}\right)^{d^k}
\]
so that for some $y\in X_k$,
\[
\mu^{*(n-n_1)}*\delta_{y_0}\left(\tilde\Ball(y,\rho)\right) \geq \frac{1}{\#X_0}\left(\frac{t^{C_0}}{2}\right)^{d^k}
\geq e^{-C_0n_0}\left(\frac{t^{C_0}}{2}\right)^{d^k}
\]
Now, since
\(
m_0+\dots+m_k \geq \frac{km_0}{3}
\)
we may choose $k$ large enough so that $m_0+\dots+m_k\geq \frac{3C_0n_0}{\eta}$, and then $n_0\asymp \log\frac{\norm{a_0}}{t}$ large enough to ensure that
\[
\rho= 
\rho_k \leq e^{-(1-\eps)(m_0+\dots+m_k)}
 \leq e^{-\frac{2C_0n_0}{\eta}}
\leq e^{-\frac{C_0n_0}{\eta}} \left(\frac{t^{C_0}}{2}\right)^{\frac{d^k}{\eta}}.
\]
The proposition follows.
\end{proof}


After using Corollary~\ref{cr:RWwiener}, we can now forget how $W_0$ is constructed from $a_0$. All what we need is that $W_0$ is a $G^\circ$-invariant rational subspace containing $V_0$. 

We now prove the lemma that was used in the above proof.
The notion of $r$-separated sets in $Y$ are with respect to the quasi-distance on $Y$. 
\begin{lem}
\label{bootstrap}
For any $\eps> 0$ there exist $c > 0$ and $m_0\in\N$ depending only on $\mu$ and $\eps$ such that the following holds for any Borel probability measure $\nu$ on $Y$ and any $m\geq m_0$.\\
Let $r > 0$ and $\rho > 0$ be such that $e^{(d+1)m}\rho < r$. 
Set
\[
r_1=e^{-m(1+\eps)}r
\quad\text{and}\quad
\rho_1=e^{-m(1-\eps)}\rho.
\]
If $X$ is an $r$-separated subset contained in a single torus in $Y$, then there is an $r_1$-separated subset $X_1 \subset Y$, contained in a single torus, with cardinality $\# X_1 \leq \# X$ and such that 
\[
\nu\bigl(\Nbdtil(X_1,\rho_1)\bigr) \geq (\mu^{*m}*\nu)\bigl(\Nbdtil(X,\rho)\bigr)^{d} - e^{-c m}.
\]
\end{lem}
\begin{proof}
In this proof, we write $X^{(\rho)}$ for $\Nbdtil(X,\rho)$. 
By Jensen's inequality and the definition of $\mu^{*m}*\nu$, (see \cite[Lemma 7.6]{BFLM} for details),
\begin{equation*}
(\mu^{*m}*\nu)(X^{(\rho)})^d
 \leq \sum_{g_1,\dots,g_d\in\Gamma} \mu^{*m}(g_1)\dots\mu^{*m}(g_d) \nu(g_1^{-1}X^{(\rho)}\cap\dots\cap g_d^{-1}X^{(\rho)}).
\end{equation*}
This implies that the set of $d$-tuples $(g_i)_{1\leq i\leq d}$ such that
\begin{equation}\label{inter}
\nu(g_1^{-1}X^{(\rho)}\cap\dots\cap g_d^{-1}X^{(\rho)})\geq (\mu^{*m}*\nu)(X^{(\rho)})^d-e^{-c m}
\end{equation}
 has $(\mu^{*m})^{\otimes d}$-measure at least $e^{-c m}$.
Using the fact that the large deviation estimates Theorem~\ref{thm:LargeD}\ref{it:LargeDn} and \ref{it:LargeDnc} are valid under the only assumption that the action is irreducible, one readily checks that \cite[Lemma~5.5]{HS2019} and its proof are also valid under this assumption.
Applying this lemma in each irreducible subrepresentation of $\R^d$, one obtains that if $c$ is chosen small enough, there must exist $g_1,\dotsc,g_d \in \Gamma$ satisfying \eqref{inter} and moreover,
\begin{equation}\label{placegi}
\forall v\in \R^d/W_0 \setminus \{0\},\quad e^{(1 - \eps)m} \leq  \max_{1\leq i\leq d} \frac{\qnorm{g_i v}}{\qnorm{v}}
\end{equation}
and (using the large deviation estimates again and the fact that $-\lambda_{\dim W}(\mu,W) \leq (\dim W - 1) \lambda_1(\mu,W)$ for any $G$-invariant $W \subset V$)
\begin{equation}\label{placeginorm}
\forall i \in \{1,\dotsc,d\},\quad \qnorm{g_i} \leq e^{(1+\eps)m} \quad\text{and}\quad \qnorm{g_i^{-1}} \leq e^{(d-1+\eps) m}.
\end{equation}
We fix such elements $g_1,\dots,g_d$ for the rest of the proof.

Since $X$ is contained in a single torus, all the $g_i$'s are contained in the same class in $G/G^\circ$.

We claim that the set $g_1^{-1}X^{(\rho)}\cap\dots\cap g_d^{-1}X^{(\rho)}$ is included in a union of at most $\#X$ balls of radius $\rho_1=e^{-(1-\eps)m}\rho$.
Indeed, from \eqref{placeginorm} and the fact that $e^{(d+1)m}\rho<r$, we find, for a given $x\in X$ and $i\geq 1$, that the set $g_1^{-1}\tilde\Ball(x,\rho)$ meets at most one component $g_i^{-1}\tilde\Ball(y,\rho)$, $y\in X$.
Therefore, there are at most $\#X$ non-empty intersections $g_1^{-1}\tilde\Ball(x_1,\rho)\cap\dotsb \cap g_d^{-1}\tilde\Ball(x_d,\rho)$, for $x_1,\dotsc,x_d\in X$.

If $x,y$ lie inside such an intersection, then, for each $i$, $\dtil(g_ix, g_iy) \leq\rho$. Then \eqref{placegi} and \eqref{placeginorm} implies that $\dtil(x,y)\leq e^{-(1-\eps)m}\rho=\rho_1$.
Thus, each intersection $g_1^{-1}\tilde\Ball(x_1,\rho)\cap \dots \cap g_d^{-1}\tilde\Ball(x_d,\rho)$ is included in a ball of radius $\rho_1$.

Finally, using \eqref{placeginorm}, we see that these $\rho_1$-balls are separated by at least $r_1=e^{-(1+\eps)m} r$.
Moreover they are contained in the same torus in $Y$ because the $g_i$'s are in the same $G^\circ$ coset.
This finishes the proof of the proposition.
\end{proof}

\subsection{A diophantine property}
From the high concentration property obtained in the previous paragraph, we want to infer that $\mu^{*(n-n_1)}*\delta_{y_0}$ is concentrated near a proper $\Gamma$-invariant subset.
The argument relies on a diophantine property of the random walk, coming from the fact that $\mu$ is supported on $\GL_d(\Z)$.


\begin{prop}[Concentration near a closed invariant subset]
\label{pr:inv}
Given $\beta>0$, there exists $C>0$ such that the following holds.

Assume \eqref{eq:finalassp}.
Then there exist $n_1\in\N^*$ such that $\frac{1}{C}\log\frac{\norm{a_0}}{t}\leq n_1\leq C\log\frac{\norm{a_0}}{t}$ and $\rho\in[e^{-Cn_1},e^{-\frac{n_1}{C}}]$ and $Q\leq \rho^{-\beta}$ such that
\[
\mu^{*(n-n_1)}*\delta_{y_0}\bigl(\Nbdtil(Z_{Q},\rho)\bigr) \geq \rho^\beta.
\]
\end{prop}

This proposition is an immediate consequence of Proposition~\ref{high} and of a diophantine property of the random walk, given by the following lemma.

\begin{lem}[Diophantine property]
For every $\beta>0$, there exist constants $C$ and $\eta>0$ depending on $\mu$ and $\beta$ such that
for any $y_0,y\in Y$, if for $n\geq C\abs{\log\rho}$, one has
\[
(\mu^{*n} * \delta_{y_0})(\tilde\Ball(y,\rho)) \geq \rho^\eta
\]
then $\dtil(y, Z_{Q}) \leq \rho^{1-\beta}$ for some $Q \leq \rho^{-\beta}$.
\end{lem}
\begin{proof}
Consider the action of $G$ on $V_F =\bigsqcup_{\gamma\in F} V/\gamma W_0\times \{\gamma\}$.
Since $V/W_0$ contains no $G^\circ$-invariant vector, for every non-zero $(v_1,v_2)$ in $V_F\times V_F$, the set $\set{g}{g v_1 = v_2}$ is a linear subvariety in $G$ of dimension less than $\dim G$. Using the spectral gap property modulo prime numbers~\cite{SGV} (or applying the first step of the proof of Proposition~\ref{ncirr}), we obtain that for $m$ large enough, for some $c_0$ independent of $(v_1,v_2)$,
\begin{equation}\label{inj}
\mu^{\otimes m}(\set{(g_1,\dots,g_m)}{g_m\dotsm g_1 v_1 = v_2}) \leq e^{-c_0m}.
\end{equation}
Fix $m$ such that
\[
e^{-\frac{c_0m}{2}}\geq \rho^\eta > e^{-c_0m}.
\]
If $C$ is large enough, the condition $n\geq C\abs{\log\rho}$ ensures that $n\geq m$.
From the assumed inequality, it follows that there exists some $y_1\in Y$ such that 
\[
(\mu^{*m} * \delta_{y_1})(\tilde\Ball(y,\rho)) \geq \rho^\eta,
\]
which implies that the set
\[
A_m = \set{(g_1,\dots,g_m)\in(\Supp\mu)^{m}}
{\dtil(g_m\dots g_1y_1,y)\leq \rho}
\]
satisfies
\[
\mu^{\otimes m}(A_m)\geq \rho^\eta > e^{-c_0m}.
\]
Using the finite exponential moment of $\mu$ and reducing the set $A_m$, we can assume further that for all $(g_i) \in A_m$, $\norm{g_m \dotsm g_1} \leq e^{C_0 m}$ for some $C_0 = C_0(\mu)$.
Since all matrices have integer coefficients, the set
\[
A'_m = \{g_m\dots g_1\ ;\ (g_1,\dots,g_m)\in A_m\}
\]
is then finite.

Write $y_1=(x_1,\gamma_1)$ and $y=(x,\gamma)$.
Recalling \eqref{inj} above, one infers that the linear map
\[
\begin{array}{lccc}
\theta\colon & V/\gamma_1W_0\times V/\gamma W_0 & \to & (V/\gamma W_0)^{A_m'}\\
& (v_1,v_2) & \mapsto & (g v_1 - v_2)_{g\in A_m'}
\end{array}
\]
is injective.
Moreover, in the canonical bases, its matrix has coefficients in $\Z$ bounded by $e^{C_0m}$, so its inverse has coefficients in $\frac{1}{Q}\Z$ for some $Q\leq e^{C_1m}$, and bounded by $e^{C_1m}$.
Therefore, any solution $(v_1,v_2)$ in $V/\gamma_1W_0\times V/\gamma W_0$ to
\[
\forall (g_1,\dots,g_m)\in A_m,\quad
g_m\dots g_1 v_1 - v_2 \in \tilde\Ball(0,\rho)+\Z^d \mymod \gamma W_0\\
\]
can be written, for some $w_1,w_2\in\Z^d$ and $u_1,u_2$ in $\tilde\Ball(0,e^{C_1 m}\rho)$,
\[
v_1 = \frac{1}{Q} w_1 + u_1 \mymod \gamma_1 W_0 \quad \text{and} \quad v_2 = \frac{1}{Q} w_2 + u_2 \mymod \gamma W_0.
\]
This applies in particular to representatives of $(x_1,x)$ in $V/\gamma_1W_0\times V/\gamma W_0$.
It follows that
\[
\dtil(y,Z_{e^{C_1m}}) \leq e^{C_1 m}\rho.
\]
If $\eta>0$ is chosen so small that $\frac{2C_1}{c_0}\eta<\beta$, one has
\[
e^{C_1m} = e^{\frac{2C_1}{c_0}\frac{c_0m}{2}} \leq \rho^{-\frac{2C_1}{c_0}\eta} \leq \rho^{-\beta}
\]
so the lemma is proved.
\end{proof}

\subsection{Unstability of closed invariant subsets}
To conclude the proof of Theorem~\ref{thm:final}, we use a variant of the argument given in \cite[\S3]{HLL2020}.
It is based on Foster's exponential recurrence criterion, applied to a well-chosen function associated to a closed invariant subset.
This technique has been used extensively in homogeneous dynamics since the work of Eskin and Margulis \cite{em}, in particular by Benoist and Quint for their study of stationary measures \cite{bq1,bq2,bq3,bqfv}.

\begin{lem}[Margulis inequality]
\label{lm:foster}
For every $\lambda\in(0,1)$, there exist constants $C,\alpha>0$ depending only on $\mu$ such that the following holds.
For $Q\geq 2$, define a function $\phi_Q\colon Y\to \R\cup\{+\infty\}$ by
\[
\phi_Q(y) = \left\{\begin{array}{ll}
\dtil(y,Z_{Q})^{-\alpha} & \mbox{if}\ \dtil(y,Z_{Q})>0\\
+\infty & \mbox{otherwise}.
\end{array}\right.
\] 
For all $y \in Y$ and all integers $n \geq 1$, 
\[
\int \phi_Q(gy) \dd \mu^{*n} (g) \leq e^{-\lambda\alpha n} \phi_Q(y) + Q^C.
\]
\end{lem}

The proof of such inequalities is an application of Furstenberg's law of large numbers \cite[Theorem~4.28]{BenoistQuint}, using also the exponential moment assumption on $\mu$.
Since it is rather standard, we leave it to the reader, and turn to the proof of Theorem~\ref{thm:final}.

\begin{proof}[Proof of Theorem~\ref{thm:final}]
Let $C,\alpha>0$ be the parameters given by Lemma~\ref{lm:foster} applied with $\lambda'=\frac{1+\lambda}{2}$ instead of $\lambda$.
Then set $\beta=\frac{\alpha}{C+2}$.

Proposition~\ref{pr:inv} shows that for some $n_1\asymp_\beta\log\frac{\norm{a_0}}{t}$ and some $\rho\in[e^{-C_1n_1},e^{-c_1n_1}]$, there exist $Q\leq\rho^{-\beta}$ such that
\[
\mu^{*(n-n_1)}*\delta_{y_0}\bigl(\Nbdtil(Z_{Q},\rho)\bigr) \geq \rho^\beta.
\]
Applying Lemma~\ref{lm:foster} yields
\begin{align*}
\rho^{-\alpha+\beta} & \leq \int \phi_Q(gy_0) \dd \mu^{*(n-n_1)} (g)\\
 & \leq e^{-\lambda'\alpha (n-n_1)} \phi_Q(y_0) + Q^C.
\end{align*}
Note that $Q^C\leq \rho^{-C\beta}\leq \frac{1}{2}\rho^{-\alpha+\beta}$ and therefore
\[
\phi_Q(y_0) = \dtil(y_0,Z_{Q})^{-\alpha}
\gg e^{\lambda'\alpha (n-n_1)}\rho^{-\alpha+\beta} 
\gg e^{\lambda'\alpha n} e^{-C_1\alpha(1-\frac{1}{C+2})n_1}.
\]
Since $n_1\asymp\log\frac{\norm{a_0}}{t}$ and $\lambda'=\lambda+\frac{1-\lambda}{2}$, taking $n\gg \frac{1}{1-\lambda}\log\frac{\norm{a_0}}{t}$ yields
\[
\dtil(y_0,Z_{Q}) \leq e^{-\lambda n}
\]
and the theorem is proved.
\end{proof}

\subsection*{Acknowledgements}
It is a pleasure to thank Yves Benoist and Elon Lindenstrauss for several useful discussions, in particular on the existence of satellite measures in the presence of compact factors.
The authors are also grateful to the anonymous referee for numerous corrections and helpful comments.
While this research was conducted, W.H. was supported by ERC 2020 grant HomDyn (grant no.~833423), KIAS Individual Grant (no.~MG080401) and the National Natural Science Foundation of China (No. 12288201).

\bibliographystyle{abbrv} 
\bibliography{bib_redu}

\end{document}